\documentclass[12pt]{amsart}
\usepackage{tikzit}
\usepackage[all]{xy}
\pdfoutput=1

\tikzstyle{node small}=[fill=none, draw=none, shape=circle, inner sep=0.2]
\tikzstyle{bullet}=[fill=black, draw=black, shape=circle, scale=0.3]
\tikzstyle{bullet small}=[fill=black, draw=black, shape=circle, scale=0.14]
\tikzstyle{spin}=[->-]
\tikzstyle{bullet mid}=[fill=black, draw=black, shape=circle, scale=0.2]

\tikzstyle{red}=[draw=red, ->]
\tikzstyle{red 1}=[draw=red, <-]
\tikzstyle{black arrow}=[->]
\tikzstyle{dashed}=[-, draw=black, densely dotted]
\tikzstyle{Red-Thick}=[-, draw=red, thick]
\tikzstyle{spin}=[-, postaction={{decorate}}]
\tikzstyle{Blue-Thick}=[-, fill=none, draw=blue, thick]
\tikzstyle{extend-edge}=[-, dashed]
\tikzstyle{blue-double}=[-, draw=blue, thick, double]

\usepackage{mathbbol}

\usepackage{ifthen}
\usepackage[utf8]{inputenc}
\usepackage{setspace}
\usepackage[english]{babel}
\usepackage{enumerate}
\usepackage{thmtools}
\usepackage[shortlabels]{enumitem}
\usepackage[T1]{fontenc}
\usepackage{tikz}
\usetikzlibrary{shapes.geometric,intersections,decorations.markings,snakes}
\usetikzlibrary{calc,intersections,through,backgrounds}
\usetikzlibrary{patterns}

\renewcommand{\l}[2]{\lambda_{#1#2}}

\DeclareMathOperator{\osp}{Osp}

\DeclareMathOperator{\ber}{Ber}
\DeclareMathOperator{\st}{st}

\DeclareMathOperator{\id}{id}

\newcommand{\ospmatrix}[9]{
\left(\begin{array}{cc|c}
#1&#2& #3 \\
#4 &#5& #6 
\\
\hline
#7 &#8 &#9
\end{array}
\right)
}

\usepackage{stackengine}
\usepackage{scalerel}

\include{theorems}

\newlength\friezelen 
\usepackage{array}

\usepackage{mathtools}
\usepackage{verbatim}
\usepackage{comment}
\usetikzlibrary{arrows}
\tikzset {->-/.style={decoration={markings, mark=at position .5 with {\arrow{latex}}}, postaction={decorate}}}
\tikzset {-->-/.style={decoration={markings, mark=at position .5 with {\arrow[scale=2]{latex}}}, postaction={decorate}}}
\newcommand{\midarrow}{\tikz \draw[-triangle 90] (0,0) -- +(.1,0);}

\usepackage{comment}
\usepackage{graphicx}
\usepackage{color}
\usepackage{etoolbox}
\usepackage[margin=1in]{geometry}
\usepackage{amsmath,amsthm,amssymb,graphicx,tikz,tikz-cd}
\usepackage{multicol}
\usepackage{hyperref}
\hypersetup{
    colorlinks = true,
    citecolor = orange,
}
\usepackage[noabbrev]{cleveref}
\usepackage{subcaption}

\makeatletter
\patchcmd{\@settitle}{\uppercasenonmath\@title}{}{}{}
\patchcmd{\@setauthors}{\MakeUppercase}{}{}{}
\patchcmd{\section}{\scshape}{}{}{}
\makeatother
\makeatletter
\patchcmd{\@maketitle}
  {\ifx\@empty\@dedicatory}
  {\ifx\@empty\@date \else {\vskip2ex 
  \centering\footnotesize\@date\par\vskip1ex}\fi
   \ifx\@empty\@dedicatory}
  {}{}
\patchcmd{\@adminfootnotes}
  {\ifx\@empty\@date\else \@footnotetext{\@setdate}\fi}
  {}{}{}
\makeatother

    \DeclareFontFamily{U}{wncy}{}
    \DeclareFontShape{U}{wncy}{m}{n}{<->wncyr10}{}
    \DeclareSymbolFont{mcy}{U}{wncy}{m}{n}
    \DeclareMathSymbol{\Sha}{\mathord}{mcy}{"58}

\renewcommand{\l}[2]{\lambda_{#1,#2}}

\newcommand{\td}[3]{\bigtriangledown^{#1}_{#2,#3}}

\newcommand{\Td}[0]{\bigtriangledown}

\theoremstyle{plain}
\newtheorem{theorem}{Theorem}[section]

\newtheorem{lemma}[theorem]{Lemma}
\newtheorem{prop}[theorem]{Proposition}
\newtheorem{conj}[theorem]{Conjecture}
\newtheorem{question}[theorem]{Question}
\newtheorem{corollary}[theorem]{Corollary}

\theoremstyle{definition}
\newtheorem{remark}[theorem]{Remark}
\newtheorem{example}[theorem]{Example}
\newtheorem{definition}[theorem]{Definition}

\title[]{Super Markov Numbers and Signed Double Dimer Covers} 

\author{Gregg Musiker}

\thanks{School of Mathematics, University of Minnesota, Minneapolis, MN 55455, USA}
\thanks{\emph{Email}:
\href{mailto:musiker@umn.edu}{musiker@umn.edu}}
\date{\today}

\begin{document}

\begin{abstract}
We provide a superalgebraic analogue of Markov numbers, which are defined as the Grassmann integer solutions to the equation $x^2 + y^2 + z^2 + (xy + yz + xz)\epsilon = 3(1+\epsilon)xyz$, as well as applications to the Decorated Super Teichm\"uller spaces associated to the once-punctured torus and certain annuli.  We conclude with further directions for study.
\end{abstract}

\maketitle
\setcounter{tocdepth}{1}
\tableofcontents
\setlength{\parindent}{0em}
\setlength{\parskip}{0.618em}

\tableofcontents

\section{Introduction}

Previous work by the author, completed jointly with N. Ovenhouse and S. Zhang \cite{moz21,moz22,moz22b}, studied non-commutative deformations of formulas arising from hyperbolic geometry and low dimensional topology, focusing on the cases of triangulated polygons.  More precisely, we presented a combinatorial formula for super $\lambda$-lengths of arcs in a decorated super-Teichm\"uller space, as defined by R. Penner and A. Zeitlin \cite{pz_19}, for the case associated to a marked disk in terms of double dimer covers of snake graphs on the one hand, and in terms of associated holonomy matrices in the supergroup $\osp(1|2)$ on the other.  (In \cite{ip2018n}, I. Ip, R. Penner, and A. Zeitlin generalize the constructions of \cite{pz_19} to higher rank and provide further details on spin structures.)  This research also led to the construction of super Fibonacci numbers associated to arcs in certain annuli, which we now extend in two different directions.  

On the one hand, the sequence of every-other Fibonacci number is an instance of a Markov number, one of the three coordinates $(x,y,z)$ comprising an integer solution to the Markov equation $x^2 + y^2 + z^2 = 3xyz$.  We consider a $\mathbb{Z_2}$-graded analogue of this equation, i.e. a super version, and its solutions which includes the super Fibonacci numbers.  Such solutions, in analogy with the classical Markov numbers, can also be interpreted as super $\lambda$-lengths for arcs of rational slope on a once-punctured torus, and we use holonomy matrices to extend our previous constructions for combinatorial formulas to obtain results for super Markov numbers.  We note that to include more general cases, we no longer have exclusively manifestly positive expansion formulas but instead express our formulas as signed sums of double dimer covers of snake graphs.  

In a second direction, we extend our results on super Fibonacci numbers, which used annuli with one marked point on each boundary, to introduce a family of super-integrable systems obtained from the decorated super-Teichm\"uller space associated to a more general marked annuli which analogously admits combinatorial formulas in terms of signed double dimer covers.

Our original work providing combinatorial formulas \cite{moz21,moz22} was constrained in the case of polygons or in cases where there was a consistent way to assign a canonically defined \emph{default orientation} to a triangulation, a \emph{positive ordering} of the associated odd variables ($\mu$-invariants), and a \emph{flip sequence} so that we continue to have the default orientation and positive ordering of odd variables for intermediate triangulations as one applies flips, thus allowing inductive arguments.  However, using the holonomy matrix formulas from \cite{moz22b} allows us to consider super $\lambda$-lengths on more general triangulated surfaces even when we are not guaranteed a choice of spin structure or ordering of odd variables that would allow us to obtain positive coefficients on each term in our expansions.  It is with such matrices that we began our investigation of super $\lambda$-lengths associated to a traingulation of a once-punctured torus.  In the ordinary case, the resulting $\lambda$-lengths correspond to Markov numbers, as we detail in Section \ref{sec:Markov}, and thus we refer to the resulting super $\lambda$-lengths resulting from this analysis as Super Markov Numbers.  After describing novel combinatorial formulas for Super Markov Numbers (Theorem \ref{thm:supermarkov}), and proving our main result in Section \ref{sec:main_proof}, we present several open questions for future research in Section \ref{sec:open}.

In Section \ref{sec:Annuli}, we provide another application to the approach of holonomy matrices by considering super $\lambda$-lengths of bridging arcs on marked annuli.  This case also directly generalizes the author's previous work on Super Fibonacci numbers, but unlike that case (which corresponded to annuli with one marked point on each boundary), when there are more marked points involved, we need to work with more odd variables, and mutations induce recurrences involving more complicated square roots.  After obtaining initial computational results using holonomy matrices, we prove some associated combinatorial formulas for such super $\lambda$-lengths using double dimer covers, and leaving other cases as conjectural.  We similarly conclude Seciton \ref{sec:Annuli} with further directions.

We begin in Section \ref{sec:prelim} by recapping the set-up from the author's previous work on combinatorial formulas for decorated super Teichm\"uller space from \cite{moz21,moz22,moz22b} and the general approach provided by holonomy matrices.  This will then allow us to give the more specialized combinatorial formulas in context as we focus on the cases of the once-punctured torus and marked annuli in Sections \ref{sec:Markov} and \ref{sec:Annuli}, respectively.

{\bf Acknowledgements:} The author thanks Esther Banaian, Sophie Morier-Genoud, Nick Ovenhouse, and Sylvester Zhang for engaging exchanges during the writing of this paper, and Jim Propp, Christophe Reutenauer, and Ralf Schiffler for formative discussions on Markov numbers. 

\section{Preliminaries}
\label{sec:prelim}

We begin by considering marked surfaces (possibly with boundary) where there is a well-defined ``\emph{default orientation}'', as defined in \cite{moz21}, which is pictured in \Cref{fig:default_spin}.   This includes the case of marked disks with marked points exclusively on its boundary.   For such a surface, the maximal groupings of consecutive triangles which share a common vertex (indicated by different colors) are called ``\emph{fan segments}'', and the common vertex they share is called the ``\emph{fan center}'' (vertices labelled $c_i$). The default orientation is defined so that the edges connecting
fan centers are oriented $c_1 \to c_2 \to c_3 \to \cdots$, and the remaining edges are oriented \emph{away} from the fan centers.

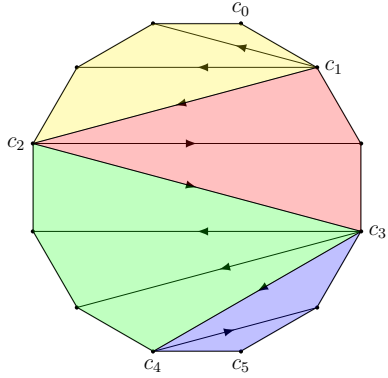
\begin{figure}[h!]
\centering
\scalebox{0.7}{
\begin{tikzpicture}[decoration={
    markings,
    mark=at position 0.5 with {\arrow[scale=1.5]{latex}}},
    scale=0.9] 

\tikzstyle{every path}=[draw] 
		\path
    node[
      regular polygon,
      regular polygon sides=12,
      draw,
      inner sep=2.2cm,
    ] (hexagon) {}
    %
    (hexagon.corner 1) node[above] {$c_0$}
    (hexagon.corner 2) node[above] {}
    (hexagon.corner 3) node[left] {}
    (hexagon.corner 4) node[left] {$c_2$}
    (hexagon.corner 5) node[below] {}
    (hexagon.corner 6) node[right] {}
    (hexagon.corner 7) node[below] {$c_4$}
    (hexagon.corner 8) node[below] {$c_5$}
    (hexagon.corner 9) node[left] {}
    (hexagon.corner 10) node[right] {$c_3$}
    (hexagon.corner 11) node[below] {}
    (hexagon.corner 12) node[right] {$c_1$}
;

\foreach \x in {1,2,...,12}{
\draw (hexagon.corner \x) node [fill,circle,scale=0.2] {};}

\draw[postaction={decorate}] (hexagon.corner 12)--(hexagon.corner 4);
\draw[postaction={decorate}] (hexagon.corner 12)--(hexagon.corner 2);
\draw[postaction={decorate}] (hexagon.corner 12)--(hexagon.corner 3);

\draw[postaction={decorate}](hexagon.corner 4)--(hexagon.corner 10);
\draw[postaction={decorate}](hexagon.corner 4)--(hexagon.corner 11);

\draw[postaction={decorate}](hexagon.corner 10)--(hexagon.corner 7);
\draw[postaction={decorate}](hexagon.corner 10)--(hexagon.corner 5);
\draw[postaction={decorate}](hexagon.corner 10)--(hexagon.corner 6);

\draw[postaction={decorate}](hexagon.corner 7)--(hexagon.corner 9);

\draw[fill=yellow, nearly transparent] (hexagon.corner 12)--(hexagon.corner 4)--(hexagon.corner 3)--(hexagon.corner 2)--(hexagon.corner 1)--cycle;
\draw[fill=red, nearly transparent] (hexagon.corner 12)--(hexagon.corner 4)--(hexagon.corner 10)--(hexagon.corner 11)--cycle;

\draw[fill=green, nearly transparent] (hexagon.corner 10)-- (hexagon.corner 4)--(hexagon.corner 5)--(hexagon.corner 6)--(hexagon.corner 7)--cycle;
\draw[fill=blue, nearly transparent] (hexagon.corner 10)--(hexagon.corner 7)--(hexagon.corner 8)--(hexagon.corner 9)--cycle;

\coordinate (m 1) at  ($0.38*(hexagon.corner 2)+0.38*(hexagon.corner 1)+0.24*(hexagon.corner 12)$);
\coordinate (m 2) at  ($0.28*(hexagon.corner 2)+0.28*(hexagon.corner 3)+0.44*(hexagon.corner 12)$);
\coordinate (m 3) at  ($0.24*(hexagon.corner 4)+0.24*(hexagon.corner 3)+0.52*(hexagon.corner 12)$);
\coordinate (m 4) at  ($0.52*(hexagon.corner 4)+0.24*(hexagon.corner 11)+0.24*(hexagon.corner 12)$);
\coordinate (m 5) at  ($0.56*(hexagon.corner 4)+0.22*(hexagon.corner 11)+0.22*(hexagon.corner 10)$);
\coordinate (m 6) at  ($0.27*(hexagon.corner 4)+0.27*(hexagon.corner 5)+0.46*(hexagon.corner 10)$);
\coordinate (m 7) at  ($0.23*(hexagon.corner 6)+0.23*(hexagon.corner 5)+0.44*(hexagon.corner 10)$);
\coordinate (m 8) at  ($0.26*(hexagon.corner 6)+0.26*(hexagon.corner 7)+0.48*(hexagon.corner 10)$);
\coordinate (m 9) at  ($0.25*(hexagon.corner 9)+0.5*(hexagon.corner 7)+0.25*(hexagon.corner 10)$);
\coordinate (m 10) at  ($0.36*(hexagon.corner 9)+0.24*(hexagon.corner 7)+0.4*(hexagon.corner 8)$);
\end{tikzpicture}
}
\caption{The default orientation of a generic triangulation where each fan segment is colored differently.}
\label{fig:default_spin}
\end{figure}

The \emph{decorated super-Teichm\"uller space} of $P$ is a superalgebra with the following generators: 
for each edge in $T$ with endpoints $i$ and $j$, an even generator $\lambda_{ij}$
(called a ``$\lambda$-length'')\footnote{The algebra is actually generated by the square rooots of the $\lambda$-lengths.}, 
and for each triangle in $T$ with vertices $i,j,k$, an odd generator $\boxed{ijk}$ (called a ``$\mu$-invariant'').

When two triangulations are related by a flip, as in \Cref{fig:super_ptolemy}, we define new elements of the algebra by the following ``\emph{super Ptolemy relations}'':
\begin{align}
    ef      &= ac+bd+\sqrt{abcd} \, \sigma\theta \label{eqn:super_ptolemy_lambda} \\
    \sigma' &= \frac{\sigma\sqrt{bd}-\theta\sqrt{ac}}{\sqrt{ac+bd}} = \frac{\sigma\sqrt{bd}-\theta\sqrt{ac}}{\sqrt{ef}} \label{eqn:super_ptolemy_mu_right} \\
    \theta' &= \frac{\theta\sqrt{bd}+\sigma \sqrt{ac}}{\sqrt{ac+bd}} = \frac{\theta\sqrt{bd}+\sigma \sqrt{ac}}{\sqrt{ef}} \label{eqn:super_ptolemy_mu_left}
\end{align}

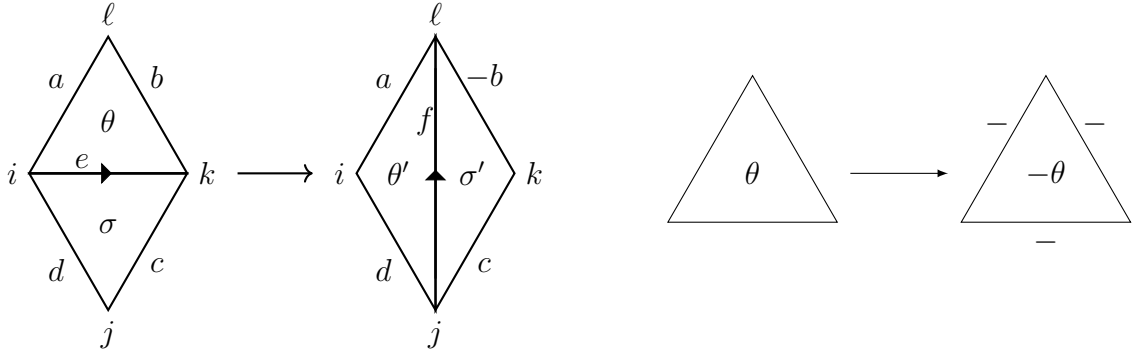
\begin{figure}[h!]
\centering
\begin{tikzpicture}[scale=0.7, baseline, thick]

    \draw (0,0)--(3,0)--(60:3)--cycle;
    \draw (0,0)--(3,0)--(-60:3)--cycle;

    \draw (0,0)--node {\midarrow} (3,0);

    \draw node[above]      at (70:1.5){$a$};
    \draw node[above]      at (30:2.8){$b$};
    \draw node[below]      at (-30:2.8){$c$};
    \draw node[below=-0.1] at (-70:1.5){$d$};
    \draw node[above] at (1,-0.12){$e$};

    \draw node[left] at (0,0) {$i$};
    \draw node[above] at (60:3) {$\ell$};
    \draw node[right] at (3,0) {$k$};
    \draw node[below] at (-60:3) {$j$};

    \draw node at (1.5,1){$\theta$};
    \draw node at (1.5,-1){$\sigma$};
\end{tikzpicture}
\begin{tikzpicture}[baseline]
    \draw[->, thick](0,0)--(1,0);
    \node[above]  at (0.5,0) {};
\end{tikzpicture}
\begin{tikzpicture}[scale=0.7, baseline, thick,every node/.style={sloped,allow upside down}]
    \draw (0,0)--(60:3)--(-60:3)--cycle;
    \draw (3,0)--(60:3)--(-60:3)--cycle;

    \draw node[above]      at (70:1.5)  {$a$};
    \draw node[above]      at (30:2.8)  {$-b$};
    \draw node[below]      at (-30:2.8) {$c$};
    \draw node[below=-0.1] at (-70:1.5) {$d$};
    \draw node[left]       at (1.7,1)   {$f$};

    \draw (1.5,-2) --node {\midarrow} (1.5,2);

    \draw node[left] at (0,0) {$i$};
    \draw node[above] at (60:3) {$\ell$};
    \draw node[right] at (3,0) {$k$};
    \draw node[below] at (-60:3) {$j$};

    \draw node at (0.8,0){$\theta'$};
    \draw node at (2.2,0){$\sigma'$};
\end{tikzpicture}  \hspace{3em} \begin {tikzpicture}[scale=1.3, baseline]
        \draw (0,1) -- (-0.866,-0.5) -- (0.866,-0.5) -- cycle;
        \draw (0,0) node {$\theta$};

        \draw[-latex] (1,0) -- (2,0);

        \begin {scope}[shift={(3,0)}]
            \draw (0,1) -- (-0.866,-0.5) -- (0.866,-0.5) -- cycle;
            \draw (0,-0.7) node{$-$};
            \draw (-0.5,0.5) node{$-$};
            \draw (0.5,0.5) node{$-$};
            \draw (0,0) node {$-\theta$};
        \end {scope}
    \end {tikzpicture}

\caption{(Left): Super Ptolemy Transformation.  (Right): Equivalence Relation.}
\label{fig:super_ptolemy}
\end{figure}
Note that in \Cref{eqn:super_ptolemy_lambda}, the order of multiplying the two odd variables $\sigma$ and $\theta$ is determined by the orientation of the 
edge being flipped (see the arrow in \Cref{fig:super_ptolemy}).

In Figure \ref{fig:super_ptolemy} (Left), the orientations of the four boundary edges 
are omitted, but the super Ptolemy transformation does change the orientation of the edge labeled $b$ (the edges $a,c,d$ keep their same orientation).

For larger triangulations, we specify a standard order to multiply together odd variables based on the orientation of interior edges.  We follow \cite[Remark 5.7]{moz21} and define the \emph{positive ordering} as follows.

\begin{definition}
\label{def:pos_ordering} 
If we label the $\mu$-invariants as $\theta_1, \theta_2, \dots, \theta_n$ in order according to the crossings by arc $\gamma$, then starting from the end we apply the following procedure.  Suppose that $\theta_{k+1}, \theta_{k+2}, \dots, \theta_n$ are already ordered.  Then if the edge separating $\theta_k$ and $\theta_{k+1}$ is oriented so that $\theta_k$ is to the right, then we declare 
$\theta_k > \theta_i$ for all $k+1 \leq i \leq n$.  On the other hand, if $\theta_k$ is to the left of $\theta_{k+1}$, then we declare $\theta_k < \theta_i$ for all $k+1 \leq i \leq n$.  Applying this first to the case $k=n-1$ and following by descent until $k=1$ yields an ordering on all $n$ $\mu$-invariants.
\end{definition}

In previous work \cite[Theorem 6.2(a)]{moz22},  we provided a combinatorial formula for expansions of super $\lambda$-lengths of arcs cutting through triangulations of polygons using \emph{double dimer covers} of \emph{snake graphs}.  We defer the definition of snake graphs to Section \ref{sec:combo} where we discuss them in detail with the specific application of Super Markov numbers in mind.  However, we note that given a graph of vertices and edges (each edge is a pair of distinct vertices), a \emph{double dimer cover} is a distinguished collection of edges (possibly with multiplicities) so that every vertex is covered exactly once.

Every double dimer cover is a disjoint union of doubled edges and cycles (a configurations of edges used singly that is homotopic to a circle).  Thus following \cite[Definition 4.4]{moz22}, it makes sense to define the weight of a double dimer cover $M$ as $$\prod_{e \in M} wt(e) \prod_{\mathrm{Cycle~}C \mathrm{~of~}M} wt(C)$$ where 
$wt(e)$ is the square-root of the label of edge $e$, and $wt(C) = \mu_i \mu_j$ where $\mu_i$ and $\mu_j$ are two $\mu$-invariants whose definition is deferred until Section \ref{sec:combo} after we have defined our special case of snake graphs. 
  
Following \cite{fg_06, mw13}, in \cite{moz22b} we defined \emph{elementary steps} on marked surfaces, and associated certain matrices from \Cref{def:holonomy_matrices}
 to them.  We begin with the special case where our marked surfaces is simply a polygon, i.e. a genus $0$ surface with a single boundary.

Fix $(S,M)$ a marked disk where $M$ is the set of marked points on the boundary and $T$ a triangulation of it.

\begin {definition} \label{def:GammaT}
    From the triangulation $T$ of the marked disk $(S,M)$, we will define a planar graph $\Gamma_T$. Inside each triangle
    of $T$, there is a hexagonal face of $\Gamma_T$ with three sides parallel to the sides of the triangle. 
    When two triangles share a side, the two vertices of $\Gamma_T$ on
    opposite sides of this edge are connected (see \Cref{fig:Gamma_T}).
\end {definition}

\begin {figure}[h]
\centering
\begin {tikzpicture}
    \draw[dashed] (-3,0) -- (0,-3) -- (3,0) -- (0,3) -- cycle;
    \draw[dashed] (0,-3) -- (0,3);

    \draw[fill=black] (-0.3, -1.5) circle (0.06);
    \draw[fill=black] (-0.3, 1.5) circle (0.06);
    \draw[fill=black] (0.3,  -1.5) circle (0.06);
    \draw[fill=black] (0.3,  1.5) circle (0.06);

    \draw[fill=black] (-0.7,-1.9) circle (0.06);
    \draw[fill=black] (0.7,-1.9) circle (0.06);
    \draw[fill=black] (-0.7,1.9) circle (0.06);
    \draw[fill=black] (0.7,1.9) circle (0.06);

    \draw[fill=black] (-2.2,-0.4) circle (0.06);
    \draw[fill=black] (-2.2,0.4) circle (0.06);
    \draw[fill=black] (2.2,-0.4) circle (0.06);
    \draw[fill=black] (2.2,0.4) circle (0.06);

    \draw (0.3,-1.5) -- (0.3,1.5) -- (-0.3,1.5) -- (-0.3,-1.5) -- cycle;
    \draw (-0.3,-1.5) -- (-0.7,-1.9) -- (-2.2,-0.4) -- (-2.2,0.4) -- (-0.7,1.9) -- (-0.3,1.5);
    \draw (0.3,-1.5) -- (0.7,-1.9) -- (2.2,-0.4) -- (2.2,0.4) -- (0.7,1.9) -- (0.3,1.5);
\end {tikzpicture}
\caption {The graph $\Gamma_T$, with $T$ in dashed lines}
\label {fig:Gamma_T}
\end {figure}
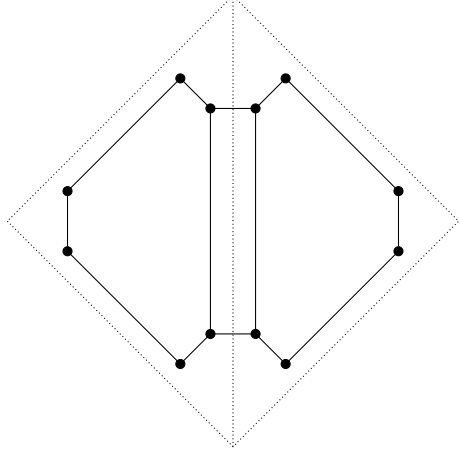

\begin {remark} \label{rmk:Gamma_T_edges_faces}
    The graph $\Gamma_T$ has 3 kinds of edges and 2 kinds of faces. The three types of edges of $\Gamma_T$ are:
    \begin {itemize}
        \item The edges parallel to arcs $\tau$ of the triangulation $T$. (If $\tau \in T$ is a boundary edge,
              then there is only one such edge of $\Gamma_T$, and if $\tau$ is an internal diagonal, then there are two such parallel edges in $\Gamma_T$.)
        \item The edges within a triangle that are \emph{not} parallel to arcs $\tau$ of $T$. (These naturally correspond to the angles
              of the triangles.)
        \item The edges which cross the arcs $\tau$ of $T$.
    \end {itemize}
  
    The two types of faces are as follows:
    \begin {itemize}
        \item Within each triangle of $T$, there is a hexagonal face of $\Gamma_T$.
        \item Surrounding each internal diagonal of $T$, there is a quadrilateral face of $\Gamma_T$.
    \end {itemize}
\end {remark}

\begin {definition}
    For a graph embedded on a surface, a \emph{graph connection} is an assignment of a matrix to each oriented edge,
    such that opposite orientations of the same edge are assigned inverse matrices. For a path in the graph,
    the \emph{holonomy} is the corresponding composition/product of matrices along the path. If the path is a loop,
    then the holonomy is also called \emph{monodromy}. A connection is called \emph{flat} if the monodromy around each contractible face
    is the identity matrix.
\end {definition}

The group $\osp(1|2)$ is defined as the set of $2|1\times 2|1$ supermatrices $g$ with\footnote{Here $\ber(g)$ denotes 
the supersymmetric analogue of the determinant of a matrix known as the Berezinian and $J$ is the super analaogue of the representation of the standard symplectic form.}
$\ber(g)=1$, and satisfying
\(g^{\st}J g=J\).
Writing $g$ as $\ospmatrix {a}{b}{\gamma} {c}{d}{\delta} {\alpha}{\beta}{e}$, these constraints can be written down explicitly in the following system of equations.
\begin{align}
	e&=1+\alpha\beta  \label{eq:osp1}\\
	e^{-1}&=ad-bc   \label{eq:osp2}\\
	\alpha &=c\gamma-a\delta \label{eq:osp3}\\
	\beta &=d\gamma-b\delta \label{eq:osp4}\\
	\gamma&=a\beta-b\alpha \label{eq:osp5}\\
	\delta&=c\beta-d\alpha \label{eq:osp6}.
\end{align}

We will now define a flat $\mathrm{Osp}(1|2)$-connection on the graph $\Gamma_T$.  

\begin{definition}[elementary steps] \label{def:holonomy_matrices}
	Given $(S,M)$ and $T$, we define the following holonomy matrices for the edges described in \Cref{rmk:Gamma_T_edges_faces}.
        They are pictured in \Cref{fig:holonomy_matrices}.
        \begin {enumerate}
            \item Inside triangle $ijk$, the counter-clockwise orientation of the edge at angle $i$ is assigned the matrix $A(h^i_{jk}|\theta) 
            = \left[\begin{array}{cc|c} 1 & 0 & 0 \\ -h & 1 & \sqrt{h} \theta \\ \hline -\sqrt{h} \theta & 0 & 1 \end{array}\right]$ where 
            $h=h^i_{jk} = \frac{\lambda_{jk}}{\lambda_{ij}\lambda_{ik}}$.
            \item Inside triangle $ijk$, the counter-clockwise orientation of the edge $ij$ is assigned the matrix $E(\lambda_{ij}) = 
            \left[\begin{array}{cc|c} 0 & -\lambda_{ij} & 0 \\ \lambda_{ij}^{-1} & 0 & 0 \\ \hline 0 & 0 & 1 \end{array}\right]$.
            \item For each internal diagonal $ij$, there are two edges of $\Gamma_T$ which cross $ij$. Supposing that the spin structure 
                  has orientation $i \to j$, the edge closer to $i$ is assigned the identity matrix $Id_{2|1}$, and the edge closer to $j$ is assigned the fermionic reflection matrix, 
                  which is given by $\rho = \left[\begin{array}{cc|c} -1 & 0 & 0 \\ 0 & -1 & 0 \\ \hline 0 & 0 & 1 \end{array}\right]$.
        \end {enumerate}
\end{definition}

\begin{figure}[h]
\centering
{
\begin{tabular}{|c|cc|cc|}
\hline &&&&\\[-1.2em]
  Type $(i)$ &
\begin{tabular}{c}
  \begin{tikzpicture}
	\begin{pgfonlayer}{nodelayer}
		\node [style=bullet] (0) at (-3.25, 0) {};
		\node [style=bullet] (1) at (-1, 4) {};
		\node [style=bullet] (2) at (1.5, 0) {};
		\node [style=bullet small] (3) at (-2.25, 1.25) {};
		\node [style=bullet small] (4) at (-1.75, 0.25) {};
		\node [style=none] (5) at (-3.75, -0.5) {$i$};
		\node [style=none] (6) at (-1, 4.75) {$k$};
		\node [style=none] (7) at (2, -0.5) {$j$};
		\node [style=none] (8) at (-1, 1.5) {$\theta$};
	\end{pgfonlayer}
	\begin{pgfonlayer}{edgelayer}
		\draw (0) to (1);
		\draw (1) to (2);
		\draw (2) to (0);
		\draw [style=red, bend left] (3) to (4);
	\end{pgfonlayer}
\end{tikzpicture}
  \end{tabular}&
  $A\left( h^i_{jk} | \theta\right)^{-1}$ &
  \begin{tabular}{c} 
   \begin{tikzpicture}
	\begin{pgfonlayer}{nodelayer}
		\node [style=bullet] (0) at (-3.25, 0) {};
		\node [style=bullet] (1) at (-1, 4) {};
		\node [style=bullet] (2) at (1.5, 0) {};
		\node [style=bullet small] (3) at (-2.25, 1.25) {};
		\node [style=bullet small] (4) at (-1.75, 0.25) {};
		\node [style=none] (5) at (-3.75, -0.5) {$i$};
		\node [style=none] (6) at (-1, 4.75) {$k$};
		\node [style=none] (7) at (2, -0.5) {$j$};
		\node [style=none] (8) at (-1, 1.5) {$\theta$};
	\end{pgfonlayer}
	\begin{pgfonlayer}{edgelayer}
		\draw (0) to (1);
		\draw (1) to (2);
		\draw (2) to (0);
		\draw [style=red 1, bend left] (3) to (4);
	\end{pgfonlayer}
\end{tikzpicture} \end{tabular}&
  $A\left( h^i_{jk} | \theta\right)$ \\&&&&\\[-1.2em]
  \hline &&&&\\[-1.2em]
  Type $(ii)$ &
  \begin{tabular}{c}
  \begin{tikzpicture}
	\begin{pgfonlayer}{nodelayer}
		\node [style=bullet] (0) at (-3.25, 0) {};
		\node [style=bullet] (1) at (-1, 3.75) {};
		\node [style=bullet] (2) at (1.25, 0) {};
		\node [style=bullet small] (3) at (-2.5, 0.25) {};
		\node [style=bullet small] (4) at (0.5, 0.25) {};
		\node [style=none] (5) at (-3.75, -0.5) {$i$};
		\node [style=none] (6) at (-1, 4.5) {$k$};
		\node [style=none] (7) at (1.75, -0.5) {$j$};
		\node [style=none] (8) at (-1, 1.5) {$\theta$};
	\end{pgfonlayer}
	\begin{pgfonlayer}{edgelayer}
		\draw (0) to (1);
		\draw (1) to (2);
		\draw (2) to (0);
		\draw [style=red] (3) to (4);
	\end{pgfonlayer}
\end{tikzpicture}
  \end{tabular}
  &
  $E(\lambda_{ij})^{-1}$ &
  \begin{tabular}{c}\begin{tikzpicture}
	\begin{pgfonlayer}{nodelayer}
		\node [style=bullet] (0) at (-3.25, 0) {};
		\node [style=bullet] (1) at (-1, 3.75) {};
		\node [style=bullet] (2) at (1.25, 0) {};
		\node [style=bullet small] (3) at (-2.5, 0.25) {};
		\node [style=bullet small] (4) at (0.5, 0.25) {};
		\node [style=none] (5) at (-3.75, -0.5) {$i$};
		\node [style=none] (6) at (-1, 4.5) {$k$};
		\node [style=none] (7) at (1.75, -0.5) {$j$};
		\node [style=none] (8) at (-1, 1.5) {$\theta$};
	\end{pgfonlayer}
	\begin{pgfonlayer}{edgelayer}
		\draw (0) to (1);
		\draw (1) to (2);
		\draw (2) to (0);
		\draw [style=red 1] (3) to (4);
	\end{pgfonlayer}
\end{tikzpicture}\end{tabular}
   & 
  $E(\lambda_{ij})$ \\&&&&\\[-1.2em]
  \hline&&&&\\[-1.2em]
  Type $(iii)$ &
 \begin{tabular}{c}\begin{tikzpicture}
	\begin{pgfonlayer}{nodelayer}
		\node [style=bullet] (0) at (-3.25, 0) {};
		\node [style=bullet] (2) at (1.25, 0) {};
		\node [style=bullet small] (3) at (0.25, 0.5) {};
		\node [style=bullet small] (4) at (0.25, -0.5) {};
		\node [style=none] (5) at (-3.75, -0.5) {$i$};
		\node [style=none] (7) at (1.75, -0.5) {$j$};
		\node [style=none] (8) at (-1, 0) {};
	\end{pgfonlayer}
	\begin{pgfonlayer}{edgelayer}
		\draw [style=black arrow] (0) to (8.center);
		\draw (2) to (8.center);
		\draw [style=red, bend left=45] (3) to (4);
		\draw [style=red, bend right=45] (4) to (3);		
	\end{pgfonlayer}
\end{tikzpicture} \end{tabular} &
    $\rho$ &
    \begin{tabular}{c} \begin{tikzpicture}
	\begin{pgfonlayer}{nodelayer}
		\node [style=bullet] (0) at (-3.25, 0) {};
		\node [style=bullet] (2) at (1.25, 0) {};
		\node [style=bullet small] (3) at (-2.5, 0.5) {};
		\node [style=bullet small] (4) at (-2.5, -0.5) {};
		\node [style=none] (5) at (-3.75, -0.5) {$i$};
		\node [style=none] (7) at (1.75, -0.5) {$j$};
		\node [style=none] (8) at (-1, 0) {};
	\end{pgfonlayer}
	\begin{pgfonlayer}{edgelayer}
		\draw [style=red, bend left=45] (3) to (4);
		\draw [style=red, bend right=45] (4) to (3);		
		\draw [style=black arrow] (0) to (8.center);
		\draw (8.center) to (2);
	\end{pgfonlayer}
\end{tikzpicture} \end{tabular}
    &
      $\id$ \\[1em]
  \hline
\end{tabular}
}

\caption {The three types of holonomy matrices. }
\label{fig:holonomy_matrices}
\end{figure}

\begin {prop} \cite[Prop. 3.6 ]{moz22b}
    The holonomy matrices from \Cref{def:holonomy_matrices} define a flat $\mathrm{Osp}(1|2)$-connection on $\Gamma_T$.
\end {prop}

\begin {remark}
    Since the connection is flat, the holonomy between two vertices of $\Gamma_T$ does not depend on the choice of path, and since the graph is planar,
    any two paths are homotopic (thought of as paths on the ambient surface).
\end {remark}

\begin {definition} \label{def:near}
    If vertex $i$ of a polygon is incident to $m$ triangles in $T$, then there are $2m$ vertices of $\Gamma_T$ corresponding to the angles
    of these triangles at $m$. We will say that any of these $2m$ vertices of $\Gamma_T$ are ``\emph{near}'' the vertex $m$.
\end {definition}

\begin {theorem} \cite[Theorem 3.10]{moz22b} \label{thm:12-entry}
    Suppose we have a triangulation $T$ of a polygon endowed with an orientation.  Let $i$ and $j$ be two vertices of the polygon, and $i'$ and $j'$ any vertices of $\Gamma_T$ that are near $i$ and $j$, and let $H$
    be the holonomy from $i'$ to $j'$. Then the $(1,2)$-entry of $H$ is equal to $\pm \lambda_{ij}$.
    \end {theorem}

\begin {lemma} \cite[Lemma 3.11]{moz22b} \label{lem:ij_independence}
    The result of \Cref{thm:12-entry} does not depend on the particular choices of $i'$ and $j'$.
\end {lemma}

In fact, we have geometric interpretations of all nine entries of the holonomy matrix $H$.

Let $T$ be a generic triangulation with default orientation, and with fan centers labelled as $c_i$ for $1\leq i\leq N$. Let $(a,b)$ be the longest diagonal in $T$ and denote $a=c_0$ and $b=c_{N+1}$.

\begin{definition} \label{def:H_ab}
Let $H_{a,b}$ denote the holonomy following a path from a vertex near $a=c_0$ (on the side closer to $c_1$) to a vertex near $b=c_{N+1}$ (on the side closer to $c_N$). We say that the holonomy is of type $\epsilon_a \epsilon_b$ where 
	 \begin{align*}\epsilon_a &= \begin{cases}
	 	0&\text{ if }(c_0,c_1,c_2)\text{ are oriented clockwise,}\\
	 	1&\text{ otherwise.}
	 \end{cases}\\
	 \epsilon_b &=
	 \begin{cases}
	 0&\text{ if }(c_{N-1},c_N,c_{N+1})\text{ are oriented clockwise,}\\
	 	1&\text{ otherwise.}
	 \end{cases}\end{align*}
\end{definition}

\begin{theorem}  \cite[Theorem 4.3]{moz22b} \label{thm:generic}
    Let $T$ be a generic triangulation endowed with an arbitrary orientation (based on its spin structure), 
    and with fan centers labelled as $c_i$ for $1\leq i\leq N$ and $a=c_0,b=c_{N+1}$. 
    The holonomy matrix $H_{a,b}$ of type $\epsilon_a \epsilon_b$ is given by 
    \[ 
        H_{a,b} = \ospmatrix
        {-{\lambda_{c_1,c_{N+1}}\over \lambda_{c_0,c_1}}}
        { (-1)^{\epsilon_a} \lambda_{c_0,c_{N+1}} }
        {\Td^{c_{N+1}}_{c_0,c_1} }
        {(-1)^{\epsilon_b}{\l{c_1}{c_N}\over \lambda_{c_0,c_1}\l {c_N}{c_{N+1}}}}
        {(-1)^{\epsilon_a+\epsilon_b -1}{\l{c_0}{c_N}\over\l{c_N}{c_{N+1}}}}
        {(-1)^{\epsilon_b-1}{1\over \l{c_N}{c_{N+1}}}\Td^{c_N}_{c_0,c_1}}
        {{1\over\lambda_{c_0,c_1}}\td {c_1} {c_N}{c_{N+1}}}
        {(-1)^{\epsilon_a-1}\td {c_0}{c_N}{c_{N+1}}} 
        {1+\star}
    \]

    Here the formula for the $(3,3)$-entry (i.e. $1+\star$) can be given two equivalent ways, which  
    follows from applications of combining \cref{eq:osp1,eq:osp2} to get $ad-bc=1-\alpha\beta$ and then cross multiplying  \cref{eq:osp3,eq:osp5} or \cref{eq:osp4,eq:osp6} to get 
    $\alpha\beta=\gamma\delta$.  Using these we write
    \[ 
        1 + \star = 1 + (-1)^{\epsilon_a-1}{1\over\lambda_{c_0,c_1}}\td {c_1}{c_N}{c_{N+1}}\td {c_0} {c_N}{c_{N+1}}
                  = 1 + (-1)^{\epsilon_b-1}{1\over \l{c_N}{c_{N+1}}}\Td^{c_{N+1}}_{c_0,c_1} \Td^{c_{N}}_{c_0,c_1}.
    \]
\end{theorem}

As discussed in Section 6 of \cite{moz22b}, even though Theorem \ref{thm:generic} is stated only for polygons, the result can be extending to other marked surfaces.  In particular, we can lift a given triangulation $T$ of an arbitrary marked surface to its universal cover and then for any arc $\gamma$ crossing $T$, there is a polygon in the universal cover comprised of the triangles crossed by $\gamma$.  We then appropriately identify odd elements as well as even elements and use Theorem \ref{thm:generic} to get the relevant matrix elements accordingly.  However, for a general marked surface, the lift of the orientation on $T$ to the universal cover will not typically be a manifestation of the default orientation (even modulo the orientation of arcs appearing on the boundary of this lifted polygon).  In particular, spin structures are equivalence classes of orientations up to negating $\mu$-invariants and reversing the orientation on incident edges, as in Figure \ref{fig:super_ptolemy} (Right).  Since Theorem \ref{thm:generic} describes the entries even in the presence of an arbitrary orientation, we can continue to apply our previous results in such cases.

\section{Application to Super Markov Numbers} \label{sec:Markov}

The decorated super-Teichm\"uller space of a once punctured torus was systematically studied in \cite{mcshane}.  In particular, an ideal triangulation of a once punctured torus consists of two triangles, which is most easily seen when a torus is presented as a square with its two horizontal edges identified and its two vertical edges identified.  The four corners of the square are thereby identified together as the unique puncture on the surface.  Lifting the torus to its universal cover, i.e. the Euclidean plane, we use the standard initial triangulation consisting of lines of slope $\frac{0}{1}$, $\frac{1}{0}$, and $\frac{-1}{1}$ as in Figure \ref{fig:super_torus} (Left).  Then, other arcs on the torus correspond to lines of rational slope, plus the vertical line of slope $\frac{1}{0}$, and all such arcs are reachable from the standard initial triangulation by a sequence of diagonal flips.

Up to renaming the three initial arcs, we may assume without loss of generality that our sequence begins with a flip of the initial arc of slope $\frac{-1}{1}$ to get an arc of slope $\frac{1}{1}$, followed by a flip of the initial arc of slope $\frac{1}{0}$ to get an arc of slope $\frac{1}{2}$, see Figure \ref{fig:super_torus} (Middle) and (Right), respectively.  From the triple of arcs of slopes $\{\frac{0}{1}, \frac{1}{1}, \frac{1}{2}\}$ we continue with further flips, as long as we never backtrack, via the rule   

$$\bigg\{\frac{a}{b}, \frac{c}{d}, \frac{e}{f}\bigg\} \to \bigg\{\frac{a}{b}, \frac{a+e}{b+f}, \frac{e}{f}\bigg\} \mathrm{~~or~~} \to \bigg\{\frac{c+e}{d+f}, \frac{c}{d}, \frac{e}{f} \bigg\}.$$

For example, $\{\frac{0}{1}, \frac{1}{1}, \frac{1}{2}\} \to \{\frac{0}{1}, \frac{1}{3}, \frac{1}{2}\}$ or 
$ \to \{\frac{2}{3}, \frac{1}{1}, \frac{1}{2}\}$ depending on whether we flip the arc of slope $\frac{1}{1}$ or the arc of slope $\frac{0}{1}$, respectively.  For any such sequence of flips (without backtracking), we obtain a rational fraction $\frac{p}{q}$ such that 
$p$ and $q$ are relatively prime and such that $0 < \frac{p}{q} \leq 1$.  (Note that if we had started our sequence with other flips, we would instead reach fractions that were negative or greater than $1$.)  Furthermore, every such fraction $0 < \frac{p}{q} \leq 1$ is reachable by a sequence of flips.  

In the ordinary geometry associated to the once-punctured torus, if we endow it with a metric so that the first three initial arcs have $\lambda$-lengths each equal to $1$, and we let $\{x,y,z\}$ denote the $\lambda$-lengths of any other triangulation of the once-punctured torus under the same metric, then $\{x,y,z\}$ is known in number theory as a Markov triple, and have the property that they are the triples that can be reached from $\{1,1,1\}$ by a sequence of flips known as Vieta jumping: 
$\{x,y,z\} \to \{3yz-x, y,z\} \mathrm{~or~} \to \{x, 3xz-y, z\}  \mathrm{~or~} \to \{x, y, 3xy - z\}.$

All such triples involve positive integers, and are in fact the nontrivial solutions to the Diophantine equation, known as the Markov equation:
$$x^2 + y^2 + z^2 = 3xyz.$$

\begin{remark}
Due to the Markov equation, the transformations for Vieta jumping can also be rewritten equivalently as 
$\{x,y,z\} \to \{\frac{y^2+z^2}{x}, y,z\} \mathrm{~or~} \to \{x, \frac{x^2+z^2}{y}, z\}  \mathrm{~or~} \to \{x, y, \frac{x^2+y^2}{z}\},$
which makes the relationship to the cluster algebra associated to the once-punctured torus clearer since, in this form, Vieta jumping agrees with the Ptolemy exchange relation applied to a triangulated quadrilateral in a torus.  See \cite{propp05} for a combinatorial analysis based on this latter perspective.
\end{remark}

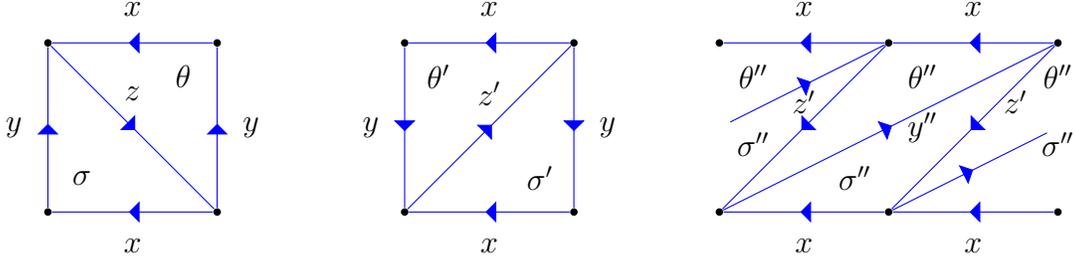
\begin{figure}
\begin{tikzpicture}[scale=0.45,every node/.style={sloped,allow upside down}]
    \node [circle,fill=black,inner sep=1pt] (0) at (-7, 3) {};
    \node [circle,fill=black,inner sep=1pt] (1) at (-7, -2) {};
    \node [circle,fill=black,inner sep=1pt] (2) at (-2, 3) {};
    \node [circle,fill=black,inner sep=1pt] (3) at (-2, -2) {};
    \node [] (4) at (-6, -1) {$\sigma$};
    \node [] (5) at (-3, 2) {$\theta$};
    
    \node [] (6) at (-4.5, -3) {$x$};
    \node [] (7) at (-4.5, 4) {$x$};
    \node [] (8) at (-8, 0.5) {$y$};
    \node [] (9) at (-1, 0.5) {$y$};
     \node [] (10) at (-4.5, 1.5) {$z$};               

    \draw [style=blue] (3) to node {\midarrow} (1);
    \draw [style=blue] (2) to node {\midarrow} (0);
    \draw [style=blue] (0) to node {\midarrow} (3);
    \draw [style=blue] (1) to node {\midarrow}  (0);
    \draw [style=blue] (3) to node {\midarrow}  (2);
\end{tikzpicture} \hspace{2em}
\begin{tikzpicture}[scale=0.45,every node/.style={sloped,allow upside down}]
    \node [circle,fill=black,inner sep=1pt] (0) at (-7, 3) {};
    \node [circle,fill=black,inner sep=1pt] (1) at (-7, -2) {};
    \node [circle,fill=black,inner sep=1pt] (2) at (-2, 3) {};
    \node [circle,fill=black,inner sep=1pt] (3) at (-2, -2) {};
    \node [] (4) at (-3, -1) {$\sigma'$};
    \node [] (5) at (-6, 2) {$\theta'$};
    
    \node [] (6) at (-4.5, -3) {$x$};
    \node [] (7) at (-4.5, 4) {$x$};
    \node [] (8) at (-8, 0.5) {$y$};
    \node [] (9) at (-1, 0.5) {$y$};
     \node [] (10) at (-4.5, 1.5) {$z'$};               

    \draw [style=blue] (3) to node {\midarrow} (1);
    \draw [style=blue] (2) to node {\midarrow} (0);
    \draw [style=blue] (1) to node {\midarrow} (2);
    \draw [style=blue] (0) to node {\midarrow}  (1);
    \draw [style=blue] (2) to node {\midarrow}  (3);
\end{tikzpicture} \hspace{2em}
\begin{tikzpicture}[scale=0.45,every node/.style={sloped,allow upside down}]
    \node [circle,fill=black,inner sep=1pt] (0) at (-7, 3) {};
    \node [circle,fill=black,inner sep=1pt] (1) at (-7, -2) {};
    \node [circle,fill=black,inner sep=1pt] (2) at (-2, 3) {};
    \node [circle,fill=black,inner sep=1pt] (3) at (-2, -2) {};
    \node [circle,fill=black,inner sep=1pt] (4) at (3, 3) {};
    \node [circle,fill=black,inner sep=1pt] (5) at (3, -2) {};    
    \node [] (6) at (-3, -1) {$\sigma''$};
    \node [] (7) at (-6, 2) {$\theta''$};
    \node [] (8) at (3, 0) {$\sigma''$};
    \node [] (9) at (-1, 2) {$\theta''$};

    \node [] (7) at (-6, 0) {$\sigma''$};
    \node [] (8) at (3, 2) {$\theta''$};
    
    \node [] (10) at (-4.5, -3) {$x$};
    \node [] (11) at (-4.5, 4) {$x$};
    \node [] (12) at (0.5, -3) {$x$};
    \node [] (13) at (0.5, 4) {$x$};
  
    \node [] (14) at (-1, 0.5) {$y''$};    
     \node [] (15) at (-4.5, 1.25) {$z'$};               
     \node [] (16) at (1.75, 1.25) {$z'$};                    

     \node [] (17) at (-7, 0.5) {};               
     \node [] (18) at (3, 0.5) {};

    \draw [style=blue] (3) to node {\midarrow} (1);
    \draw [style=blue] (2) to node {\midarrow} (0);
    \draw [style=blue] (2) to node {\midarrow} (1);
    \draw [style=blue] (5) to node {\midarrow} (3);
    \draw [style=blue] (4) to node {\midarrow} (2);
    \draw [style=blue] (4) to node {\midarrow} (3);
    \draw [style=blue] (1) to node {\midarrow} (4);    
    \draw [style=blue] (3) to node {\midarrow} (18);    
    \draw [style=blue] (17) to node {\midarrow} (2);    
    
\end{tikzpicture}

\caption{(Left): Triangulation of a once-punctured torus with its two triangles oriented cyclically.  (Middle): Result after flipping $z$ to get $z'$.  (Right): Then flipping $y$ to get $y''$.  The triangulation continues to be oriented cyclically.}
\label{fig:super_torus}
\end{figure}

In the case of decorated super-Teichm\"uller space of a once punctured torus, with a triangulation cyclically oriented as illustrated in \Cref{fig:super_torus} (Left), flipping any of the three edges results in isomorphic triangulation that is also cyclically oriented.  Consequently, the super $\lambda$-lengths satisfy a super analogue of the Markov equation \cite[Equation 26]{mcshane}:
$$ x^2 + y^2 + z^2 + (xy + yz + xz) \sigma\theta = 3 ( 1 + \sigma \theta) xy z.$$
Note that we can let $\epsilon=\sigma\theta$ and then $\epsilon^2=0$ and we can write solutions as the form $M+\hat{M}\epsilon$.
Mimicing the classical theory, and starting with an initial triangulation where all three $\lambda$-lengths equal $1$, we will refer to the super $\lambda$-lengths associated to arcs of slope $0 < \frac{p}{q} \leq 1$ on the once-punctured torus as {\bf Super Markov Numbers}.  Furthermore, by using the Super Ptolemy exchange relation in the formulas for Vieta jumping,  we can get from one triple of Super Markov Numbers to another by 
$\{x,y,z\} \to \{ (y^2 ~+~ yz\sigma\theta ~+~  z^2)x^{-1}, y,z\} \mathrm{~or~} \{x,y,z\} \to \{x, (x^2 ~+~ xz\sigma\theta ~+~  z^2)y^{-1}, z\}  \mathrm{~or~} \{x,y,z\} \to \{x, y, (x^2 ~+~ xy\sigma\theta ~+~ y^2)z^{-1}\}.$

For the set of ordinary Markov numbers, there is a unique sequence of Vieta jumps that transforms the initial Markov triple correspolnding to rational slopes $\frac{0}{1}$, $\frac{1}{0}$, and $\frac{-1}{1}$ to a Markov triple of Markov numbers corresponding to three compatible rational slopes.  (The one hundred year-old Uniqueness Conjecture \cite{aigner2015markov} goes even further and asserts that for any non-initial Markov number, there is a unique sequence of Vieta jumps that reaches a triple containing that number.)  We lift this indexing by rational slopes to our aforementioned super analogues by letting the Super Markov Triple $\{SM_{a/b}, SM_{c/d}, SM_{e/f}\}$ correspond to the triple obtained by the same sequence of Vieta jumps that takes the initial ordinary Markov triple to that corresponding to the Markov triple for rational slopes $\{a/b, c/d, e/f\}$.   
Using this, the first several Super Markov Numbers arise as triples as:
\begin{eqnarray*}
\{SM_{0/1}, SM_{1/1}, SM_{1/2}\} &=& \{1, \hspace{5em} 2+\sigma\theta, \hspace{1em} 5 + 6\sigma\theta\}, \\
\{SM_{0/1}, SM_{1/2}, SM_{1/3}\} &=& \{1, \hspace{4.5em} 5 + 6\sigma\theta, \hspace{1em} 13 + 26\sigma\theta\}, \\
\{SM_{1/1}, SM_{1/2}, SM_{2/3}\} &=& \{2+\sigma\theta, \hspace{2.5em} 5 + 6\sigma\theta, \hspace{1em} 29 + 74\sigma\theta\}, \\ 
\{SM_{1/1}, SM_{2/3}, SM_{3/4}\} &=& \{2+\sigma\theta, \hspace{1.5em} 29 + 74\sigma\theta, \hspace{1em} 169 + 668\sigma\theta\}, \\
\{SM_{1/2}, SM_{1/3}, SM_{2/5}\} &=& \{5+6\sigma\theta, \hspace{1em} 13 + 26\sigma\theta, \hspace{1em} 194 + 801\sigma\theta\}, \\
\{SM_{1/2}, SM_{2/3}, SM_{3/5}\} &=& \{5+6\sigma\theta, \hspace{1em} 29 + 74\sigma\theta, \hspace{1em} 433 + 2032\sigma\theta\}.
\end{eqnarray*}

\subsection{Matrix Approach to Super Markov Numbers}
\label{sec:holonomy}

As a first approach to computing further Super Markov Numbers, we utilize the Holonomy matrix formulas provided by Theorem \ref{thm:generic}.  To do this, we need to first introduce some notation from the Combinatorics on Words.

Given $0 < \frac{p}{q} \leq 1$ with $\gcd(p,q)=1$, we draw $\gamma$ of slope $p/q$ on the universal cover of the torus.  As described in \cite[Section 1.2]{berstel2009combinatorics}, one can use the sequence of residues $a_n = nq \mod (p+q)$ for $0 \leq n \leq p+q$ to determine the upper Christoffel word, equivalently a lattice path of north and east steps that lies just above line $\gamma$.  In particular, we construct the corresponding
upper Christoffel word of length $(p+q)$ as $w_1w_2\cdots w_{p+q}$ via the characteristic function $w_n = \chi( a_{n-1} < a_n)$.  This determines a lattice path from $(0,0)$ to $(p,q)$ so that the $n$th step is a north step $N$ whenever $a_{n-1} < a_n$ and is an east step $E$ whenever $a_{n-1} > a_n$.

To see why this construction works, for each lattice point $(i,j)$, we associate the rational number $\frac{jq - ip}{q}$.  Notice that $(i,j)$ lies strictly above the line $\gamma$ of slope $\frac{p}{q}$ if and only if the corresponding rational number $\frac{jq - ip}{q} = j - i \frac{p}{q} > 0$.  Thus under this association, and assuming that the point $(i,j)$ lies on the lattice path just above the line $\gamma$ of slope $\frac{p}{q}$, then the point $(i+1,j)$ also lies above the line $\gamma$ if and only if the quantity $jq - (i+1)p$ is still positive.  Thus the ensuing step of the lattice path will be $E$ only when we can subtract $p$ from $(jq - ip)$ and still result in a positive number.  Otherwise, the ensuing step is $N$ and we add $q$ to this quantity.   Since the difference between successive values is either $-p$ or $+q$, these quantities, i.e. the numerators of the resulting fractions, are exactly the values of $a_n = nq \mod (p+q)$.

\begin{example}
If $p/q = 3/5$, we get $(a_0,a_1,\dots, a_8) = (0,5,2,7,4,1,6,3,0)$ and hence the upper Christoffel word is $10100100$, corresponding to the lattice path given by 
$NENEENEE$, which lies just above the arc $\gamma$ of slope $3/5$, and goes through the sequence of points $(0,0), (0,1), (1,1), (1,2), (2,2), (3,2), (3,3), (4,3), (5,3)$ which are associated to the fractions $\frac{0}{5},~ \frac{5}{5},~ \frac{2}{5},~ \frac{7}{5},~ \frac{4}{5}, ~\frac{1}{5}, ~ \frac{6}{5},~ \frac{3}{5},~\frac{0}{5}$,
subtracting $\frac{3}{5}$ when possible, and adding $\frac{5}{5}$ otherwise.
\end{example}

\begin{definition} \label{def:truncated}
Given the Upper Christoffel word $w$ associated to fraction $p/q$, let $\widetilde{w}$ denote the {\bf truncated Christoffel word} obtained by removing initial $1$ and final $0$ from $w$.  Equivalently, we have removed the initial $N$ step and final $E$ step from the corresponding lattice path.
\end{definition}

\begin{remark}
\label{Rem:Truncated}
As explained in \cite[Proposition 4.2]{berstel2009combinatorics}, if we instead truncate the lower Christoffel word, equivalently a lattice path of north and east steps that lies just {\bf below} line $\gamma$, then we recover the same {\bf truncated Christoffel word}, hence why we do not need the descriptor ``upper" or ``lower'' in Definition \ref{def:truncated}.  In particular, if the upper lattice path is $N \widetilde{w} E$ and the lower lattice path is $E \widetilde{w'} N$, then we get $\widetilde{w} = \widetilde{w'}$.
\end{remark}

\begin{prop}
If $\widetilde{w}$ is the truncated Christoffel word associated to the line $\gamma$ of slope $p/q$ on the once-punctured torus, then the Holonomy matrix associated to the path given by line $\gamma$ is given as the matrix product $M_{fin}~M_{p+q} \cdots M_2 M_1$ where for $1 \leq i \leq p+q$, we have
\begin{eqnarray*}
M_i &=& 
\left[ \begin{matrix} 
0 & - x & 0  \\
-\frac{1}{x} & \frac{x^2 + y^2 + z^2}{yz} + x \frac{z + y}{yz} \sigma\theta & 
( \sqrt{ \frac{z}{xy}} - \sqrt{ \frac{y}{xz}}) \ \sigma + \sqrt{\frac{x}{yz}}\theta \\
0 &( \sqrt{ \frac{xz}{y}} - \sqrt{ \frac{xy}{z}}) \ \sigma + x \sqrt{\frac{x}{yz}}\theta & 1 
\end{matrix}\right]
 \mathrm{~~if~}\widetilde{w}_i = 0, \mathrm{~~or} \\
M_i &=& 
\left[ \begin{matrix} 
\frac{z}{x} & y & -\sqrt{\frac{yz}{x}}\sigma \\
\frac{y}{x^2} & \frac{x^2 + y^2}{xy} + \frac{y}{z}\sigma\theta & -\frac{y}{x} \sqrt{\frac{y}{xz}} \sigma + \sqrt{\frac{y}{xz}}\theta \\
- \frac{1}{x} \sqrt{\frac{yz}{x}} \theta & - \sqrt{\frac{xy}{z}} \sigma - y \sqrt{\frac{y}{xz}}\theta & 
1 - \frac{y}{x} \sigma\theta
\end{matrix}\right]
\mathrm{~~if~}\widetilde{w}_i = 1, \mathrm{~and~}
\end{eqnarray*}
\begin{equation*}
M_{fin} = 
\left[ \begin{matrix} 
\frac{y}{z} & \frac{x^2 + y^2}{z} + \frac{x y}{z} \sigma \theta & - y \sqrt{ \frac{y}{xz} } \sigma +\sqrt{\frac{ xy}{z}} \theta \\
-\frac{z}{x^2} & - \frac{y}{x} & \frac{1}{x} \sqrt{\frac{yz}{x}} \sigma \\
- \frac{1}{x}\sqrt{\frac{yz}{x}} \sigma &  - \sqrt{\frac{xy}{z} } \sigma - y \sqrt{\frac{ y}{xz}} \theta &
1 - \frac{y}{x} \sigma\theta
\end{matrix}\right].
\end{equation*}

\end{prop}

\begin{proof}

Following the elementary steps corresponding to the canonical path associated to $\gamma$, we may write out $H$ as the matrix product
$M_{fin}~M_{p+q} \cdots M_2 M_1$ where
$$M_i = A^{-1}\bigg( \frac{z}{xy} ~\vline~ \sigma\bigg) ~ Id_{2|1} ~ A^{-1}\bigg( \frac{x}{yz} ~\vline~ \theta\bigg) ~\rho ~ A^{-1}\bigg( \frac{y}{xz} ~\vline~ \sigma\bigg)
~ E^{-1}(x) \mathrm{~~if~}\widetilde{w}_i = 0, \mathrm{~~or}$$
$$M_i = \rho ~ A\bigg( \frac{y}{xz} ~\vline~ \theta\bigg) ~ Id_{2|1} ~ A\bigg( \frac{x}{yz} ~\vline~ \sigma\bigg) ~ E(y) ~ A\bigg( \frac{z}{xy} ~\vline~ \sigma\bigg) \mathrm{~~if~}\widetilde{w}_i = 1,$$
$$\mathrm{and~~~} M_{fin} = E(x) ~A\bigg( \frac{y}{xz} ~\vline~ \theta\bigg) ~Id_{2|1}~A\bigg( \frac{x}{yz} ~\vline~ \sigma\bigg) ~E(y) ~ A\bigg( \frac{z}{xy} ~\vline~ \sigma\bigg)$$
as in Figures \ref{fig:canon_paths} and \ref{fig:canon_paths0}.
In particular, suppose we begin our canonically defined path in the triangle labeled as $\sigma$ at the angle bordered by sides $x$ and $y$ but at the point closer to side $x$.  We proceed via a sequence of elementary steps so that at the end of each such sequence, we are back in the same relative position.  We multiply together from right-to-left the matrices corresponding to each sequence of elementary steps.  Finally, for any such canonical path, we end with a sequence of steps that ends in the triangle labeled as $\theta$ (again at the angle bordered by sides $x$ and $y$ but at the point closer to side $x$).  Since the places where the line $\gamma$ crosses the arc $x$ then $z$ (as opposed to $y$ then $z$) is encoded by the truncated Christoffel word, we use grouping of elementary steps of one type or the other.  We cross triangles and thus include $\mu$-invariants $\sigma$ and $\theta$ as well as the fermionic reflection $\rho$ accordingly as well.
\end{proof}

\begin{figure}
\begin{tikzpicture}[scale=0.8,every node/.style={sloped,allow upside down}]

    \node [circle,fill=black,inner sep=1pt] (0) at (-7, 3) {};
    \node [circle,fill=black,inner sep=1pt] (1) at (-7, -2) {};
    \node [circle,fill=black,inner sep=1pt] (2) at (-2, 3) {};
    \node [circle,fill=black,inner sep=1pt] (3) at (-2, -2) {};
    \node [] (4) at (-6, -1) {$\sigma$};
    \node [] (5) at (-3, 2) {$\theta$};
    
    \node [] (6) at (-4.5, -3) {$x$};
    \node [] (7) at (-4.5, 4) {$x$};
    \node [] (8) at (-8, 0.5) {$y$};
    \node [] (9) at (-1, 0.5) {$y$};
     \node [] (10) at (-4.5, 1.5) {$z$};               

    \draw [style=blue] (3) to node {\midarrow} (1);
    \draw [style=blue] (2) to node {\midarrow} (0);
    \draw [style=blue] (0) to node {\midarrow} (3);
    \draw [style=blue] (1) to node {\midarrow}  (0);
    \draw [style=blue] (3) to node {\midarrow}  (2);

    \node [circle,fill=black,inner sep=1pt] (11) at (3, 3) {};
    \node [circle,fill=black,inner sep=1pt] (12) at (3, -2) {};
    \node [] (13) at (-1, -1) {$\sigma$};
    \node [] (14) at (2, 2) {$\theta$};
    
    \node [] (15) at (0.5, -3) {$x$};
    \node [] (16) at (0.5, 4) {$x$};
    \node [] (18) at (4, 0.5) {$y$};
     \node [] (19) at (0.5, 1.5) {$z$};               

    \draw [style=blue] (12) to node {\midarrow} (3);
    \draw [style=blue] (12) to node {\midarrow} (11);
    \draw [style=blue] (2) to node {\midarrow} (12);
    \draw [style=blue] (11) to node {\midarrow} (2);            
    
    \node  [circle,fill=black,inner sep=2pt] (21) at (-5.75, -1.75) {};
    \node  [circle,fill=black,inner sep=2pt] (22) at (-6.75, -0.75) {};   
     
    \node  [circle,fill=black,inner sep=2pt] (23) at (-3.25, -1.75) {};            
    \node  [circle,fill=black,inner sep=2pt] (24) at (-2.9, -1.4) {};    
    \node  [circle,fill=black,inner sep=2pt] (25) at (-2.6, -1.1) {}; 
    \node  [circle,fill=black,inner sep=2pt] (26) at (-2.25, -0.75) {};
       
    \node  [circle,fill=black,inner sep=2pt] (27) at (-0.75, -1.75) {};    
    \node  [circle,fill=black,inner sep=2pt] (28) at (-1.75, -0.75) {};

    \draw [style=red] (21) to (23);    
    \draw [style=red] (23) to (24);    
    \draw [style=red] (24) to (25);    
    \draw [style=red] (25) to (26);    
    \draw [style=red] (26) to (28);    
    \draw [style=red] (28) to (27);

\end{tikzpicture}
\begin{tikzpicture}[scale=0.8,every node/.style={sloped,allow upside down}]

    \node [circle,fill=black,inner sep=1pt] (0) at (-7, 3) {};
    \node [circle,fill=black,inner sep=1pt] (1) at (-7, -2) {};
    \node [circle,fill=black,inner sep=1pt] (2) at (-2, 3) {};
    \node [circle,fill=black,inner sep=1pt] (3) at (-2, -2) {};
    \node [] (4) at (-6, -1) {$\sigma$};
    \node [] (5) at (-3, 2) {$\theta$};
    
    \node [] (6) at (-4.5, -3) {$x$};
    \node [] (7) at (-4.5, 4) {$x$};
    \node [] (8) at (-8, 0.5) {$y$};
    \node [] (9) at (-1, 0.5) {$y$};
     \node [] (10) at (-4.5, 1.5) {$z$};               

    \draw [style=blue] (3) to node {\midarrow} (1);
    \draw [style=blue] (2) to node {\midarrow} (0);
    \draw [style=blue] (0) to node {\midarrow} (3);
    \draw [style=blue] (1) to node {\midarrow}  (0);
    \draw [style=blue] (3) to node {\midarrow}  (2);

    \node [circle,fill=black,inner sep=1pt] (11) at (-7, 8) {};
    \node [circle,fill=black,inner sep=1pt] (12) at (-2, 8) {};
    \node [] (13) at (-6, 4) {$\sigma$};
    \node [] (14) at (-3, 7) {$\theta$};
    
    \node [] (16) at (-4.5, 9) {$x$};
    \node [] (17) at (-8, 5.5) {$y$};
    \node [] (18) at (-1, 5.5) {$y$};
     \node [] (19) at (-4.5, 6.5) {$z$};               

    \draw [style=blue] (0) to node {\midarrow} (11);
    \draw [style=blue] (12) to node {\midarrow} (11);
    \draw [style=blue] (2) to node {\midarrow} (12);
    \draw [style=blue] (11) to node {\midarrow} (2);     
    
    \node  [circle,fill=black,inner sep=2pt] (21) at (-5.75, -1.75) {};
    \node  [circle,fill=black,inner sep=2pt] (22) at (-6.75, -0.75) {};    
    \node  [circle,fill=black,inner sep=2pt] (23) at (-6.75, 1.75) {};        
    \node  [circle,fill=black,inner sep=2pt] (24) at (-6.4, 2.1) {};    
    \node  [circle,fill=black,inner sep=2pt] (25) at (-6.1, 2.4) {};    
    \node  [circle,fill=black,inner sep=2pt] (26) at (-5.75, 2.75) {};   
    
    \node  [circle,fill=black,inner sep=2pt] (29) at (-5.75, 3.25) {};    
    \node  [circle,fill=black,inner sep=2pt] (30) at (-6.75, 4.25) {};    
 
    \draw [style=red] (21) to (22);    
    \draw [style=red] (22) to (23);        
    \draw [style=red] (23) to (24);    
    \draw [style=red] (24) to (25);    
    \draw [style=red] (25) to (26);    
    \draw [style=red] (26) to (29);

\end{tikzpicture} 
\caption{Paths for $M_i$ where $\widetilde{w_i} = 0$ (on the left) or $1$ (on the right).}
\label{fig:canon_paths}
\end{figure}
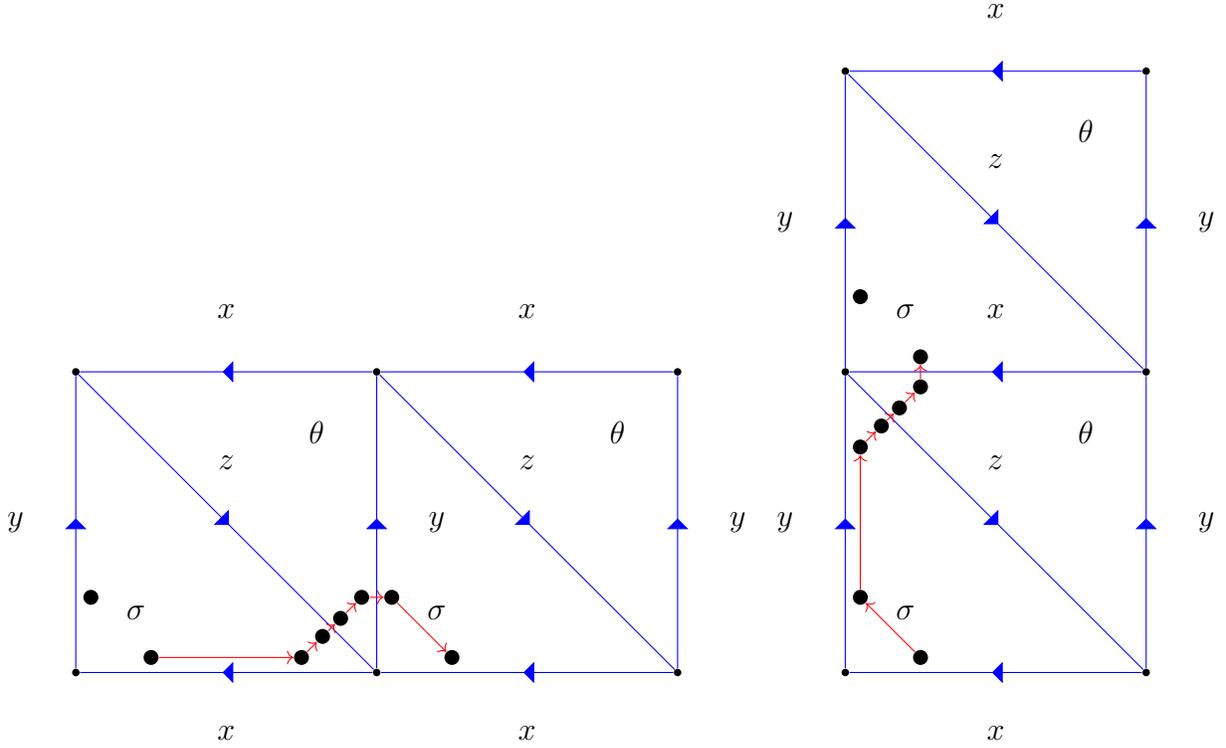

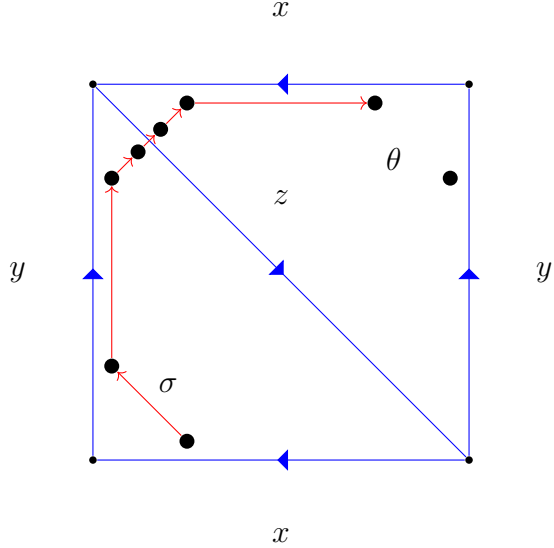
\begin{figure}
\begin{tikzpicture}[scale=1.0,every node/.style={sloped,allow upside down}]
    \node [circle,fill=black,inner sep=1pt] (0) at (-7, 3) {};
    \node [circle,fill=black,inner sep=1pt] (1) at (-7, -2) {};
    \node [circle,fill=black,inner sep=1pt] (2) at (-2, 3) {};
    \node [circle,fill=black,inner sep=1pt] (3) at (-2, -2) {};
    \node [] (4) at (-6, -1) {$\sigma$};
    \node [] (5) at (-3, 2) {$\theta$};
    
    \node [] (6) at (-4.5, -3) {$x$};
    \node [] (7) at (-4.5, 4) {$x$};
    \node [] (8) at (-8, 0.5) {$y$};
    \node [] (9) at (-1, 0.5) {$y$};
     \node [] (10) at (-4.5, 1.5) {$z$};               

    \draw [style=blue] (3) to node {\midarrow} (1);
    \draw [style=blue] (2) to node {\midarrow} (0);
    \draw [style=blue] (0) to node {\midarrow} (3);
    \draw [style=blue] (1) to node {\midarrow}  (0);
    \draw [style=blue] (3) to node {\midarrow}  (2);
    
    \node  [circle,fill=black,inner sep=2pt] (11) at (-5.75, -1.75) {};
    \node  [circle,fill=black,inner sep=2pt] (12) at (-6.75, -0.75) {};    
    \node  [circle,fill=black,inner sep=2pt] (13) at (-6.75, 1.75) {};        
    \node  [circle,fill=black,inner sep=2pt] (14) at (-6.4, 2.1) {};    
    \node  [circle,fill=black,inner sep=2pt] (15) at (-6.1, 2.4) {};    
    \node  [circle,fill=black,inner sep=2pt] (16) at (-5.75, 2.75) {};   
    \node  [circle,fill=black,inner sep=2pt] (17) at (-3.25, 2.75) {};    
    \node  [circle,fill=black,inner sep=2pt] (18) at (-2.25, 1.75) {};    
 
    \draw [style=red] (11) to (12);    
    \draw [style=red] (12) to (13);        
    \draw [style=red] (13) to (14);    
    \draw [style=red] (14) to (15);    
    \draw [style=red] (15) to (16);    
    \draw [style=red] (16) to (17);            
    
\end{tikzpicture}
\caption{Path for $M_{fin}$.}
\label{fig:canon_paths0}
\end{figure}

\begin{remark}
If we let $x=y=z=1$, as is our convention for defining Super Markov Numbers $SM_{p/q}$, in contrast to the super $\lambda$-lengths, the matrices simplify to 
$$M_i = 
\left[ \begin{matrix} 
0 & - 1 & 0  \\
-1 & 3 + 2 \sigma\theta & \theta \\
0 & \theta & 1 
\end{matrix}\right]
\mathrm{~~if~}\widetilde{w}_i = 0, \mathrm{~~or~~}
M_i = 
\left[ \begin{matrix} 
1 & 1 & -\sigma \\
1 & 2 + \sigma\theta & - \sigma + \theta \\
- \theta & - \sigma - \theta & 1 - \sigma\theta
\end{matrix}\right]
\mathrm{~~if~}\widetilde{w}_i = 1,$$
$$ \mathrm{~~and~~~} M_{fin} =
\left[ \begin{matrix} 
1& 2+ \sigma\theta & - \sigma + \theta \\
-1 & - 1 & \sigma \\
- \sigma &  - \sigma - \theta &
1 -  \sigma\theta
\end{matrix}\right].$$
\end{remark}

We then read off $SM_{p/q}$ as the $(1,2)$-entry of the holonomy matrix $H$ corresponding to the line $\gamma$ of slope $p/q$.

\begin{example}
If $p/q = 1/1$, the Upper Christoffel word is $10$ and hence the truncated Christoffel word is empty.  The corresponding Holonomy matrix is simply $M_{fin}$ and the $(1,2)$-entry is indeed $2 + \sigma\theta = SM_{1/1}$ as desired.

If $p/q = 1/2$, the Upper Christoffel word is $100$ and hence the truncated Christoffel word is $0$.  The corresponding Holonomy matrix is 
$$
\left[ \begin{matrix} 
1& 2+  \sigma\theta & - \sigma + \theta \\
-1 & - 1 & \sigma \\
- \sigma &  - \sigma - \theta &
1 -  \sigma\theta
\end{matrix}\right]
\left[ \begin{matrix} 
0 & - 1 & 0  \\
-1 & 3 + 2 \sigma\theta & \theta \\
0 & \theta & 1 
\end{matrix}\right] =
\left[\begin{matrix} 2 + \sigma\theta & 5 + 6 \sigma\theta & -\sigma + 3\theta \\ 
-1 & -2 - \sigma\theta & \sigma - \theta \\
-\sigma - \theta & -3\sigma - \theta & 1 - 2\sigma\theta \end{matrix}\right]$$
 and the $(1,2)$-entry is indeed $5 + 6\sigma\theta = SM_{1/2}$ as desired. 

If $p/q = 2/3$, the Upper Christoffel word is $10100$ and hence the truncated Christoffel word is $010$.  The corresponding Holonomy matrix is 
$$
\left[ \begin{matrix} 
1& 2+  \sigma\theta & - \sigma + \theta \\
-1 & - 1 & \sigma \\
- \sigma &  - \sigma - \theta &
1 -  \sigma\theta
\end{matrix}\right]
\left[ \begin{matrix} 
0 & - 1 & 0  \\
-1 & 3 + 2 \sigma\theta & \theta \\
0 & \theta & 1 
\end{matrix}\right]
\left[ \begin{matrix} 
1 & 1 & -\sigma \\
1 & 2 + \sigma\theta & - \sigma + \theta \\
- \theta & - \sigma - \theta & 1 - \sigma\theta
\end{matrix}\right]
\left[ \begin{matrix} 
0 & - 1 & 0  \\
-1 & 3 + 2 \sigma\theta & \theta \\
0 & \theta & 1 
\end{matrix}\right]
=$$
$$\left[\begin{matrix} 12 + 22\sigma\theta & 29 + 74 \sigma\theta & -8\sigma + 20\theta \\ 
-5 - 6\sigma\theta  & -12 - 22\sigma\theta & 4\sigma - 8\theta \\
-8\sigma - 4\theta & -20\sigma - 8\theta & 1 - 16\sigma\theta \end{matrix}\right]$$
 and the $(1,2)$-entry is indeed $29 + 74\sigma\theta = SM_{2/3}$ as desired. 

For our running example of $p/q=3/5$, we get truncated Christoffel word $010010$ and holonomy matrix
$$\left[\begin{matrix} 179 + 706\sigma\theta & 433 + 2032 \sigma\theta & -112\sigma + 303\theta \\ 
-74 - 237\sigma\theta  & -179 - 706\sigma\theta & 47\sigma - 125\theta \\
-125\sigma - 47\theta & -303\sigma - 112\theta & 1 - 241\sigma\theta \end{matrix}\right].$$
The $(1,2)$-entry is $SM_{3/5} = 433 + 2032\sigma\theta$. 
\end{example}

\subsection{Combinatorial Approach to Super Markov Numbers}

\label{sec:combo}

Our main result regarding Super Markov Numbers is a combinatorial interpretation that extends the previously known snake graph formula for classical Markov numbers.
We use the constructions of Christoffel words and associated lattice paths as in Section \ref{sec:holonomy}.  To such an upper lattice path, we may associate a {\bf snake graph} $G_{p/q}$ by the following procedure.  We first build a word we called the {\bf snake word} by applying the following substitutions.

Step 1: To account for the initial step (which must always be a $N$ step), begin the snake word with a $z$.

Step 2: For the remaining steps in our upper lattice path (except for the final step), for each $N$ step in the upper lattice path, append an $xz$ to the snake word, and for each $E$ step in the truncated lattice path, append a $yz$. 

Step 3: For the final step of the upper lattice path (which must always be an $E$ step), we leave the snake word alone.  Thus, our final snake word has length $2(p+q)-3$.

Step 4: Snake graphs are comprised of tiles, i.e. quadrilaterals or $4$-cycles, that are labeled as $X$, $Y$, or $Z$, where the label determines the corresponding edge weighting and is also used elsewhere in the expansion formulas.

If a tile is labeled $X$, its two horizontal edges have weights of $y$ and its two vertical edges have weights of $z$. 

If a tile is labeled $Y$, its two horizontal edges have weights of $z$ and its two vertical edges have weights of $x$.  

If a tile is labeled $Z$, its two horizontal edges have weights of $y$ and its two vertical edges have weights of $x$.  

Step 5: Beginning with a single tile with label $Z$ for the first letter of the snake word, add tiles to the $N$ or $E$ of the previously placed tile iteratively as follows.  If the ensuing letter is $x$, place a tile labeled $X$ to the $N$.  If the ensuing letter is $y$, place a tile labeled $Y$ to the $E$.  And if the ensuing letter is $z$, place a tile labeled $Z$ continuing the same direction as its predecessor.  The resulting snake graph consists of $2(p+q)-3$ tiles. 

Step 6: Each of these can also be triangulated by drawing an additional diagonal arrow on the NW-SE axis and for usage later, we note that every tile with a label of $Z$ also includes labels by $\mu$-invariants such that $\sigma$ appears in the lower-left triangle and $\theta$ appears in the upper right.  For the tiles with labels of $X$ or $Y$, their two triangles get labels by $\mu$-invariants matching the adjacent triangles in the tiles with a label of $Z$.  This labelling follows from the unfolding map of \cite[Sec. 4.5]{ms10}.

This procedure constructs the snake graph $G_{p/q}$ corresponding to arc $\gamma$ of slope $p/q$, as described in \cite{propp05} or \cite{MSW_11}, since tile $X$ arises whenever $\gamma$ would cross a horizontal arc in the fundamental domain, tile $Y$ arises whenever $\gamma$ would cross a vertical arc in the fundamental domain, and tile $Z$ arises whenever $\gamma$ would cross a diagonal arc in the fundamental domain.  

\begin{example}
In our continuing example of $p/q = 3/5$, corresponding to the upper lattice path $N(ENEENE)E$ (with truncated lattice path $ENEENE$) , the associated snake word is $z ~y z ~x z ~y z ~y z ~x z ~y z$.  These are illustrated in Figures \ref{fig:torus-3-5} and \ref{fig:snake-3-5}. 
\end{example}

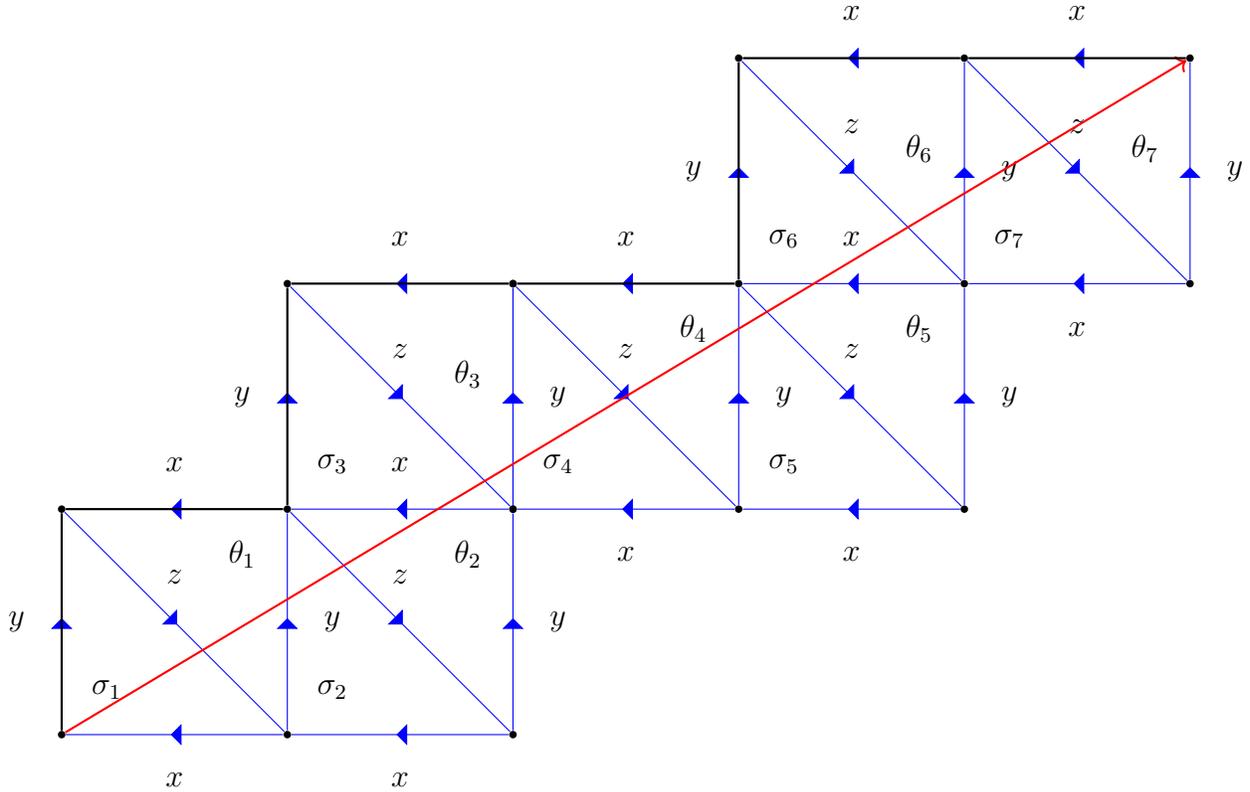
\begin{figure}
\begin{tikzpicture}[scale=0.6,every node/.style={sloped,allow upside down}]

    \node [circle,fill=black,inner sep=1pt] (0) at (-7, 3) {};
    \node [circle,fill=black,inner sep=1pt] (1) at (-7, -2) {};
    \node [circle,fill=black,inner sep=1pt] (2) at (-2, 3) {};
    \node [circle,fill=black,inner sep=1pt] (3) at (-2, -2) {};
    \node [] (4) at (-6, -1) {$\sigma_1$};
    \node [] (5) at (-3, 2) {$\theta_1$};
    
    \node [] (6) at (-4.5, -3) {$x$};
    \node [] (7) at (-4.5, 4) {$x$};
    \node [] (8) at (-8, 0.5) {$y$};
    \node [] (9) at (-1, 0.5) {$y$};
     \node [] (10) at (-4.5, 1.5) {$z$};               

    \draw [style=blue] (3) to node {\midarrow} (1);
    \draw [style=blue] (2) to node {\midarrow} (0);
    \draw [style=blue] (0) to node {\midarrow} (3);
    \draw [style=blue] (1) to node {\midarrow}  (0);
    \draw [style=blue] (3) to node {\midarrow}  (2);

    \node [circle,fill=black,inner sep=1pt] (11) at (3, 3) {};
    \node [circle,fill=black,inner sep=1pt] (12) at (3, -2) {};
    \node [] (13) at (-1, -1) {$\sigma_2$};
    \node [] (14) at (2, 2) {$\theta_2$};
    
    \node [] (15) at (0.5, -3) {$x$};
    \node [] (16) at (0.5, 4) {$x$};
    \node [] (18) at (4, 0.5) {$y$};
     \node [] (19) at (0.5, 1.5) {$z$};               

    \draw [style=blue] (12) to node {\midarrow} (3);
    \draw [style=blue] (12) to node {\midarrow} (11);
    \draw [style=blue] (2) to node {\midarrow} (12);
    \draw [style=blue] (11) to node {\midarrow} (2);      
    
    \node [circle,fill=black,inner sep=1pt] (20) at (3, 8) {};
    \node [circle,fill=black,inner sep=1pt] (21) at (-2, 8) {};
    \node [] (22) at (-1, 4) {$\sigma_3$};
    \node [] (23) at (2, 6) {$\theta_3$};
    
    \node [] (25) at (0.5, 9) {$x$};
    \node [] (26) at (-3, 5.5) {$y$};
    \node [] (27) at (4, 5.5) {$y$};
     \node [] (28) at (0.5, 6.5) {$z$};               

    \draw [style=blue] (20) to node {\midarrow} (21);
    \draw [style=blue] (11) to node {\midarrow} (20);    
    \draw [style=blue] (2) to node {\midarrow} (21);    
    \draw [style=blue] (21) to node {\midarrow} (11);    
    
    \node [circle,fill=black,inner sep=1pt] (30) at (3, 8) {};
    \node [circle,fill=black,inner sep=1pt] (31) at (3, 3) {};
    \node [circle,fill=black,inner sep=1pt] (32) at (8, 8) {};
    \node [circle,fill=black,inner sep=1pt] (33) at (8, 3) {};
    \node [] (34) at (4, 4) {$\sigma_4$};
    \node [] (35) at (7, 7) {$\theta_4$};
    
    \node [] (36) at (5.5, 2) {$x$};
    \node [] (37) at (5.5, 9) {$x$};
    \node [] (39) at (9, 5.5) {$y$};
     \node [] (40) at (5.5, 6.5) {$z$};               

    \draw [style=blue] (33) to node {\midarrow} (31);
    \draw [style=blue] (32) to node {\midarrow} (30);
    \draw [style=blue] (30) to node {\midarrow} (33);
    \draw [style=blue] (31) to node {\midarrow}  (30);
    \draw [style=blue] (33) to node {\midarrow}  (32);

    \node [circle,fill=black,inner sep=1pt] (41) at (13, 8) {};
    \node [circle,fill=black,inner sep=1pt] (42) at (13, 3) {};
    \node [] (43) at (9, 4) {$\sigma_5$};
    \node [] (44) at (12, 7) {$\theta_5$};
    
    \node [] (45) at (10.5, 2) {$x$};
    \node [] (46) at (10.5, 9) {$x$};
    \node [] (48) at (14, 5.5) {$y$};
     \node [] (49) at (10.5, 6.5) {$z$};               

    \draw [style=blue] (42) to node {\midarrow} (33);
    \draw [style=blue] (42) to node {\midarrow} (41);
    \draw [style=blue] (32) to node {\midarrow} (42);
    \draw [style=blue] (41) to node {\midarrow} (32);      
    
    \node [circle,fill=black,inner sep=1pt] (50) at (13, 13) {};
    \node [circle,fill=black,inner sep=1pt] (51) at (8, 13) {};
    \node [] (52) at (9, 9) {$\sigma_6$};
    \node [] (53) at (12, 11) {$\theta_6$};
    
    \node [] (55) at (10.5, 14) {$x$};
    \node [] (56) at (7, 10.5) {$y$};
    \node [] (57) at (14, 10.5) {$y$};
     \node [] (58) at (10.5, 11.5) {$z$};               

    \draw [style=blue] (50) to node {\midarrow} (51);
    \draw [style=blue] (41) to node {\midarrow} (50);    
    \draw [style=blue] (32) to node {\midarrow} (51);    
    \draw [style=blue] (51) to node {\midarrow} (41);              
          
    \node [circle,fill=black,inner sep=1pt] (60) at (18, 8) {};
    \node [circle,fill=black,inner sep=1pt] (61) at (18, 13) {};
    \node [] (62) at (14, 9) {$\sigma_7$};
    \node [] (63) at (17, 11) {$\theta_7$};
    
    \node [] (64) at (15.5, 7) {$x$};
    \node [] (65) at (15.5, 14) {$x$};
    \node [] (67) at (19, 10.5) {$y$};
     \node [] (68) at (15.5, 11.5) {$z$};               

    \draw [style=blue] (60) to node {\midarrow} (61);
    \draw [style=blue] (50) to node {\midarrow} (60);    
    \draw [style=blue] (61) to node {\midarrow} (50);  
    \draw [style=blue] (60) to node {\midarrow} (41);        
          
    \draw [style=red,thick] (1) to (61);    
    
    \draw [color=black, thick] (1) to (0);
    \draw [color=black, thick] (0) to (2);
    \draw [color=black, thick] (2) to (21);    
    \draw [color=black, thick] (21) to (20);
    \draw [color=black, thick] (20) to (32);            
    \draw [color=black, thick] (32) to (51);            
    \draw [color=black, thick] (51) to (50);            
    \draw [color=black, thick] (50) to (61);                      
\end{tikzpicture}
\caption{Arc $\gamma$ of slope $3/5$ (in red) crossing the triangulation of a once-punctured torus lifted to its universal cover, with upper lattice path $NENEENEE$ illustrated (in black)  For each copy of the fundamental domain, we use a different subscript to label the two associated $\mu$-invariants.}
\label{fig:torus-3-5}
\end{figure}

\begin{figure}
\begin{tikzpicture}[scale=0.3,every node/.style={sloped,allow upside down}]

    \node [circle,fill=black,inner sep=1pt] (0) at (0, 0) {};
    \node [circle,fill=black,inner sep=1pt] (1) at (5, 0) {};
    \node [circle,fill=black,inner sep=1pt] (2) at (10, 0) {};
    \node [circle,fill=black,inner sep=1pt] (3) at (15, 0) {};
    \node [circle,fill=black,inner sep=1pt] (4) at (0, 5) {};    
    \node [circle,fill=black,inner sep=1pt] (5) at (5, 5) {};
    \node [circle,fill=black,inner sep=1pt] (6) at (10, 5) {};
    \node [circle,fill=black,inner sep=1pt] (7) at (15, 5) {};
    \node [circle,fill=black,inner sep=1pt] (8) at (10, 10) {};
    \node [circle,fill=black,inner sep=1pt] (9) at (15, 10) {};    
    \node [circle,fill=black,inner sep=1pt] (10) at (20, 10) {};
    \node [circle,fill=black,inner sep=1pt] (11) at (25, 10) {};    
    \node [circle,fill=black,inner sep=1pt] (12) at (30, 10) {};
    \node [circle,fill=black,inner sep=1pt] (13) at (35, 10) {};    
    \node [circle,fill=black,inner sep=1pt] (14) at (10, 15) {};
    \node [circle,fill=black,inner sep=1pt] (15) at (15, 15) {};    
    \node [circle,fill=black,inner sep=1pt] (16) at (20, 15) {};
    \node [circle,fill=black,inner sep=1pt] (17) at (25, 15) {};    
    \node [circle,fill=black,inner sep=1pt] (18) at (30, 15) {};
    \node [circle,fill=black,inner sep=1pt] (19) at (35, 15) {};
    \node [circle,fill=black,inner sep=1pt] (20) at (30, 20) {};
    \node [circle,fill=black,inner sep=1pt] (21) at (35, 20) {};
    \node [circle,fill=black,inner sep=1pt] (22) at (40, 20) {};
    \node [circle,fill=black,inner sep=1pt] (23) at (45, 20) {};
    \node [circle,fill=black,inner sep=1pt] (24) at (30, 25) {};
    \node [circle,fill=black,inner sep=1pt] (25) at (35, 25) {};    
    \node [circle,fill=black,inner sep=1pt] (26) at (40, 25) {};
    \node [circle,fill=black,inner sep=1pt] (27) at (45, 25) {};    
    
    \draw [style=blue] (0) to (3);
    \draw [style=blue] (4) to (7);
    \draw [style=blue] (8) to (13);
    \draw [style=blue] (14) to (19);
    \draw [style=blue] (20) to (23);              
    \draw [style=blue] (24) to (27);       

    \draw [style=blue] (0) to (4);
    \draw [style=blue] (1) to (5);       
    \draw [style=blue] (2) to (14);   
    \draw [style=blue] (3) to (15);           
    \draw [style=blue] (10) to (16);
    \draw [style=blue] (11) to (17);          
    \draw [style=blue] (12) to (24);
    \draw [style=blue] (13) to (25);
    \draw [style=blue] (22) to (26);
    \draw [style=blue] (23) to (27);

    \node [] (28) at (2.5, 2.5) {$Z$};
    \node [] (29) at (7.5, 2.5) {$Y$};
    \node [] (30) at (12.5, 2.5) {$Z$};
    \node [] (31) at (12.5, 7.5) {$X$};
    \node [] (32) at (12.5, 12.5) {$Z$};
    \node [] (33) at (17.5, 12.5) {$Y$};
    \node [] (34) at (22.5, 12.5) {$Z$};
    \node [] (35) at (27.5, 12.5) {$Y$};
    \node [] (36) at (32.5, 12.5) {$Z$};
    \node [] (37) at (32.5, 17.5) {$X$};
    \node [] (38) at (32.5, 22.5) {$Z$};
    \node [] (39) at (37.5, 22.5) {$Y$};
    \node [] (40) at (42.5, 22.5) {$Z$};

    \draw [style=blue,dotted] (1) to (4);
    \draw [style=blue,dotted] (2) to (5);
    \draw [style=blue,dotted] (3) to (6);
    \draw [style=blue,dotted] (7) to (8);
    \draw [style=blue,dotted] (9) to (14);              
    \draw [style=blue,dotted] (10) to (15);       
    \draw [style=blue,dotted] (11) to (16);
    \draw [style=blue,dotted] (12) to (17);
    \draw [style=blue,dotted] (13) to (18);
    \draw [style=blue,dotted] (19) to (20);
    \draw [style=blue,dotted] (21) to (24);              
    \draw [style=blue,dotted] (22) to (25);       
    \draw [style=blue,dotted] (23) to (26);
          
    \node [] (28) at (1, 1) {$\sigma_1$};
    \node [] (29) at (6, 1) {$\theta_1$};
    \node [] (30) at (11, 1) {$\sigma_2$};
    \node [] (31) at (11, 6) {$\theta_2$};
    \node [] (32) at (11, 11) {$\sigma_3$};
    \node [] (33) at (16, 11) {$\theta_3$};
    \node [] (34) at (21, 11) {$\sigma_4$};
    \node [] (35) at (26, 11) {$\theta_4$};
    \node [] (36) at (31, 11) {$\sigma_5$};
    \node [] (37) at (31, 16) {$\theta_5$};
    \node [] (38) at (31, 21) {$\sigma_6$};
    \node [] (39) at (36, 21) {$\theta_6$};
    \node [] (40) at (41, 21) {$\sigma_7$};

    \node [] (28) at (4, 4) {$\theta_1$};
    \node [] (29) at (9, 4) {$\sigma_2$};
    \node [] (30) at (14, 4) {$\theta_2$};
    \node [] (31) at (14, 9) {$\sigma_3$};
    \node [] (32) at (14, 14) {$\theta_3$};
    \node [] (33) at (19, 14) {$\sigma_4$};
    \node [] (34) at (24, 14) {$\theta_4$};
    \node [] (35) at (29, 14) {$\sigma_5$};
    \node [] (36) at (34, 14) {$\theta_5$};
    \node [] (37) at (34, 19) {$\sigma_{6}$};
    \node [] (38) at (34, 24) {$\theta_6$};
    \node [] (39) at (39, 24) {$\sigma_7$};
    \node [] (40) at (44, 24) {$\theta_7$};

\end{tikzpicture}
\caption{Snake Graph corresponding to the arc $\gamma$ of slope $3/5$ with tiles corresponding to crossings of $\gamma$ with the lifted triangulation as in \Cref{fig:torus-3-5}.}
\label{fig:snake-3-5}
\end{figure}
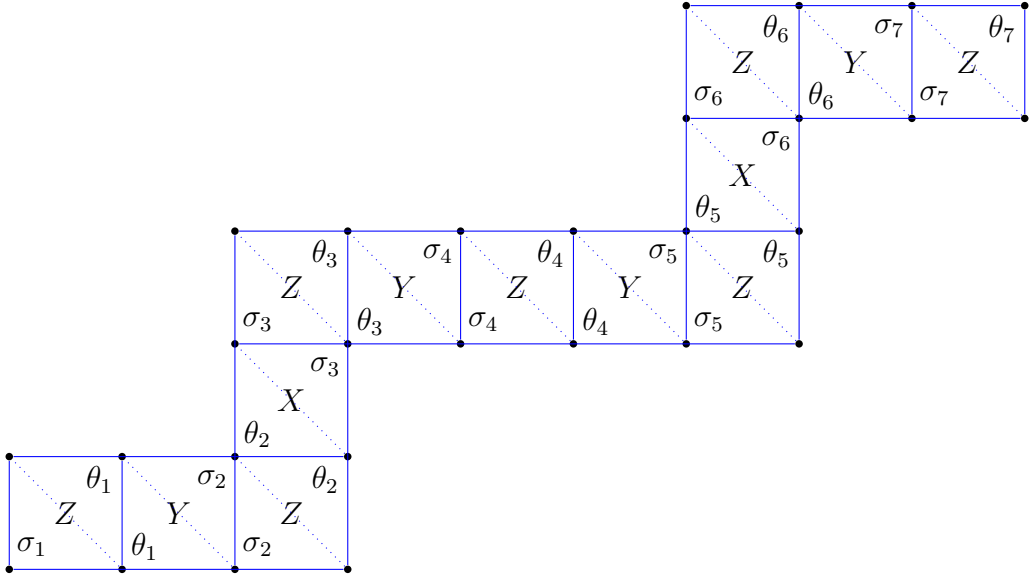

\begin{remark}
For our present application of expansion formulas for super $\lambda$-lengths, we build a snake graph $G_{p/q}$ as above, but instead of utilizing only the two $\mu$-invariants $\sigma$ and $\theta$ throughout the entire graph, we consider the alphabet of $\mu$-invariants $ \sigma_1,\sigma_2, \dots, \sigma_{p+q-1},~~$  $\theta_1, \theta_2, \dots, \theta_{p+q-1}$ so that reading from left-to-right, the $i$th iteration of the tile with label $Z$ uses $\mu$-invariants $\sigma_i$ and $\theta_i$.  We will later identify all $\sigma_i$'s as equivalent and all $\theta_i$'s as equivalent, so that in particular $\sigma_i \sigma_j = 0$ and $\theta_i \theta_j = 0$ for all choices of $i$ and $j$.  However, this notation distinguishing between $\mu$-invariants will be useful during our exposition.
\end{remark}

Since the once-punctured torus has no boundary, we also have less freedom in choosing an orientation that corresponds to our spin structure.  We pick an orientation that is cyclically oriented on the two triangles in our standard triangulation, as we previously illustrated in \Cref{fig:super_torus}.  In particular, we orient triangles clockwise on the lower left triangle (corresponding to $\mu$-invariant $\sigma$) and orient triangles counter-clockwise on the upper right triangle (corresponding to $\mu$-invariant $\theta$).  An equivalent way of describing this orientation is upward on vertical arcs, leftward on horizontal arcs, and southeastern on diagonal arcs, and this orientation is consistent when identifying opposite sides of the torus to each other.

We recall from \cite{propp05} that the classical Markov number $M_{p/q}$ counts the number of perfect matching (a.k.a. dimer covers) of the snake graph $G_{p/q}$ constructed from $p/q$ as above.  (See also the construction of the snake graph $G_{p/q}$ from \cite{rabideau2020continued} where it is built by placing tiles to the right in positions corresponding to $0$'s in the Christoffel word and placing tiles on the top in position corresponding to $1$'s in the Christoffel word.)  Note also that each tile of $G_{p/q}$ can be associated with either initial arc $x$, $y$, or $z$, determined by the order in which arcs of the initial triangulation are crossed by the arc of slope $\frac{p}{q}$.

We now extend this definition to instead enumerate double dimer covers (which are multisets of edges such that every vertex is incident to exactly two distinguished edges) of the snake graph $G_{p/q}$.  Every ordinary perfect matching of $G_{p/q}$ corresponds to a double dimer cover by simply doubling every utilized edge.  We additionally get other double dimer covers which contain cycles (where each edge along the cycle is utilized once) as well as (possibly) doubled edges.  With snake graphs now defined, we can complete our definition of double dimer cover weights, as based on \cite[Definition 4.4]{moz22}.  Given double dimer cover $M$, we note that a cycle $C$ as a component of $M$ always encircles a number of triangles in the snake graph as embedded inside the universal cover.  Thus, there is a triangle in the lower-left of such a cycle $C$ and a triangle in the upper-right of such as cycle $C$.  We thus define $wt(C) = \mu_i \mu_j$ where $\mu_i$ is the $\mu$-invariant of the lower left and $\mu_j$ is the $\mu$-invariant of the upper right.  Then
$$wt(M) = \prod_{e \in M} wt(e) \prod_{\mathrm{Cycle~}C \mathrm{~of~}M} wt(C).$$ We lastly note that the product $\prod_{\mathrm{Cycle~}C \mathrm{~of~}M} wt(C)$ is re-ordered so that it is expressed according to the positive ordering defined in Section \ref{sec:prelim}.

As a consequence of 
the way that odd elements (i.e. $\mu$-invariants) are defined for the once-punctured torus, not all possible double dimer covers contribute to our enumeration. Any double dimer cover containing two or more cycles contributes zero to this enumeration and double dimers containing a single cycle are weighted as described in \Cref{def:markov_soul}.

\begin{definition}
\label{def:markov_soul}
The value of $\hat{M}_{p/q}$ is the signed enumeration of double dimer covers on $G_{p/q}$ containing a single cycle of odd length as follows:

(\# containing a single cycle whose leftmost tile is $X$ and whose rightmost tile is also $X$)  

+ (\# containing a single cycle whose leftmost tile is $Y$ and whose rightmost tile is also $Y$)  

+ (\# containing a single cycle whose leftmost tile is $Z$ and whose rightmost tile is also $Z$)  

- (\# containing a single cycle whose leftmost tile is $X$ but whose rightmost tile is $Y$)  

- (\# containing a single cycle whose leftmost tile is $Y$ but whose rightmost tile is $X$).  
\end{definition}

We deduce the follow combinatorial interpretation, in terms of double dimer covers, of Super Markov numbers accordingly.

\begin{theorem}
\label{thm:supermarkov}
Let $\gamma_{p/q}$ be an arc of slope $\frac{p}{q}$, such that $0 < \frac{p}{q} \leq 1$, on the once-punctured torus, where we have chosen the standard initial triangulation corresponding to $\{\frac{0}{1}, \frac{1}{0}, \frac{-1}{1}\}$ and set the (super) $\lambda$-lengths of these three initial arcs all equal to $1$.  Then the super $\lambda$-length of $\gamma_{p/q}$ 
is given by the Super Markov number $SM_{p/q} = M_{p/q} + \hat{M}_{p/q} \sigma \theta$ where $M_{p/q}$ is the classical Markov 
number associated to the arc of slope $\frac{p}{q}$ and $\hat{M}_{p/q}$ is the signed enumeration given in \Cref{def:markov_soul}.
\end{theorem}

We prove this result in \Cref{sec:main_proof} breaking it down into a series of lemmas.

\section{Proof of \Cref{thm:supermarkov}}
\label{sec:main_proof}

We begin by describing how the snake words constructed in Steps 1 - 3 of \Cref{sec:combo} allow us to identify a polygon and cut it out of the universal cover of the triangulated once-punctured torus.

\begin{lemma}
\label{lem:snake_word}
Given the upper lattice path defined by fraction $p/q$ (where $0 < p/q \leq 1$), the snake word defined by $N \to xz$ and $E \to yz$ (with special rules of $z$ for the first step of $N$ and no letter for the final step of $E$) describes the sequence of initial arcs crossed by $\gamma_{p/q}$, the line of slope $p/q$ in the universal cover of the triangulated once-punctured torus.
\end{lemma}

\begin{proof}
Let $a_0,a_1,\dots, a_{p+q}$ denote the points incident to adjacent steps of the upper lattice paths, as described in the beginning of \Cref{sec:holonomy}.  We then define points $b_0, b_1, \dots, b_{p+q}$ on the line segment $\gamma_{p/q}$ by setting $b_0=a_0$, $b_{p+q}=a_{p+q}$ and for each $1 \leq i \leq p+q-1$, we set $b_i$ to be a point in the interior of the upper right triangle of the triangulated fundamental domain of the once-punctured torus incident to $a_i$.  See \Cref{fig:snake_word_proof}.

Since the slope of $\gamma_{p/q}$ is assumed to be positive, the initial arc first crossed by $\gamma_{p/q}$ is necessarily $z$.  If the slope is $\frac{1}{1}$, this is in fact the only arc crossed by $\gamma$ and we verify that the snake word of $z$ indeed corresponds to the upper lattice path of $NE$.  Next, in addition to assuming that the slope of $\gamma_{p/q}$ is positive, we assume that it is bounded between $0$ and $1$.  
As a consequence, the second and third of the initial arcs crossed by $\gamma$ are necessarily $y$ then $z$ (in that order) and the first two steps of the upper lattice path must be $NE$.  We observe that after following along line segment $\gamma_{p/q}$ so that we have crossed $z$ and $y$, we are now at the point $b_2$.   We see that the prefix of such snake words is $zyz$, which matches the construction from an upper lattice path with prefix $NE$.  

Following $\gamma_{p/q}$ to the point $b_3$, the next arcs crossed by $\gamma$ are then either $zx$ or $zy$ (in that order) depending on whether the upper lattice path's next step is $N$ or $E$, respectively.  We continue applying this rule iteratively until we reach the point $b_{p+q-1}$.  Each time, the steps of the upper lattice path exactly matches the next two letters of the constructed snake word.  Since $b_{p+q-1}$ already lies in the interior of the upper right triangle of the fundamental domain, following $\gamma_{p/q}$ from $b_{p+q-1}$ to $b_{p+q}$, the endpoint of $\gamma_{p/q}$ does not contribute any further crossed arcs or elements to the snake word.  At the same time, the final step of the upper lattice path from $a_{p+q-1}$ to $a_{p+q}$ must be an $E$ step.
\end{proof}

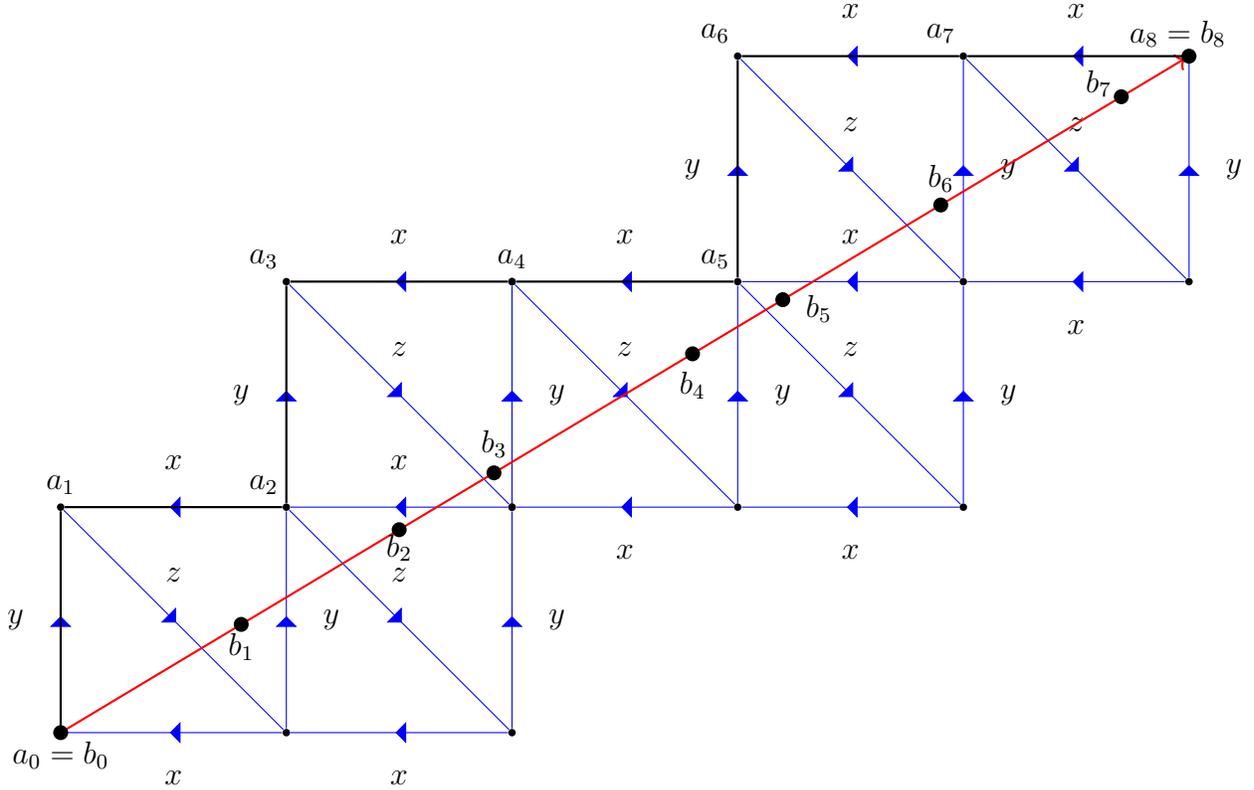
\begin{figure}
\begin{tikzpicture}[scale=0.6,every node/.style={sloped,allow upside down}]

    \node [circle,fill=black,inner sep=1pt] (0) at (-7, 3) {};
    \node [circle,fill=black,inner sep=1pt] (1) at (-7, -2) {};
    \node [circle,fill=black,inner sep=1pt] (2) at (-2, 3) {};
    \node [circle,fill=black,inner sep=1pt] (3) at (-2, -2) {};
    
    \node [] (6) at (-4.5, -3) {$x$};
    \node [] (7) at (-4.5, 4) {$x$};
    \node [] (8) at (-8, 0.5) {$y$};
    \node [] (9) at (-1, 0.5) {$y$};
     \node [] (10) at (-4.5, 1.5) {$z$};               

    \draw [style=blue] (3) to node {\midarrow} (1);
    \draw [style=blue] (2) to node {\midarrow} (0);
    \draw [style=blue] (0) to node {\midarrow} (3);
    \draw [style=blue] (1) to node {\midarrow}  (0);
    \draw [style=blue] (3) to node {\midarrow}  (2);

    \node [circle,fill=black,inner sep=1pt] (11) at (3, 3) {};
    \node [circle,fill=black,inner sep=1pt] (12) at (3, -2) {};
    
    \node [] (15) at (0.5, -3) {$x$};
    \node [] (16) at (0.5, 4) {$x$};
    \node [] (18) at (4, 0.5) {$y$};
     \node [] (19) at (0.5, 1.5) {$z$};               

    \draw [style=blue] (12) to node {\midarrow} (3);
    \draw [style=blue] (12) to node {\midarrow} (11);
    \draw [style=blue] (2) to node {\midarrow} (12);
    \draw [style=blue] (11) to node {\midarrow} (2);      
    
    \node [circle,fill=black,inner sep=1pt] (20) at (3, 8) {};
    \node [circle,fill=black,inner sep=1pt] (21) at (-2, 8) {};
    
    \node [] (25) at (0.5, 9) {$x$};
    \node [] (26) at (-3, 5.5) {$y$};
    \node [] (27) at (4, 5.5) {$y$};
     \node [] (28) at (0.5, 6.5) {$z$};               

    \draw [style=blue] (20) to node {\midarrow} (21);
    \draw [style=blue] (11) to node {\midarrow} (20);    
    \draw [style=blue] (2) to node {\midarrow} (21);    
    \draw [style=blue] (21) to node {\midarrow} (11);    
    
    \node [circle,fill=black,inner sep=1pt] (30) at (3, 8) {};
    \node [circle,fill=black,inner sep=1pt] (31) at (3, 3) {};
    \node [circle,fill=black,inner sep=1pt] (32) at (8, 8) {};
    \node [circle,fill=black,inner sep=1pt] (33) at (8, 3) {};
    
    \node [] (36) at (5.5, 2) {$x$};
    \node [] (37) at (5.5, 9) {$x$};
    \node [] (39) at (9, 5.5) {$y$};
     \node [] (40) at (5.5, 6.5) {$z$};               

    \draw [style=blue] (33) to node {\midarrow} (31);
    \draw [style=blue] (32) to node {\midarrow} (30);
    \draw [style=blue] (30) to node {\midarrow} (33);
    \draw [style=blue] (31) to node {\midarrow}  (30);
    \draw [style=blue] (33) to node {\midarrow}  (32);

    \node [circle,fill=black,inner sep=1pt] (41) at (13, 8) {};
    \node [circle,fill=black,inner sep=1pt] (42) at (13, 3) {};
    
    \node [] (45) at (10.5, 2) {$x$};
    \node [] (46) at (10.5, 9) {$x$};
    \node [] (48) at (14, 5.5) {$y$};
     \node [] (49) at (10.5, 6.5) {$z$};               

    \draw [style=blue] (42) to node {\midarrow} (33);
    \draw [style=blue] (42) to node {\midarrow} (41);
    \draw [style=blue] (32) to node {\midarrow} (42);
    \draw [style=blue] (41) to node {\midarrow} (32);      
    
    \node [circle,fill=black,inner sep=1pt] (50) at (13, 13) {};
    \node [circle,fill=black,inner sep=1pt] (51) at (8, 13) {};
    
    \node [] (55) at (10.5, 14) {$x$};
    \node [] (56) at (7, 10.5) {$y$};
    \node [] (57) at (14, 10.5) {$y$};
     \node [] (58) at (10.5, 11.5) {$z$};               

    \draw [style=blue] (50) to node {\midarrow} (51);
    \draw [style=blue] (41) to node {\midarrow} (50);    
    \draw [style=blue] (32) to node {\midarrow} (51);    
    \draw [style=blue] (51) to node {\midarrow} (41);              
          
    \node [circle,fill=black,inner sep=1pt] (60) at (18, 8) {};
    \node [circle,fill=black,inner sep=1pt] (61) at (18, 13) {};
    
    \node [] (64) at (15.5, 7) {$x$};
    \node [] (65) at (15.5, 14) {$x$};
    \node [] (67) at (19, 10.5) {$y$};
     \node [] (68) at (15.5, 11.5) {$z$};               

    \draw [style=blue] (60) to node {\midarrow} (61);
    \draw [style=blue] (50) to node {\midarrow} (60);    
    \draw [style=blue] (61) to node {\midarrow} (50);  
    \draw [style=blue] (60) to node {\midarrow} (41);        
          
    \draw [style=red,thick] (1) to (61);    
    
    \draw [color=black, thick] (1) to (0);
    \draw [color=black, thick] (0) to (2);
    \draw [color=black, thick] (2) to (21);    
    \draw [color=black, thick] (21) to (20);
    \draw [color=black, thick] (20) to (32);            
    \draw [color=black, thick] (32) to (51);            
    \draw [color=black, thick] (51) to (50);            
    \draw [color=black, thick] (50) to (61);                      
    
    \node  [circle,fill=black,inner sep=2pt] (100) at (-7, -2) {};
    \node  [] (200) at (-7, -2.5) {$a_0=b_0$};

    \node  [circle,fill=black,inner sep=2pt] (101) at (-3, 0.4) {};
    \node  [] (201) at (-3, -0.1) {$b_1$};
    \node  [] (301) at (-7, 3.5) {$a_1$};

    \node  [circle,fill=black,inner sep=2pt] (102) at (0.5, 2.5) {};
    \node  [] (202) at (0.5, 2.1) {$b_2$};
    \node  [] (302) at (-2.5, 3.5) {$a_2$};

    \node  [circle,fill=black,inner sep=2pt] (103) at (2.6, 3.76) {};
    \node  [] (203) at (2.6, 4.4) {$b_3$};
    \node  [] (303) at (-2.5, 8.5) {$a_3$};

    \node  [circle,fill=black,inner sep=2pt] (104) at (7, 6.4) {};
    \node  [] (204) at (7, 5.7) {$b_4$};
    \node  [] (304) at (3, 8.5) {$a_4$};

    \node  [circle,fill=black,inner sep=2pt] (105) at (9, 7.6) {};
    \node  [] (205) at (9.8, 7.4) {$b_5$};
    \node  [] (305) at (7.5, 8.5) {$a_5$};

    \node  [circle,fill=black,inner sep=2pt] (106) at (12.5, 9.7) {};
    \node  [] (206) at (12.5, 10.3) {$b_6$};
    \node  [] (305) at (7.5, 13.5) {$a_6$};

    \node  [circle,fill=black,inner sep=2pt] (107) at (16.5, 12.1) {};
    \node  [] (207) at (16, 12.4) {$b_7$};
    \node  [] (307) at (12.5, 13.5) {$a_7$};

    \node  [circle,fill=black,inner sep=2pt] (108) at (18, 13) {};
    \node  [] (308) at (17.75, 13.5) {$a_8=b_8$};
    
\end{tikzpicture}
\caption{Points $b_0, b_1, b_2, \dots, b_8$ on arc $\gamma_{3/5}$.}
\label{fig:snake_word_proof}
\end{figure}

As a consequence of \Cref{lem:snake_word}, we let $T_\gamma$ denote the polygon (cut out of the universal cover of the triangulated once-punctured torus) inscribing the sequence of arcs crossed by $\gamma=\gamma_{p/q}$.  This polygon, $T_\gamma$, can equivalently be defined as the boundary between the upper lattice path and the lower lattice path.  The cyclic orientation of the two triangles making up the fundamental domain of the triangulated once-punctured torus, as in \Cref{fig:super_torus} (Left) induces an orientation of $T_\gamma$ so that every arc labeled $x$ is oriented to the left, every arc labeled $y$ is oriented upwards, and every arc labeled $z$ is oriented southeastern.  To transform this orientation into the default orientation on the interior of $T_\gamma$, i.e. not including boundary edges, we negate $\mu$-invariants as described in the next result.

\begin{lemma}
\label{lem:negations}
If the upper Christoffel word for the line $\gamma$ of slope $p/q$ has $1$'s in positions $i_1=1, i_2,\dots, i_p$, then on the interior edges of $T_\gamma$, the default orientation differs from the lift of the cyclic orientation by negating $\sigma_{i_2}, \sigma_{i_3},\dots, \sigma_{i_p}$.
\end{lemma}

\begin{proof}
Since the polygon $T_\gamma$ is triangulated based on the fundamental domain as in \Cref{fig:super_torus} (Left), we get a configuration where
each fan segment is either a single triangle or a quadrilateral, depending on the upper Christoffel word.  In particular, every entry of the upper Christoffel word, besides the final entry, corresponds to either two triangular fan segments or a single quadrilateral fan segment.  If an entry is the prefix or suffix of an instance of the two-letter subword $01$, it corresponds to a quadrilateral fan segment.  All other entries correspond to two triangular fan segments.

Of these fan segments, only the quadrilateral fan segments corresponding to entry $1$ (in position $i$) lead to edge orientations in the default orientation that differ from the cyclic orientation, and the difference corresponds to the orientation on the interior edges $x$ and $z$, which borders $\sigma_i$ exactly as stated. 
\end{proof}

The last piece of theory we need before proving our main theorem is the positive ordering determined by the default orientation on $T_\gamma$.

If we label the $\mu$-invariants as $\theta_1, \theta_2, \dots, \theta_n$ in order according to the crossings by arc $\gamma$, then starting from the end we apply the following procedure.  Suppose that $\theta_{k+1}, \theta_{k+2}, \dots, \theta_n$ are already ordered.  Then if the edge separating $\theta_k$ and $\theta_{k+1}$ is oriented so that $\theta_k$ is to the right, then we declare $\theta_k > \theta_i$ for all $k+1 \leq i \leq n$.  On the other hand, if $\theta_k$ is to the left of $\theta_{k+1}$, then we declare $\theta_k < \theta_i$ for all $k+1 \leq i \leq n$.  Applying this first to the case $k=n-1$ and following by descent until $k=1$ yields an ordering on all $n$ $\mu$-invariants.

\begin{lemma}
\label{lem:pos_ordering}
Given the same hypotheses as in \Cref{lem:negations}, the positive ordering associated to the associated cyclic orientation (after negating $\sigma_i$'s for $i$'s indexing positions of non-initial $1$'s in the upper Christoffel word) is given by the following two step process.

Step 1: Initialize the ordering as $\sigma_1 > \sigma_2 > \cdots \sigma_n > \theta_n > \theta_{n-1} > \dots > \theta_1.$

Step 2: For every instance of $01$ in the upper Christoffel word, where $i$ denotes the position of this non-initial $1$, switch the placements of $\sigma_{i}$ with $\theta_{i-1}$ and negate $\sigma_i$.
\end{lemma}

\begin{proof}
Following \Cref{def:pos_ordering}, the cyclic orientation of $T_\gamma$ induces the positive ordering
$$ \sigma_1 > \sigma_2 > \cdots \sigma_n > \theta_n > \theta_{n-1} > \dots > \theta_1,$$ which leads to well-defined formulas when we project down to the original torus with only two $\mu$-invariants (by letting $\sigma_1 = \sigma_2 = \cdots = \sigma_n = \sigma$ and $\theta_1 = \theta_2 = \cdots = \theta_n = \theta$) since $\sigma_j > \theta_k$ for all $j$ and $k$.

However, after negating the values of $\sigma_i$'s based on the positions of non-initial $1$'s in the upper Christoffel word, the positive ordering associated to the resulting default orientation may have the sequences of $\pm \sigma_j$'s and $\theta_k$'s interweaved. 

In particular, starting from the last quadrilateral, we begin determining our ordering by looking at the edge between the $\mu$-invariants $\theta_{p+q-1}$ and $\pm \sigma_{p+q-1}$ and using \Cref{def:pos_ordering} to determine if $\theta_{p+q-1} > \pm \sigma_{p+q-1}$ or  $\theta_{p+q-1} < \pm \sigma_{p+q-1}$.  We then move along $\gamma$ backwards, assuming we have previously constructed a positive ordering to the set of $\mu$-invariants 
$\bigg\{\pm \sigma_{k+1}, \pm \sigma_{k+2}, \dots, \pm \sigma_{p+q-1}, \theta_{p+q-1}, \theta_{p+q-2}, \dots, \theta_{k+2}, \theta_{k+1}, \theta_k\bigg\}$.  We will denote this ordered block as $B_k$.  We observe that if no $1$ appears to right of position $k$, then the ordering of $B_k$ would be $
\sigma_{k+1} > \sigma_{k+2} > \cdots > \sigma_n > \theta_n > \theta_{n-1} > \dots > \theta_{k}$, which matches the ordering expected from the cyclic orientation.  

The first occurrence of an edge which is oriented differently is the edge $z$ between $\sigma_i$ and $\theta_i$ such that the $i$th position of the upper Christoffel word is $1$.  Consequently, instead of $\sigma_i$ being greater than every $\mu$-invariant coming after it in the triangulation (which is the case in the cyclic triangulation), we instead get that $(-\sigma_i)$ is less than every $\mu$-invariant coming afterwards.

We next see the edge $x$ between $\theta_{i-1}$ and $\sigma_i$ is oriented oppositely.  Thus instead of $\theta_{i-1}$ being minimal among all $\mu$-invariants coming afterwards, we get that $\theta_{i-1}$ is maximal.  Thus if we originally had 
$$\sigma_i > B_i > \theta_{i-1}$$ and we have replaced this with $$\theta_{i-1} > B_i > (-\sigma_i).$$
We continue inserting $\theta_j$ and $\sigma_k$'s in order as in the cyclic triangulation case until the next time we see an occurrence of $01$ earlier in the upper Christoffel word.

Iterating this process, we get the positive ordering defined by the two step process as desired.
\end{proof}

\begin{example}
For arc $\gamma = \gamma_{3/5}$, the default orientation of $T_\gamma$ is determined by the fan centers $c_0,c_1,\dots, c_{11}$ as in \Cref{fig:default-3-5}.  The upper Christoffel word for this example is $10100100$, which has non-initial $1$'s in positions $3$ and $6$. (The quadrilateral fan segments correspond to the subwords $01$ in the $2$nd and $3$rd entries as well as the $5$th and $6$th entries.)  Negating $\sigma_3$ and $\sigma_6$ in the cyclic orientation of \Cref{fig:torus-3-5} results in this default orientation by reversing the direction of the red arrows incident to $c_4$ and $c_8$, respectively.  The positive ordering associated to this default orientation is 
$$\sigma_1 > \sigma_2 > \theta_2 > \sigma_4 > \sigma_5 > \theta_5 > \sigma_7 > \theta_7 > \theta_6 > (-\sigma_6) > \theta_4 > \theta_3 > (-\sigma_3) > \theta_1.$$ 
\end{example}

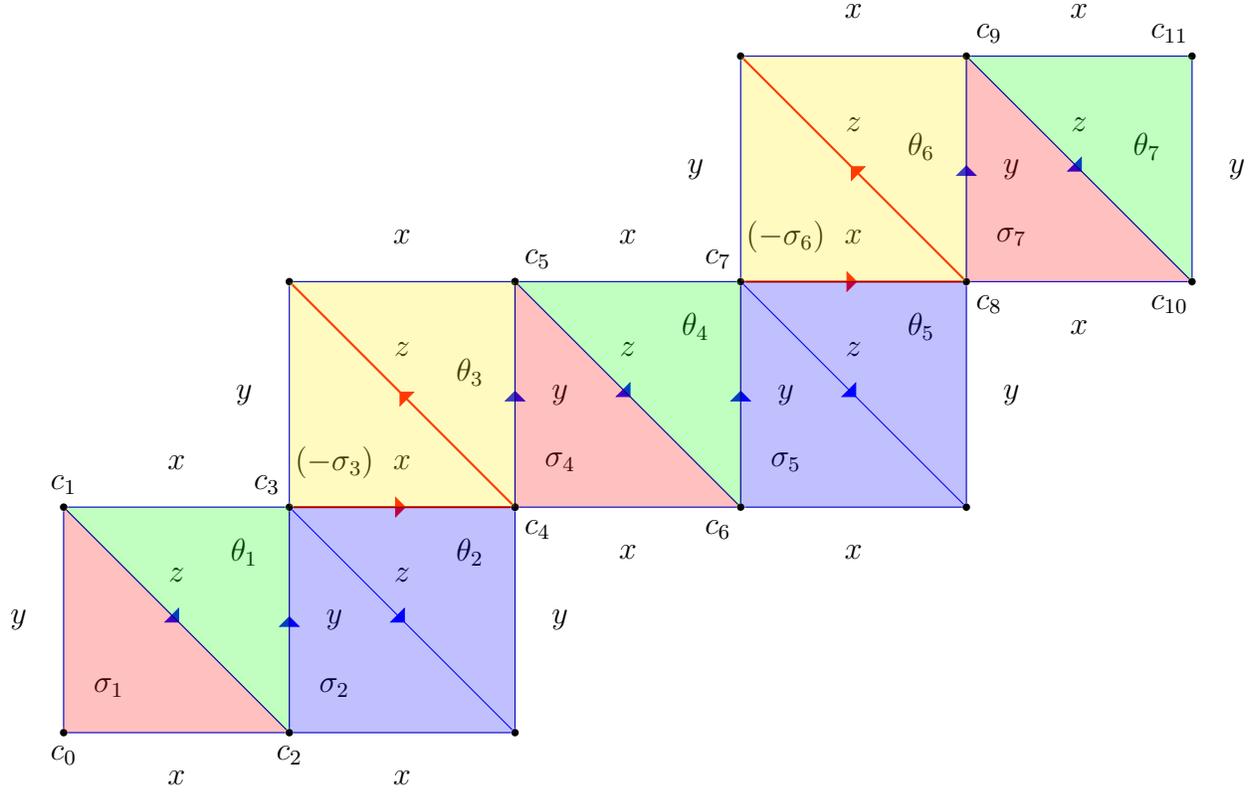
\begin{figure}
\begin{tikzpicture}[scale=0.6,every node/.style={sloped,allow upside down}]

    \node [circle,fill=black,inner sep=1pt] (0) at (-7, 3) {};
    \node [circle,fill=black,inner sep=1pt] (1) at (-7, -2) {};
    \node [circle,fill=black,inner sep=1pt] (2) at (-2, 3) {};
    \node [circle,fill=black,inner sep=1pt] (3) at (-2, -2) {};
    \node [] (4) at (-6, -1) {$\sigma_1$};
    \node [] (5) at (-3, 2) {$\theta_1$};
    
    \node [] (6) at (-4.5, -3) {$x$};
    \node [] (7) at (-4.5, 4) {$x$};
    \node [] (8) at (-8, 0.5) {$y$};
    \node [] (9) at (-1, 0.5) {$y$};
     \node [] (10) at (-4.5, 1.5) {$z$};               

    \draw [style=blue] (3) to (1);
    \draw [style=blue] (2) to (0);
    \draw [style=blue] (0) to node {\midarrow} (3);
    \draw [style=blue] (1) to  (0);
    \draw [style=blue] (3) to node {\midarrow}  (2);

    \node [circle,fill=black,inner sep=1pt] (11) at (3, 3) {};
    \node [circle,fill=black,inner sep=1pt] (12) at (3, -2) {};
    \node [] (13) at (-1, -1) {$\sigma_2$};
    \node [] (14) at (2, 2) {$\theta_2$};
    
    \node [] (15) at (0.5, -3) {$x$};
    \node [] (16) at (0.5, 4) {$x$};
    \node [] (18) at (4, 0.5) {$y$};
     \node [] (19) at (0.5, 1.5) {$z$};               

    \draw [style=blue] (12) to  (3);
    \draw [style=blue] (12) to (11);
    \draw [style=blue] (2) to node {\midarrow} (12);
    \draw [color=red,thick] (2) to node {\midarrow} (11);      
    
    \node [circle,fill=black,inner sep=1pt] (20) at (3, 8) {};
    \node [circle,fill=black,inner sep=1pt] (21) at (-2, 8) {};
    \node [] (22) at (-1, 4) {$(-\sigma_3)$};
    \node [] (23) at (2, 6) {$\theta_3$};
    
    \node [] (25) at (0.5, 9) {$x$};
    \node [] (26) at (-3, 5.5) {$y$};
    \node [] (27) at (4, 5.5) {$y$};
     \node [] (28) at (0.5, 6.5) {$z$};               

    \draw [style=blue] (20) to (21);
    \draw [style=blue] (11) to  (20);    
    \draw [style=blue] (2) to  (21);    
    \draw [color=red,thick] (11) to node {\midarrow} (21);    
    
    \node [circle,fill=black,inner sep=1pt] (30) at (3, 8) {};
    \node [circle,fill=black,inner sep=1pt] (31) at (3, 3) {};
    \node [circle,fill=black,inner sep=1pt] (32) at (8, 8) {};
    \node [circle,fill=black,inner sep=1pt] (33) at (8, 3) {};
    \node [] (34) at (4, 4) {$\sigma_4$};
    \node [] (35) at (7, 7) {$\theta_4$};
    
    \node [] (36) at (5.5, 2) {$x$};
    \node [] (37) at (5.5, 9) {$x$};
    \node [] (39) at (9, 5.5) {$y$};
     \node [] (40) at (5.5, 6.5) {$z$};               

    \draw [style=blue] (33) to  (31);
    \draw [style=blue] (32) to  (30);
    \draw [style=blue] (30) to node {\midarrow} (33);
    \draw [style=blue] (31) to node {\midarrow}  (30);
    \draw [style=blue] (33) to node {\midarrow}  (32);

    \node [circle,fill=black,inner sep=1pt] (41) at (13, 8) {};
    \node [circle,fill=black,inner sep=1pt] (42) at (13, 3) {};
    \node [] (43) at (9, 4) {$\sigma_5$};
    \node [] (44) at (12, 7) {$\theta_5$};
    
    \node [] (45) at (10.5, 2) {$x$};
    \node [] (46) at (10.5, 9) {$x$};
    \node [] (48) at (14, 5.5) {$y$};
     \node [] (49) at (10.5, 6.5) {$z$};               

    \draw [style=blue] (42) to  (33);
    \draw [style=blue] (42) to (41);
    \draw [style=blue] (32) to node {\midarrow} (42);
    \draw [color=red,thick] (32) to node {\midarrow} (41);      
    
    \node [circle,fill=black,inner sep=1pt] (50) at (13, 13) {};
    \node [circle,fill=black,inner sep=1pt] (51) at (8, 13) {};
    \node [] (52) at (9, 9) {$(-\sigma_6)$};
    \node [] (53) at (12, 11) {$\theta_6$};
    
    \node [] (55) at (10.5, 14) {$x$};
    \node [] (56) at (7, 10.5) {$y$};
    \node [] (57) at (14, 10.5) {$y$};
     \node [] (58) at (10.5, 11.5) {$z$};               

    \draw [style=blue] (50) to  (51);
    \draw [style=blue] (41) to  node {\midarrow}  (50);    
    \draw [style=blue] (32) to   (51);    
    \draw [color=red,thick] (41) to node {\midarrow} (51);              
          
    \node [circle,fill=black,inner sep=1pt] (60) at (18, 8) {};
    \node [circle,fill=black,inner sep=1pt] (61) at (18, 13) {};
    \node [] (62) at (14, 9) {$\sigma_7$};
    \node [] (63) at (17, 11) {$\theta_7$};
    
    \node [] (64) at (15.5, 7) {$x$};
    \node [] (65) at (15.5, 14) {$x$};
    \node [] (67) at (19, 10.5) {$y$};
     \node [] (68) at (15.5, 11.5) {$z$};               

    \draw [style=blue] (60) to  (61);
    \draw [style=blue] (50) to node {\midarrow} (60);    
    \draw [style=blue] (61) to  (50);  
    \draw [style=blue] (60) to (41);        
          
     \node [] at (-7, -2.5) {$c_0$};
     \node [] at (-7, 3.5) {$c_1$};
     \node [] at (-2, -2.5) {$c_2$};
     \node [] at (-2.5, 3.5) {$c_3$};
     \node [] at (3.5, 2.5) {$c_4$};
     \node [] at (3.5, 8.5) {$c_5$};
     \node [] at (7.5, 2.5) {$c_6$};
     \node [] at (7.5, 8.5) {$c_7$};
     \node [] at (13.5, 7.5) {$c_8$};
     \node [] at (13.5, 13.5) {$c_9$};
     \node [] at (17.5, 7.5) {$c_{10}$};
     \node [] at (17.5, 13.5) {$c_{11}$};
                    
     \draw[fill=red, nearly transparent] (-7,-2) -- (-2,-2) -- (-7,3) -- cycle;
     \draw[fill=green, nearly transparent] (-7,3) -- (-2,-2) -- (-2,3) -- cycle;
     \draw[fill=blue, nearly transparent] (-2,3) -- (-2,-2) -- (3,-2) -- (3,3) -- cycle;
     \draw[fill=yellow, nearly transparent] (3,3) -- (-2,3) -- (-2,8) -- (3,8) -- cycle;

     \draw[fill=red, nearly transparent] (3,3) -- (8,3) -- (3,8) -- cycle;
     \draw[fill=green, nearly transparent] (3,8) -- (8,3) -- (8,8) -- cycle;
     \draw[fill=blue, nearly transparent] (8,8) -- (8,3) -- (13,3) -- (13,8) -- cycle;
     \draw[fill=yellow, nearly transparent] (13,8) -- (8,8) -- (8,13) -- (13,13) -- cycle;

     \draw[fill=red, nearly transparent] (13,8) -- (18,8) -- (13,13) -- cycle;
     \draw[fill=green, nearly transparent] (13,13) -- (18,8) -- (18,13) -- cycle;
                     
\end{tikzpicture}
\caption{Default Orientation for the polygon $T_\gamma$ cut out by the arc $\gamma$ of slope $3/5$; with labels of fan centers $c_0, c_1, \dots , c_{11}$ and shading of fan segments.}
\label{fig:default-3-5}
\end{figure}

As a consequence of \Cref{lem:pos_ordering}, we get the following result.

\begin{corollary}
\label{cor:mult_pairs}
Given polygon $T_\gamma$, and using negations to get from the cyclic orientation to the default orientation and then the positive ordering, we recover the following multiplications involving pairs of $\mu$-invariants.

Firstly, $\theta_j \theta_k = \sigma_j \sigma_k = 0$ for any indices $j$ and $k$ since after we identify all $\theta_j$'s to each other and all $\sigma_k$'s to each other, the resulting multiplicative expressions square to zero.

Secondly, define $\Sigma_x$ (resp. $\Theta_x$) to be the set of indices $i$ (resp. $(i-1)$) such that $i$ is the position of a non-initial occurrence of $1$ in the upper Christoffel word.  Let $\Sigma_y$ and $\Theta_y$ denote the complements $( \Sigma_x)^c = \{1,2,\dots, p+q-1\} \setminus \Sigma_x$ and $( \Theta_x)^c = \{1,2,\dots, p+q-1\} \setminus \Theta_x$, respectively.

Then for $1 \leq j , k \leq p+q-1$, the product $\sigma_j \theta_k$ is positive in this ordering if and only if one of the following hold:

1) $j \in \Sigma_y$ and $k \in \Theta_y$.

2) $j \in \Sigma_y$, $k \in \Theta_x$ and $j \leq k$.

3) $j \in \Sigma_x$ and $k \in \Theta_x$.

4) $j \in \Sigma_x$, $k \in \Theta_y$ and $j \leq k$.

Otherwise, this product contributes a negative sign.  In particular, we have negative signs for $\sigma_j \theta_k$ for 
 $j \in \Sigma_y, k \in \Theta_x$ and $j > k$ or $j \in \Sigma_x, k \in \Theta_y$ and $j > k$. 
\end{corollary}

\begin{proof}
Using Lemma \ref{lem:pos_ordering} and letting $~\Sha~$ denote a shuffle or interspersing of two subsets, the positive ordering  can be summarized as $(\Sigma_y ~\Sha~ \Theta_x) >  (\Theta_y ~\Sha~ -\Sigma_x)$.

In more detail, the first portion of the positive ordering associated to the default orientation is the union of $\Sigma_y$ and $\Theta_x$, with each subset written with respect to subscripts in increasing order and so that the elements of $\Theta_x$ fit into the gaps left due to the elements of $\Sigma_x = (\Sigma_y)^c$.  Note in particular that there is a bijection between $\Sigma_x$ and $\Theta_x$ by $i \to (i-1)$.  The second portion of the positive ordering is then the union of $\Theta_y$ and $-\Sigma_x$ with respect to subscripts in decreasing order, analogously interspersed, and with negative signs multiplying elements of $\Sigma_x$.

Consequently, any element of $\Sigma_y$ precedes any element of $\Theta_y$, and any element of $\Theta_x$ precedes an element of $-\Sigma_x$.  However, the latter comes with a negative sign and so we reverse the order to account for this.  We thus have properties (1) and (3) as above.  We similarly obtain properties (2) and (4) but due to the nature of the shuffles, as indicated by $\Sha$, noting that the relative order of subscripts $j$ and $k$ matter in these two cases.
\end{proof}

\begin{remark}
\label{rem:tilesXY}
Using the notation of Corollary \ref{cor:mult_pairs}, for $2 \leq j \leq p+q-1$, $j \in \Sigma_x$ (resp. $\Sigma_y$) if and only if $\sigma_j$ is the upper right entry of a tile of the snake graph whose diagonal is labeled as $X$ (resp. $Y$).
Furthermore, for $1 \leq k \leq p+q-2$ $k \in \Theta_x$ (resp. $\Theta_y$) if and only if $\theta_j$ is the lower left entry of a tile of the snake graph whose diagonal is labeled as $X$ (resp. $Y$).  

By convention, we say that $1 \in \Sigma_y$ and $(p+q-1) \in \Theta_y$ although since $1 = j \leq k$ for all $k$ (and $j \leq k = p+q-1$ for all $j$), we are guaranteed that $\sigma_j \theta_k$ projects down to $ + \sigma\theta$ in such cases regardless of whether we say $1 \in \Sigma_y$ or in $\Sigma_x$ (and analogous if we say $(p+q-1) \in \Theta_y$ or in $\Theta_x$.

Finally, the remaining tiles have diagonals labeled with $Z$'s and $\sigma_i$ in the lower left and $\theta_i$ in the upper right, where $i$ runs from $1$ to $(p+q-1)$ as we run through the snake graph in order from beginning to end.
\end{remark}

\begin{example}
Continuing with our running example where arc $\gamma = \gamma_{3/5}$ with default orientation 
$$\sigma_1 > \sigma_2 > \theta_2 > \sigma_4 > \sigma_5 > \theta_5 > \sigma_7 > \theta_7 > \theta_6 > (-\sigma_6) > \theta_4 > \theta_3 > (-\sigma_3) > \theta_1,$$
we have $\Sigma_x = \{3,6\}$, $\Sigma_y = \{1,2,4,5,7\}$, $\Theta_x = \{2,5\}$, and $\Theta_y = \{1,3,4,6,7\}$.  We get positive products

$\sigma_1 \theta_k$ for $1 \leq k \leq 7$, ~ $\sigma_2 \theta_k$ for $1 \leq k \leq 7$, ~ $\sigma_3 \theta_k$ for $2 \leq k \leq 7$, ~ $\sigma_4 \theta_k$ for $k \in \{1,3,4,5,6,7\}$, 

$\sigma_5 \theta_k$ for $k \in \{1,3,4,5,6,7\}$, ~~~~~~~~~~~~ $\sigma_6 \theta_k$ for $k \in \{2,5,6,7\}$, ~ $\sigma_7 \theta_k$ for $k \in \{1,3,4,6,7\}$.

Hence, $\sigma_3 \theta_1$, ~ $\sigma_4 \theta_2$, ~ $\sigma_5 \theta_2$, ~ $\sigma_6 \theta_1$, ~ $\sigma_6 \theta_3$, ~ $\sigma_6 \theta_4$, ~ $\sigma_7 \theta_2$, ~ $\sigma_7\theta_5$ all equate to $-\sigma\theta$ instead.

See the snake graph of Figure \ref{fig:snake-3-5} for comparison.
\end{example}

We now put all of this together to prove our main theorem.

\begin{proof} [Proof of \Cref{thm:supermarkov}]

We use \cite[Theorem 5.2]{moz22b} which complements Theorems \ref{thm:12-entry} and \ref{thm:generic} from above to give a combinatorial formula (in terms of double dimer covers) for the super $\lambda$-length of $\gamma = \gamma_{p/q}$ relative to the default orientation of the polygon $T_\gamma$.  Lemma \ref{lem:snake_word} ensures that the snake graph associated to $T_\gamma$ indeed has the same pattern of tiles as in the ordinary Markov case of \cite{propp05}, and Lemma \ref{lem:negations} describes how to insert negative signs into certain $\mu$-invariants of $T_\gamma$ to stay in the equivalence class of the cyclic orientation on the triangulated once-punctured torus in the fundamental domain.  Lemma \ref{lem:pos_ordering} and Corollary \ref{cor:mult_pairs} then records which products of the form $\sigma_j \theta_k$ project down to $\sigma \theta$ versus its negation $-\sigma \theta$ with respect to the positive ordering.

We use Remark \ref{rem:tilesXY} to conclude that the only cycles that survive this enumeration as $+ \sigma \theta$ are those corresponding to cycles from tile $X$ to tile $X$, tile $Y$ to tile $Y$, or tile $Z$ to tile $Z$.  Similarly, the only cycles that survive as $-\sigma \theta$ are those from tile $X$ to tile $Y$ or from tile $Y$ to tile $X$. 

In particular, going from tile $Z$ to tile $Z$ has $\sigma_j$ to $\theta_k$ where $j \leq k$.  Going from $X$ to $X$ has $\theta_k$ to $\sigma_j$ where $k < j$ but $k \in \Theta_x$ and $j \in \Sigma_x$.  Thus we negate twice to see $\theta_k\sigma_j$ projects down to $ - \theta\sigma = +\sigma\theta$.  Similarly, going from $Y$ to $Y$ has $\theta_k$ to $\sigma_j$ where $k < j$ but $k \in \Theta_y$ and $j \in \Sigma_y$, yielding $+\sigma\theta$ after two negations again.  

If we instead go from $X$ to $Y$, we obtain $\theta_k \sigma_j$ where $k < j$ and $k \in \Theta_x$ but $j \in \Sigma_y$.  We thus can only negate once in this case and project down to $+\theta\sigma = - \sigma\theta$ in this case.  Analogously, going from $Y$ to $X$ yields $\theta_k \sigma_j$ where $k < j$, $k \in \Theta_y$ and $j \in \Sigma_x$, and we project down to $+\theta\sigma = - \sigma\theta$.

Finally, going from tile $Z$ to $X$ or $Y$, or the reverse, will yield $\sigma_j \sigma_k = 0$ or $\theta_j \theta_k = 0$, respectively.

Any double dimer cover that includes either more than one cycle or a cycle of even length will lead to a $\sigma^2=0$ or $\theta^2=0$ and hence does not contribute either.  Also the double dimer covers with no cycles are in bijection with ordinary perfect matchings with doubled edges used, and hence were already included in the enumeration counted by the ordinary Markov number $M_{p/q}$.
\end{proof}

\begin{figure}
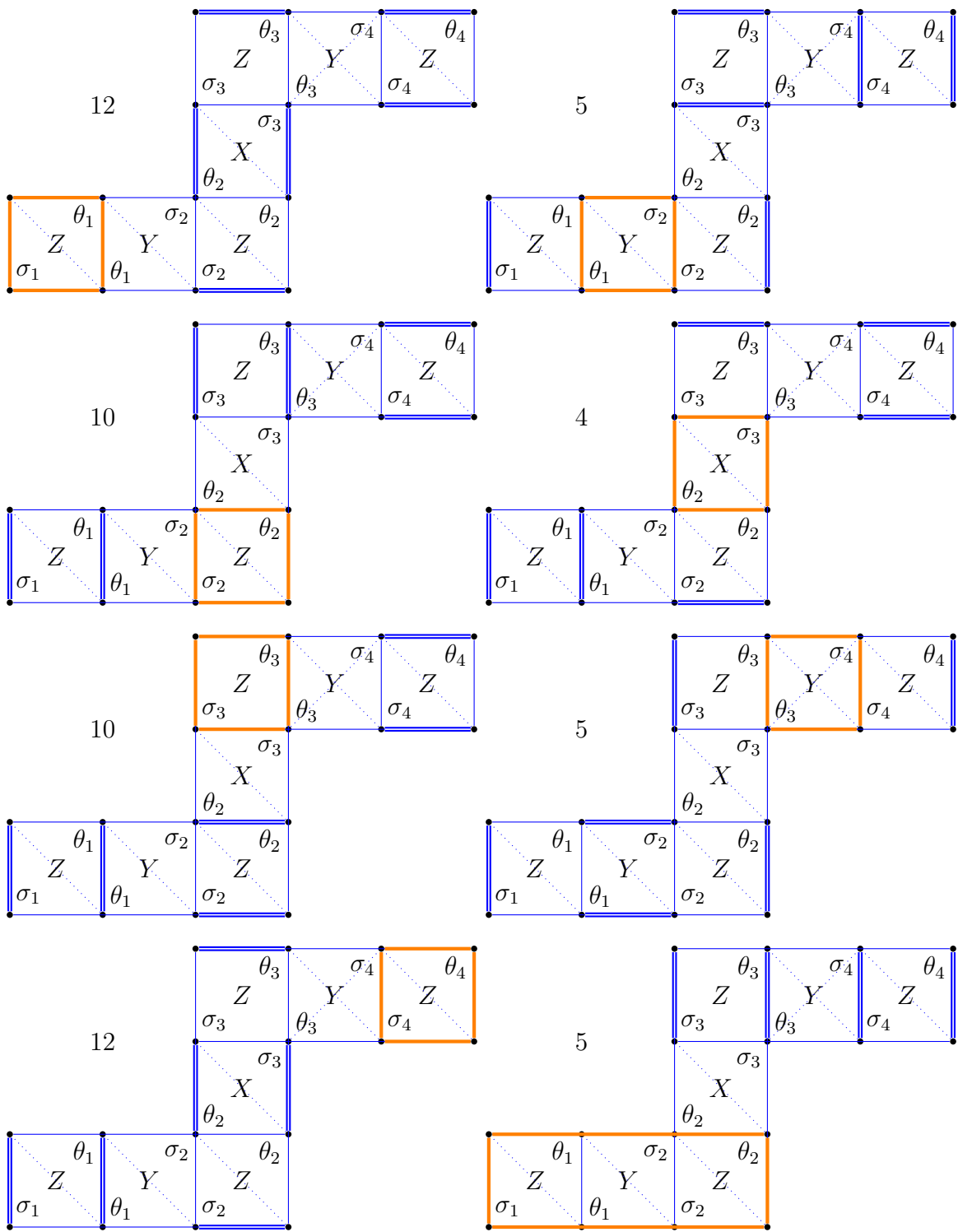



\caption{Examples of double dimer covers contributing positively to $SM_{2/3} = M_{2/3} + \hat{M}_{2/3} = 29 + 74 \sigma\theta = 29 + (78 - 4)\sigma\theta$; see also Figure \ref{fig:DDexamplesPos2}.}
\label{fig:DDexamplesPos}
\end{figure}

\begin{figure}
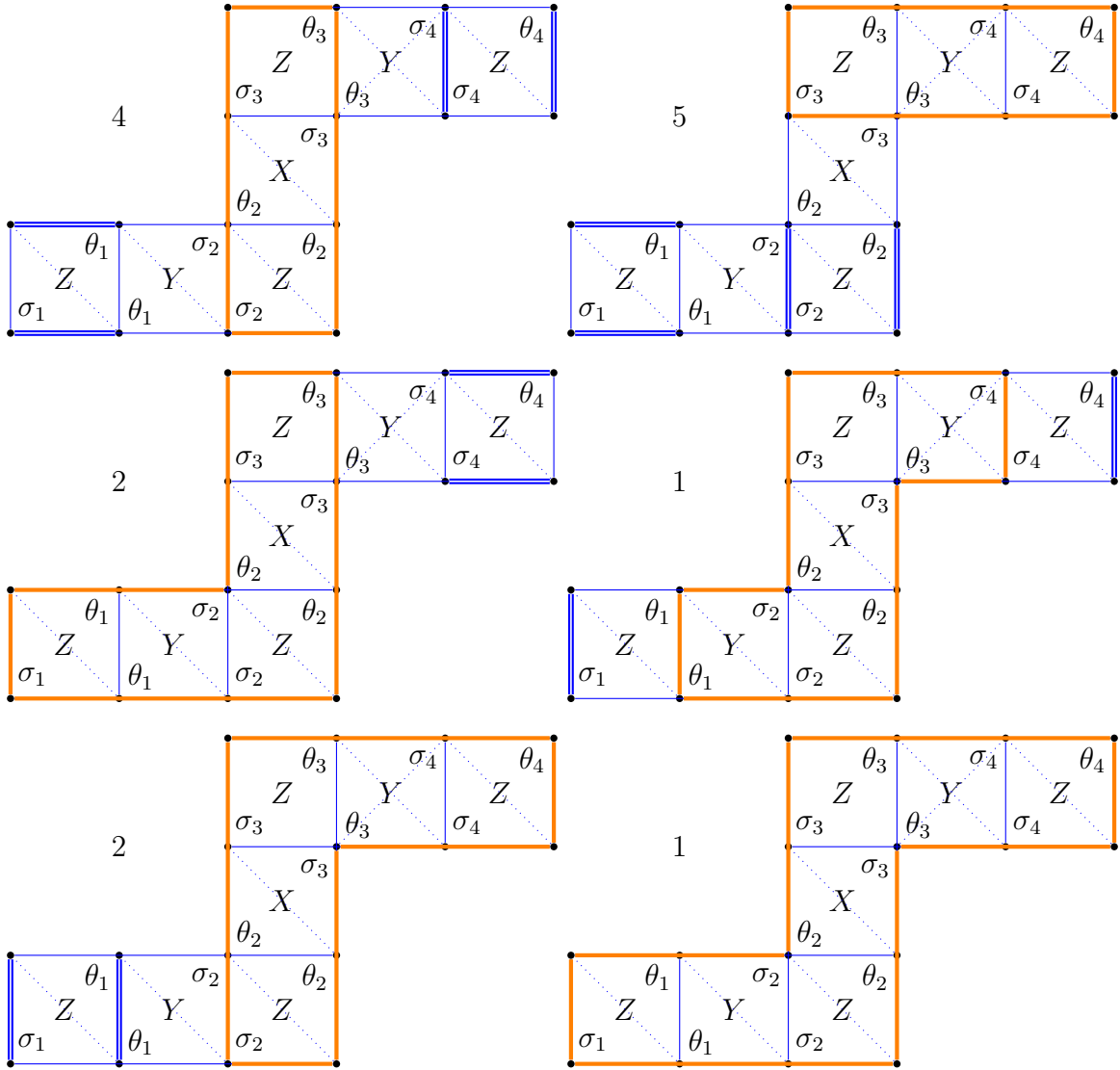


\vspace{1em}

%

\caption{Continuation of Figure \ref{fig:DDexamplesPos}.  In particular, the enumeration of $78 \sigma\theta$ decomposes as  $(12 + 5 + 10 + 4 + 10 + 5 + 12 + 5 + 4 + 5 + 2 + 1 + 2 + 1)\sigma\theta$.}
\label{fig:DDexamplesPos2}
\end{figure}

\begin{figure}
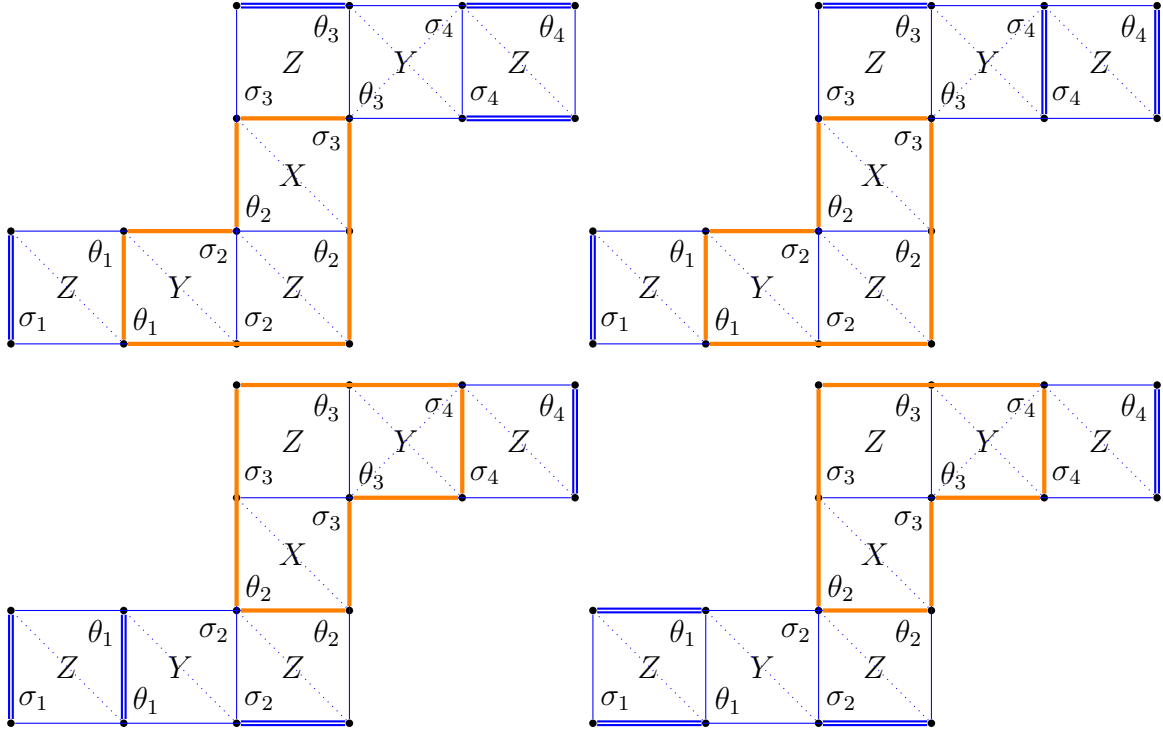


%

\caption{The four double dimer covers contributing negatively to $SM_{2/3} = M_{2/3} + \hat{M}_{2/3}\sigma\theta = 29 + 74 \sigma\theta = 29 + (78 - 4)\sigma\theta$.}
\label{fig:DDexamplesNeg}
\end{figure}

\begin{example}

See Figures \ref{fig:DDexamplesPos}, \ref{fig:DDexamplesPos2}, and \ref{fig:DDexamplesNeg} for double dimer covers counted by $\hat{M}_{2/3}$, i.e. the soul of $SM_{2/3} = M_{2/3} + \hat{M}_{2/3}\sigma\theta = 29 + 74 \sigma\theta = 29 + (78 - 4) \sigma\theta$.  In Figures \ref{fig:DDexamplesPos} and \ref{fig:DDexamplesPos2}, the numbers next to each representative signify the number of double dimer covers containing this particular cycle of odd length in the snake graph, while Figure \ref{fig:DDexamplesNeg} shows all four double dimer covers corresponding to the $-4\sigma\theta$.

Double dimer covers consisting entirely of doubled edges biject with ordinary perfect matchings (or dimer covers) of the snake graph $G_{2/3}$ which is counted by the ordinary Markov number $29$.  Cycles of single edges are indicated in orange and contribute $\pm \sigma\theta$ accordingly.
\end{example}

\begin{remark}
If we let the $\lambda$-lengths of the three arcs $x$, $y$, and $z$ in the initial triangulation be variables instead of setting them to be $1$, the argument still goes through in the same way and we get formulas in $\mathbb{Z}[x^{\pm 1/2}, y^{\pm 1/2}, z^{\pm 1/2}]\langle \sigma\theta \rangle$ instead.
\end{remark}

\section{Open questions related to Super Markov numbers}
\label{sec:open}

Despite the signed generating functions appearing in the proof and the provided formulas for super Markov numbers, as well as their weighted version, 
in practice, all examples we have computed resulted in expressions with positive signs only.  This leads us to the following conjecture and question.

\begin{conj} \label{conj:no-signs}
Assuming the spin-structure on the once-punctured torus induced from the cyclic orientation of the two triangles in the fundamental domain (clockwise on one of them, counter-clockwise on the other), the formulas for super $\lambda$-lengths of arcs in the associated decorated super Teichm\"uller space can all be expressed using only positive coefficients.
\end{conj}

\begin{question}
Assuming the correctness of Conjecture \ref{conj:no-signs}, is there an adaptation of the above combinatorial interpretation, \Cref{def:markov_soul}, which uses a signed enumeration of double dimer covers, requiring only a positive expansion formula?  In particular, can one further restrict which double dimer covers contribute to the partition function (similar to how Lindstrom-Gessel-Viennot often allows one to use the signed expansion formula of a matrix determinant as the positive generating function for non-intersecting lattice paths) therefore manifesting a solely positive enumeration with no cancellations of signed terms needed?
\end{question}

As another direction, we note that V. Osienko \cite{ovsienko2023shadow} considers a shadow sequence related to Markov numbers.  However, his values differ from the ones we obtain as coefficients of the $\varepsilon = \theta\sigma$ in $SM_{p/q}$'s.  

\begin{question}
What is the relationship between our family of Super Markov numbers and the shadow sequences, and shadow Markov equation, that Ovsienko obtains by deforming the relevant difference equations, see Table \ref{table:Ovsienko}?
\end{question}

\begin{table}
\begin{tabular}{cc}
Ovsienko Shadow Markov Number & Super Markov Number \\
$2+ \textcolor{red}{4\epsilon}$ & $SM_{1/1} = 2 + \textcolor{blue}{\sigma\theta}$ \\
$5+ \textcolor{red}{13\epsilon}$ & $SM_{1/2} = 5 + \textcolor{blue}{6 \sigma\theta}$ \\
$13+ \textcolor{red}{40\epsilon}$ & $SM_{1/3} = 13 + \textcolor{blue}{26 \sigma\theta}$ \\
$34+ \textcolor{red}{120\epsilon}$ & $SM_{1/4} = 34 + \textcolor{blue}{97 \sigma\theta}$\\
$89+ \textcolor{red}{354\epsilon}$ & $SM_{1/5} = 89 + \textcolor{blue}{332 \sigma\theta}$\\
$29+ \textcolor{red}{117\epsilon}$ & $SM_{2/3} = 29 + \textcolor{blue}{74 \sigma\theta}$\\
$194+ \textcolor{red}{976\epsilon}$ & $SM_{2/5} = 194 + \textcolor{blue}{801 \sigma\theta}$\\
$169+ \textcolor{red}{921\epsilon}$ & $SM_{3/4} = 169 + \textcolor{blue}{688 \sigma\theta}$\\
$433+ \textcolor{red}{2592\epsilon}$ & $SM_{3/5} = 433 + \textcolor{blue}{2032 \sigma\theta}$
\end{tabular}
\caption{Ovsienko's Shadow Markov Numbers versus Super Markov Number.}
\label{table:Ovsienko}
\end{table}

We note that in the case of Super Fibonacci numbers, which coincide with Super Markov numbers of the form $SM_{1/q}$, Ovsienko similarly studied a related, but different, shadow sequence of Fibonacci numbers.  However, in this case, the relationship between his work and ours has been clarified, as we indicated in \cite[Footnote 8]{moz22}.

\begin{table}
\begin{tabular}{cccccccccccccc}
$p / q$ & 2 & 3 & 4 & 5 & 6 & 7 & 8 & 9 & 10 & 11 & 12 & 13 \\
1 & 6 & 26 & 97 & 332 & 	1076 & 3361 & 10226 & 30510 & 89665 & 260376 & 748776 & 2136001 \\
2 & ~ & 74 & ~ & 801 & ~ & 7714 & ~& 68718 & ~ & 581249 & ~ & 4743966 \\
3 & ~ & ~ & 668 & 2032 & ~ & 18192 & 53724 & ~ & 454436 & 1306168 & ~ & 10567054 \\
4 & ~ & ~ & ~ & 5284 & ~ & 44257 & ~ & 364778 & ~ & 2971498 & ~ & 23705633 \\
5 & ~ & ~ & ~ & ~ & 38914 & 110138 & 310113 & 868730 & ~ & 6837825 & 19232192 & 53829172 \\
6 & ~ & ~ & ~ & ~ & ~ & 274126 & ~ & ~ & ~ & 16076512 & ~ & 122875168 \\
\end{tabular}
\caption{Coefficients of $\theta\sigma$ in $SM_{p/q}$, for small and relatively prime values of $p$ and $q$, where columns are $q$ and rows are $p$.}
\label{table:MoreSM}
\end{table}

Continuing from Table \ref{table:Ovsienko}, coefficients of $\theta\sigma$ in additional Super Markov Numbers grow as in Table \ref{table:MoreSM}.  This data motivates the following question, which the author thanks E. Banaian for bringing to his attention:
\begin{question}
Is the analogue of Aigner's mononicity conjecture \cite{aigner2015markov} true for the $\hat{M}_{p/q}$'s, i.e. the  souls of Super Markov numbers?  In other words, if we let $p_1, p_2 < q_1, q_2$, with $\gcd(p_i, q_i) = 1$ in all cases, do we have the three inequalities
\begin{eqnarray} \hat{M}_{p_1/q_1} &<&  \hat{M}_{p_2/q_1} \mathrm{~~if~~} p_1 < p_2? \\
\hat{M}_{p_1/q_1} &<&  \hat{M}_{p_1/q_2} \mathrm{~~if~~} q_1 < q_2? \\
\hat{M}_{p_1/q_1} &<&  \hat{M}_{p_2/q_2} \mathrm{~~if~~} p_1 + q_1 = p_2 = q_2 \mathrm{~and~} q_1 < q_2?
\end{eqnarray}
\end{question}

Note that Rabideau-Schiffler \cite{rabideau2020continued} proved the second of these inequalities for ordinary Markov numbers using cluster algebras and snake graphs, and McShane \cite{mcshane2021convexity} used convexity to prove all three of these inequalities for ordinary Markov numbers.

\section{Applications to Super Annuli}
\label{sec:Annuli}

As a generalization of the case of Super Fibonacci numbers studied in \cite[Sec. 11]{moz22} and \cite[Sec. 6]{moz22b}, instead of considering the decorated super Teichm\"uller space associated to an annulus with one marked point on each boundary, we can instead consider the decorated super Teichm\"uller space associated to an annulus with $p$ marked points on one boundary and $q$ marked points on the other.

We begin by focusing our investigation on marked annuli where $p=1$ (without loss of generality assume this is the inner boundary) and such that the initial triangulation is assumed to consist entirely of bridging arcs with one arc emanating from $(q-1)$ of the marked points on the outer boundary and two arcs emanating from the $q$th marked point on the outer boundary.  We label the arcs from $1$ to $(q+1)$ in clockwise order (with respect to the inner boundary) beginning with one of the two arcs bridging the unique marked point on the inner boundary to the $q$th marked point, and ending with the other arc bridging the unique marked point on the inner boundary to the $q$th marked point.  Let $x_1, x_2, \dots, x_{q+1}$ denote the super $\lambda$-lengths of these initial arcs accordingly.  For $1 \leq i \leq q$, let  $\theta_i$ denote the $\mu$-invariant defined by the triangle defined by the arcs $x_i$, $x_{i+1}$, as well as part of the outer boundary, and let $\sigma_{q+1}$ denote the $\mu$-invariant defined by the triangle defined by the arcs $x_{q+1}$, $x_1$, and the inner boundary.  We also choose a spin structure so that arc $1$ is oriented towards the inner boundary point while the arcs, $2$, $3$, $\dots$, $q$ are oriented away from the marked point on the inner boundary.  This orientation ensures that when we mutate arc $1$, then $\theta_1$ lies to its left and $\sigma_{q+1}$ lies to its right, mimicking the configuration of Figure \ref{fig:super_ptolemy} (Left).  See Figure \ref{fig:annulus} for an example when $q=2$.

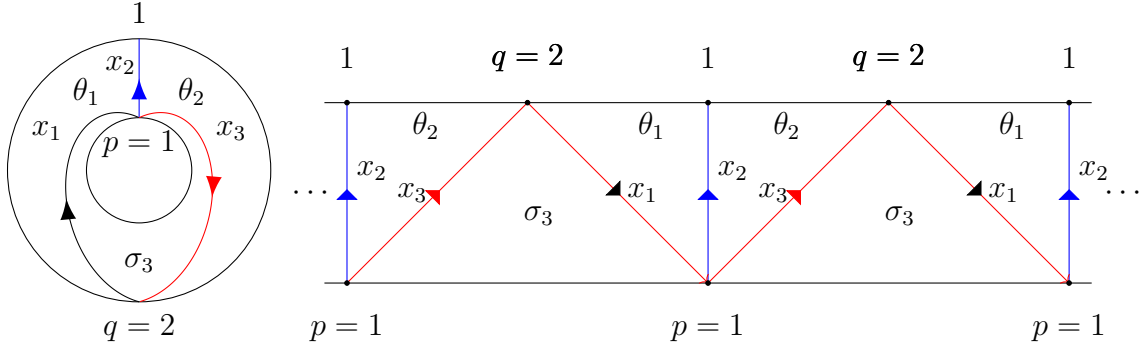
\begin{figure}
\begin{tikzpicture}[scale=0.7]
    \node () at (-1,1.5) {$\theta_1$};
    \node () at (1,1.5) {$\theta_2$};    
    \node () at (0,-1.75) {$\sigma_3$};

    \draw (0,0) circle (1);
    \draw (0,0) circle (2.5);

    \draw [-->-, blue] (0,1) -- (0,2.5);

 \draw[-->-, color=red, domain=1.0:0.5, samples=100] plot ({(-2.5+1.5*\x)*cos(2*pi*\x r - pi/2 r}, {(-1)*(4.0-3.0*\x)*sin(2*pi*\x r - pi/2 r)});

    \draw[-->-, color=black, domain=0.5:1, samples=100] plot ({(2.5-1.5*\x)*cos(2*pi*\x r - pi/2 r}, {(-1)*(4.0-3.0*\x)*sin(2*pi*\x r - pi/2 r)});

\node () at (-1.75,0.75) {$x_1$};
\node () at (-0.35,2) {$x_2$};
\node () at (1.75,0.75) {$x_3$};

\node () at (0, 0.5) {$p=1$};
\node () at (0, 3.0) {$1$};
\node () at (0, -3.0) {$q=2$};

\end{tikzpicture}
\begin{tikzpicture}[scale=0.6,every node/.style={sloped,allow upside down}]
    \draw (0,-2.5) node {};

    \draw (-6.5,2)  -- (10.5,2);
    \draw (-6.5,-2) -- (10.5,-2);

    \foreach \x in {-6, 2}
        \draw (\x+4.25, -0.5) node {$\sigma_3$};

    \foreach \x in {-2, 6}
        \draw (\x-2.25, +1.5) node {$\theta_2$};

   \foreach \x in {2,10}
        \draw (\x-1.25, +1.5) node {$\theta_1$};

    \draw (-6.75, 0) node {$\cdots$};
    \draw (11.25,  0)  node {$\cdots$};

    \foreach \x in {-10, -2, 6}
        \draw [style=blue] (\x+4,-2) to node {\midarrow}  (\x+4,2);        
    \foreach \x in {-10, -2, 6}        
        \node () at (\x+4.55, 0.5) {$x_2$};        

    \foreach \x in {-2, 6}
        \draw [style=red] (\x,2) to node {\midarrow}  (\x+4,-2);
    \foreach \x in {-2, 6}        
        \node () at (\x+2.55, 0) {$x_1$};
                
    \foreach \x in {-6, 2}
        \draw [color=red] (\x,-2) -- node {\midarrow} (\x+4,2);
    \foreach \x in {-6, 2}        
        \node () at (\x+1.45, 0) {$x_3$};

    \foreach \x in {-6, 2, 10} {
        \draw [fill=black] (\x,-2) circle (0.05);
        \node () at (\x, -3) {$p=1$};
        \node () at (\x, 3) {$1$};
        \node () at (-2, 3) {$q=2$};
        \node () at (6, 3) {$q=2$};
        
    }
    \foreach \x in {-6, -2, 2, 6, 10} {        
        \draw [fill=black] (\x,2) circle (0.05);

    }

\end{tikzpicture}
\caption{(Left): Triangulation of an annulus when $q=2$. (Right): its universal cover.}
\label{fig:annulus}
\end{figure}
\begin{figure}
\begin{tikzpicture}[scale=0.7]
    \node () at (-1,1.5) {$\sigma_4$};
    \node () at (1,1.5) {$\theta_2$};    
    \node () at (0,-2.15) {$\theta_3$};

    \draw (0,0) circle (1);
    \draw (0,0) circle (2.5);

    \draw [-->-, blue] (0,2.5) -- (0,1);

 \draw[-->-, color=red, domain=1.0:0.5, samples=100] plot ({(-3.5+2.5*\x)*cos(2*pi*\x r - pi/2 r}, {(-1)*(4.0-3.0*\x)*sin(2*pi*\x r - pi/2 r)});

\begin{scope}[xscale= -1]
		\node [circle,fill=black,inner sep=1pt] (1) at (0, 1) {};
		\node [ ] (8) at (1.5, 1.5) {};
		\node [ ] (9) at (1.5, -0.75) {};
		\node [ ] (10) at (0, -1.75) {};
		\node [ ] (11) at (-1.25, -1.25) {};
		\node [ ] (13) at (-1.25, 0.5) {};
		\node [circle,fill=black,inner sep=1pt] (17) at (0, 2.5) {};

		\draw [color=black, in=45, out=180] (1) to (13.center);
		\draw [color=black, in=135, out=-135, looseness=0.75] (13.center) to (11.center);
		\draw [color=black, in=-180, out=-45] (11.center) to (10.center);
		\draw [color=black, in=-120, out=0] (10.center) to (9.center);
		\draw [-->-, color=black, in=-60, out=60, looseness=0.75] (9.center) to (8.center);
		\draw [color=black, in=0, out=120, looseness=0.75] (8.center) to (17);
\end{scope}

\node () at (-2.15,0.5) {$x_4$};
\node () at (-0.35,2) {$x_2$};
\node () at (1.85,0.5) {$x_3$};

\node () at (0, 0.5) {$p=1$};
\node () at (0, 3.0) {$1$};
\node () at (0, -3.0) {$q=2$};

\end{tikzpicture}
\begin{tikzpicture}[scale=0.6,every node/.style={sloped,allow upside down}]
    \draw (0,-2.5) node {};

    \draw (-6.5,2)  -- (10.5,2);
    \draw (-6.5,-2) -- (10.5,-2);

    \foreach \x in {-6, 2}
        \draw (\x+1.75, 2-0.5) node {$\theta_2$};

    \foreach \x in {-2, 6}
        \draw (\x+1.25, -2+1.5) node {$\sigma_4$};

   \foreach \x in {2,10}
        \draw (\x-3.25, +1.5) node {$\theta_3$};

    \draw (-6.75, 0) node {$\cdots$};
    \draw (11.25,  0)  node {$\cdots$};

    \foreach \x in {-10, -2, 6}
        \draw [style=blue] (\x+4,2) to node {\midarrow}  (\x+4,-2);
    \foreach \x in {-10, -2, 6}        
        \node () at (\x+4.55, 0.5) {$x_2$};

    \foreach \x in {-2, 6}
        \draw [-->-,color=red] (\x-4,-2) to (\x,2);
    \foreach \x in {-2, 6}        
        \node () at (\x-2.15, 0.5) {$x_3$};
                
    \foreach \x in {2,10}
        \draw [color=black] (\x-8,-2) -- node {\midarrow} (\x,2);
    \foreach \x in {-6, 2}        
        \node () at (\x+3.75, 0.5) {$x_4$};

    \foreach \x in {-6, 2, 10} {
        \draw [fill=black] (\x,-2) circle (0.05);
        \node () at (\x, -3) {$p=1$};
        \node () at (\x, 3) {$1$};
        \node () at (-2, 3) {$q=2$};
        \node () at (6, 3) {$q=2$};
        
    }
    \foreach \x in {-6, -2, 2, 6, 10} {        
        \draw [fill=black] (\x,2) circle (0.05);
    }

\end{tikzpicture}
\caption{Mutation of the triangulation of the annulus from Figure \ref{fig:annulus}.}
\label{fig:annulus2}
\end{figure}
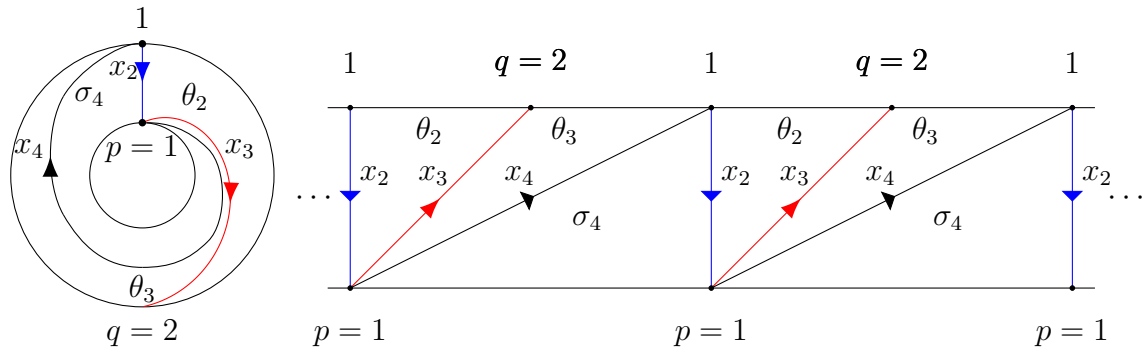

Mutating $1, 2, 3, \dots, (q+1)$ induces a Dehn twist of the annulus and repeatedly applying this mutating sequence yields a one-parameter sequence of super $\lambda$-lengths $x_{q+2}, x_{q+3}, \dots$.  Allowing mutations along the reverse of this sequence, i.e. $(q+1), q, \dots, 2, 1, \dots$, analogously yields $x_0, x_{-1}, \dots$, and thus we obtain super $\lambda$-lengths $x_n$, for all integers $n$, associated to all bridging arcs in this annulus.  We also construct $\mu$-invariants $\sigma_n$ for all integers $n$ by increasing (resp. decreasing) indices by one as we apply mutation (resp. backwards mutation), and invariants 
$\theta_n$ for all integers $n$ by increasing (resp. decreasing) indices by $q$ as we apply mutation (resp. backwards mutation).

Due to the reversal of orientation, see Figure \ref{fig:super_ptolemy} (Left), observe that mutating arc $1$ leads to the reversal of orientation of arc $2$.  Additionally, the newly created arc $(q+2)$ is oriented away from the inner boundary.  Consequently, by adding one to every index, we recover the same local configuration as the original triangulation (up to winding or shearing).  Iterating this process, the local oriented triangulation is also preserved under Dehn twists.

Applying the Super Ptolemy Relation yields the defining recurrences for this super-integrable system:
\begin{eqnarray}
\label{eq:x} x_n x_{n-q-1} &=& x_{n-1}x_{n-q} + 1 + \sqrt{x_{n-1}x_{n-q}} \sigma_{n-1} \theta_{n-q-1} \\
\label{eq:sig} \sigma_n &=& \left(\sigma_{n-1}\sqrt{x_{n-1}x_{n-q}} - \theta_{n-q-1}\right) / \sqrt{x_{n-q-1}x_n} \\
\label{eq:thet} \theta_{n-1} &=& \left(\theta_{n-q-1}\sqrt{x_{n-1}x_{n-q}} + \sigma_{n-1}\right) / \sqrt{x_{n-q-1}x_n}
\end{eqnarray}

If we then multiply Equations (\ref{eq:sig}) and (\ref{eq:thet}) together, we get 
$\sigma_n \theta_{n-1} = \frac{x_{n-1}x_{n-q} + 1}{x_n x_{n-q-1}} \sigma_{n-1} \theta_{n-q-1}$, which we can use 
$(\sigma_{n-1}\theta_{n-q-1})^2 =0$ to obtain the expression
$$\sigma_n \theta_{n-1} = \frac{x_{n-1}x_{n-q} + 1 + \sqrt{x_{n-1}x_{n-q}} \sigma_{n-1}\theta_{n-q-1}}{x_n x_{n-q-1}} \sigma_{n-1} \theta_{n-q-1} = \sigma_{n-1} \theta_{n-q-1}.$$
In particular, the product of $\mu$-invariants corresponding to the quadrilateral where we mutated is a constant of this mutation.

\begin{example}
For $q=1$ and if we initialize $x_1=x_2=1$ and
$\theta_1 = \theta$, $\sigma_2 = \sigma$, then we recover $\sigma_n \theta_{n-1} = \sigma_{n-1}\theta_{n-2}$ for all $n$.  Letting $\sigma\theta$ denote this constant value for $\sigma_n \theta_{n-1} $ (even as $n$ varies), we 
thereby recover $$x_n x_{n-2} = x_{n-1}^2 + 1 + x_{n-1} \sigma \theta$$ for all $n$ and get the Super Fibonacci numbers as studied in \cite[Sec 11]{moz22}.
\end{example}

Additionally, similar to the use of modified $\mu$-invariants in \cite[Appendix A]{moz21} or the $\Delta_{jk}^i$'s as defined in \cite[Definition 1.2]{moz22b}, we define
normalized $\mu$-invariants by multiplying each $\mu$-invariant by $\sqrt{\frac{c}{ab}}$ when the $\mu$-invariant is associated to the angle of the triangle $(a,b,c)$ opposite of the side of weight $c$.  Using the angles opposite boundary edges (of weight $1$), we get the following normalized $\mu$-invariants for all $n \in \mathbb{Z}$:
$$\widetilde{\sigma_n} = \sigma_n / \sqrt{x_n x_{n-q}} \mathrm{~~~and~~~} \widetilde{\theta_{n-1}} = \theta_{n-1} / \sqrt{x_n x_{n-1}}.$$
Using the notation of normalized $\mu$-invariants, Equations (\ref{eq:x})-(\ref{eq:thet}) can be rewritten as the recurrences 
\begin{eqnarray}
\label{eq:x2} x_n x_{n-q-1} &=& x_{n-1}x_{n-q} + 1 + x_{n-1}x_{n-q} x_{n-q-1} \widetilde{\sigma_{n-1}} \widetilde{\theta_{n-q-1}}\\
\nonumber &=& x_{n-1}x_{n-q}\left(1 + x_{n-q-1} \widetilde{\sigma_{n-1}} \widetilde{\theta_{n-q-1}} \right) + 1 \\
\label{eq:sig2} \widetilde{\sigma_n} &=& \frac{x_{n-1}}{x_{n}} \widetilde{\sigma_{n-1}} - \frac{1}{x_{n}}\widetilde{\theta_{n-q-1}} \\
\label{eq:thet2} \widetilde{\theta_{n-1}} &=& \frac{x_{n-q}}{x_{n}}  \widetilde{\theta_{n-q-1}} + \frac{1}{x_{n}}\widetilde{\sigma_{n-1}}
\end{eqnarray}

\begin{example}
If $q=2$ and we initalize $x_1=x_2=x_3=1$, then 
$\widetilde{\theta_1}=\theta_1$, $\widetilde{\theta_2}=\theta_2$, $\widetilde{\sigma_3}= \sigma_3$, and we obtain
\begin{eqnarray*}
x_4 &=&  2 ~~~ + \hspace{3em}  \sigma_3 \theta_1\\
\widetilde{\theta_3} &=& ~~\theta_1 + \hspace{3em}  \sigma_3 \\ 
\widetilde{\sigma_4} &=& - \theta_1 +\hspace{3em}  \sigma_3 \\ ~ \\
x_5 &=& 3 + \theta_2 \theta_1 + \sigma_3 \theta_1 + \sigma_3 \theta_2 \\
\widetilde{\theta_4} &=& \hspace{2em} - \theta_1 + 2 \theta_2 +  \sigma_3 - \sigma_3 \theta_2 \theta_1 \\ 
\widetilde{\sigma_5} &=& \hspace{2em}  - \theta_1 - \theta_2 +  \sigma_3  \\ ~ \\
x_6 &=& 7 + \theta_2 \theta_1 + 7 \sigma_3 \theta_1 + 3 \sigma_3 \theta_2 \\
\widetilde{\theta_5} &=& \hspace{1.5em} ~~ 2 \theta_1 - \theta_2 +  4 \sigma_3 + 2 \sigma_3 \theta_2 \theta_1 \\  
\widetilde{\sigma_6} &=& \hspace{1.5em} - 3 \theta_1 - 2 \theta_2 +  \sigma_3 - \sigma_3 \theta_2 \theta_1 \\  ~ \\
x_7 &=& 11 + 7 \theta_2 \theta_1 + 9 \sigma_3 \theta_1 + 9 \sigma_3 \theta_2 \\
x_8 &=& 26 + 10 \theta_2 \theta_1 + 40 \sigma_3 \theta_1 + 25 \sigma_3 \theta_2 \\
x_9 &=& 41 + 38 \theta_2 \theta_1 + 56 \sigma_3 \theta_1 + 56 \sigma_3 \theta_2 \\
x_{10} &=& 97 + 64 \theta_2 \theta_1 + 204 \sigma_3 \theta_1 + 148 \sigma_3 \theta_2  \\
x_{11} &=& 153 + 186 \theta_2 \theta_1 + 296 \sigma_3 \theta_1 + 296 \sigma_3 \theta_2 \\
x_{12} &=& 362 + 342 \theta_2 \theta_1 + 969 \sigma_3 \theta_1 + 760 \sigma_3 \theta_2 
\end{eqnarray*}
as the first several resulting expressions from applying this recurrence.  We will provide a combinatorial interpretation for the weighted versions of the super $\lambda$-lengths $x_n$ in terms of signed double dimer covers as our next result. 
\end{example}

\begin{prop}
Given an annulus with $p=1$ marked point on boundary and $q=2$ marked points on the other, the super $lambda$-lengths $x_n$ as defined above have the following combinatorial interpretation as a signed enumeration:

Consider the affine cluster algebra of type $\widetilde{A}_{1,2}$, which has an initial seed given by the quiver 

$\xymatrix{
1 \ar@/^1pc/[rr]  \ar@/_/[r] & 2 \ar@/_/[r] & 3
}$.  Let $G_n(1,2)$ denote the snake graph associated to the ordinary $\lambda$-length (equiv. cluster variable) $x_n$ for this cluster algebra, with tiles labeled as $1$, $2$, or $3$, following the construction of \cite{MSW_11}.  See Figure \ref{fig:Ann21}.

Then $x_n$ is the signed enumeration of double dimer covers on $G_n(1,2)$ as follows: 

Weight of all double dimer covers consisting entirely of doubled edges

+ \bigg(Weight of double dimer covers containing a single cycle whose leftmost tile is $1$ and whose rightmost tile is also $1$

+ Weight of double dimer covers containing a single cycle whose leftmost tile is $2$ and whose rightmost tile is $3$ \bigg) 
$\sigma_3 \theta_1$

+ \bigg(Weight of double dimer covers containing a single cycle whose leftmost tile is $2$ and whose rightmost tile is also $2$

- Weight of double dimer covers containing a single cycle whose leftmost tile is $3$ and whose rightmost tile is $1$ \bigg) 
$\theta_2\theta_1$

+ \bigg(Weight of double dimer covers containing a single cycle whose leftmost tile is $3$ and whose rightmost tile is also $3$

+ Weight of double dimer covers containing a single cycle whose leftmost tile is $1$ and whose rightmost tile is $2$ \bigg) 
$\sigma_3\theta_2$
\end{prop}

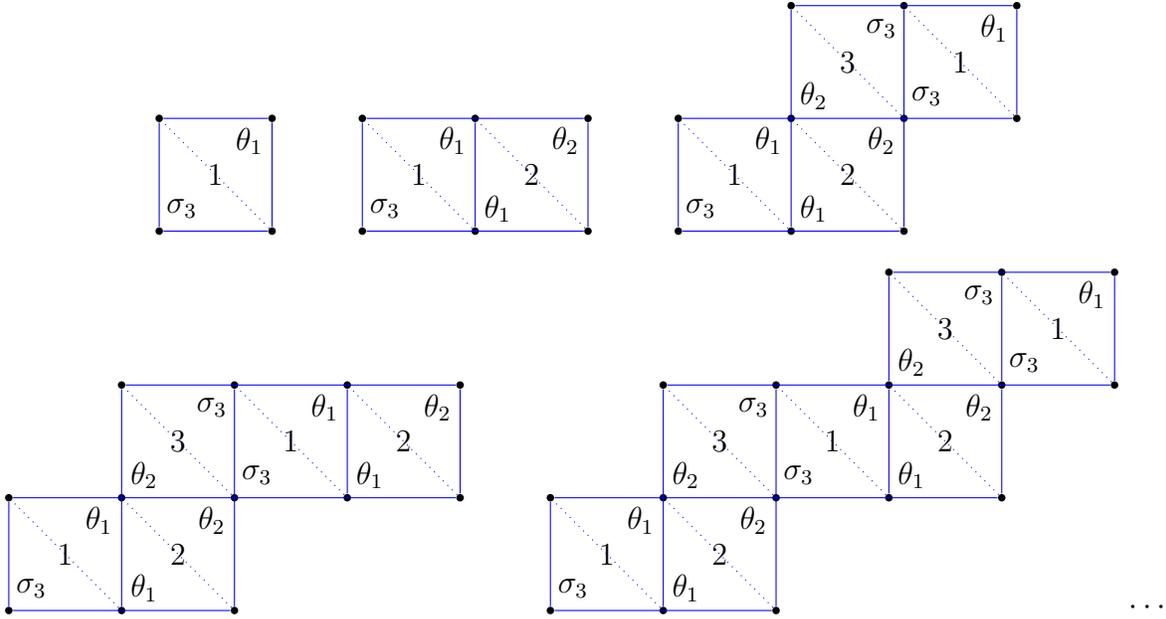
\begin{figure}

\begin{tikzpicture}[scale=0.3,every node/.style={sloped,allow upside down}]
    \node [circle,fill=black,inner sep=1pt] (0) at (0, 0) {};
    \node [circle,fill=black,inner sep=1pt] (1) at (5, 0) {};
    \node [circle,fill=black,inner sep=1pt] (4) at (0, 5) {};    
    \node [circle,fill=black,inner sep=1pt] (5) at (5, 5) {};
     
    \draw [style=blue] (0) to (1);
    \draw [style=blue] (1) to (5);
    \draw [style=blue] (4) to (5);
    \draw [style=blue] (0) to (4);

    \node [] (28) at (2.5, 2.5) {$1$};

    \draw [style=blue,dotted] (1) to (4);
          
    \node [] (28) at (1, 1) {$\sigma_3$};

    \node [] (28) at (4, 4) {$\theta_1$};

\end{tikzpicture}  \hspace{2em}
\begin{tikzpicture}[scale=0.3,every node/.style={sloped,allow upside down}]
    \node [circle,fill=black,inner sep=1pt] (0) at (0, 0) {};
    \node [circle,fill=black,inner sep=1pt] (1) at (5, 0) {};
    \node [circle,fill=black,inner sep=1pt] (2) at (10, 0) {};
    \node [circle,fill=black,inner sep=1pt] (4) at (0, 5) {};    
    \node [circle,fill=black,inner sep=1pt] (5) at (5, 5) {};
    \node [circle,fill=black,inner sep=1pt] (6) at (10, 5) {};
     
    \draw [style=blue] (0) to (2);
    \draw [style=blue] (4) to (6);

    \draw [style=blue] (0) to (4);
    \draw [style=blue] (1) to (5);       
    \draw [style=blue] (2) to (6);   

    \node [] (28) at (2.5, 2.5) {$1$};
    \node [] (29) at (7.5, 2.5) {$2$};

    \draw [style=blue,dotted] (1) to (4);
    \draw [style=blue,dotted] (2) to (5);
          
    \node [] (28) at (1, 1) {$\sigma_3$};
    \node [] (29) at (6, 1) {$\theta_1$};

    \node [] (28) at (4, 4) {$\theta_1$};
    \node [] (29) at (9, 4) {$\theta_2$};

\end{tikzpicture} \hspace{2em}
\begin{tikzpicture}[scale=0.3,every node/.style={sloped,allow upside down}]
    \node [circle,fill=black,inner sep=1pt] (0) at (0, 0) {};
    \node [circle,fill=black,inner sep=1pt] (1) at (5, 0) {};
    \node [circle,fill=black,inner sep=1pt] (2) at (10, 0) {};
    \node [circle,fill=black,inner sep=1pt] (3) at (0, 5) {};    
    \node [circle,fill=black,inner sep=1pt] (4) at (5, 5) {};
    \node [circle,fill=black,inner sep=1pt] (5) at (10, 5) {};
    \node [circle,fill=black,inner sep=1pt] (6) at (15, 5) {};
    \node [circle,fill=black,inner sep=1pt] (7) at (5, 10) {};    
    \node [circle,fill=black,inner sep=1pt] (8) at (10, 10) {};
    \node [circle,fill=black,inner sep=1pt] (9) at (15, 10) {};    
     
    \draw [style=blue] (0) to (2);
    \draw [style=blue] (3) to (6);
    \draw [style=blue] (7) to (9);

    \draw [style=blue] (0) to (3);
    \draw [style=blue] (1) to (7);       
    \draw [style=blue] (2) to (8);   
    \draw [style=blue] (6) to (9);           

    \node [] (28) at (2.5, 2.5) {$1$};
    \node [] (29) at (7.5, 2.5) {$2$};
    \node [] (30) at (7.5, 7.5) {$3$};
    \node [] (31) at (12.5, 7.5) {$1$};

    \draw [style=blue,dotted] (1) to (3);
    \draw [style=blue,dotted] (2) to (4);
    \draw [style=blue,dotted] (5) to (7);
    \draw [style=blue,dotted] (6) to (8);              
          
    \node [] (28) at (1, 1) {$\sigma_3$};
    \node [] (29) at (6, 1) {$\theta_1$};
    \node [] (30) at (6, 6) {$\theta_2$};
    \node [] (31) at (11, 6) {$\sigma_3$};

    \node [] (28) at (4, 4) {$\theta_1$};
    \node [] (29) at (9, 4) {$\theta_2$};
    \node [] (30) at (9, 9) {$\sigma_3$};
    \node [] (31) at (14, 9) {$\theta_1$};

\end{tikzpicture} 

\vspace{1em}

\begin{tikzpicture}[scale=0.3,every node/.style={sloped,allow upside down}]
    \node [circle,fill=black,inner sep=1pt] (0) at (0, 0) {};
    \node [circle,fill=black,inner sep=1pt] (1) at (5, 0) {};
    \node [circle,fill=black,inner sep=1pt] (2) at (10, 0) {};
    \node [circle,fill=black,inner sep=1pt] (3) at (0, 5) {};    
    \node [circle,fill=black,inner sep=1pt] (4) at (5, 5) {};
    \node [circle,fill=black,inner sep=1pt] (5) at (10, 5) {};
    \node [circle,fill=black,inner sep=1pt] (6) at (15, 5) {};
    \node [circle,fill=black,inner sep=1pt] (7) at (5, 10) {};    
    \node [circle,fill=black,inner sep=1pt] (8) at (10, 10) {};
    \node [circle,fill=black,inner sep=1pt] (9) at (15, 10) {};    

    \node [circle,fill=black,inner sep=1pt] (10) at (20, 5) {};
    \node [circle,fill=black,inner sep=1pt] (11) at (20, 10) {};

    \draw [style=blue] (0) to (2);
    \draw [style=blue] (3) to (10);
    \draw [style=blue] (7) to (11);

    \draw [style=blue] (0) to (3);
    \draw [style=blue] (1) to (7);       
    \draw [style=blue] (2) to (8);   
    \draw [style=blue] (6) to (9);
    \draw [style=blue] (10) to (11);           

    \node [] (28) at (2.5, 2.5) {$1$};
    \node [] (29) at (7.5, 2.5) {$2$};
    \node [] (30) at (7.5, 7.5) {$3$};
    \node [] (31) at (12.5, 7.5) {$1$};
    \node [] (32) at (17.5, 7.5) {$2$};

    \draw [style=blue,dotted] (1) to (3);
    \draw [style=blue,dotted] (2) to (4);
    \draw [style=blue,dotted] (5) to (7);
    \draw [style=blue,dotted] (6) to (8);      
    \draw [style=blue,dotted] (10) to (9);        
          
    \node [] (28) at (1, 1) {$\sigma_3$};
    \node [] (29) at (6, 1) {$\theta_1$};
    \node [] (30) at (6, 6) {$\theta_2$};
    \node [] (31) at (11, 6) {$\sigma_3$};

    \node [] (28) at (4, 4) {$\theta_1$};
    \node [] (29) at (9, 4) {$\theta_2$};
    \node [] (30) at (9, 9) {$\sigma_3$};
    \node [] (31) at (14, 9) {$\theta_1$};          

    \node [] (32) at (16, 6) {$\theta_1$};
    \node [] (33) at (19, 9) {$\theta_2$};

\end{tikzpicture} 
 \hspace{2em}
\begin{tikzpicture}[scale=0.3,every node/.style={sloped,allow upside down}]

    \node [circle,fill=black,inner sep=1pt] (0) at (0, 0) {};
    \node [circle,fill=black,inner sep=1pt] (1) at (5, 0) {};
    \node [circle,fill=black,inner sep=1pt] (2) at (10, 0) {};
    \node [circle,fill=black,inner sep=1pt] (3) at (0, 5) {};    
    \node [circle,fill=black,inner sep=1pt] (4) at (5, 5) {};
    \node [circle,fill=black,inner sep=1pt] (5) at (10, 5) {};
    \node [circle,fill=black,inner sep=1pt] (6) at (15, 5) {};
    \node [circle,fill=black,inner sep=1pt] (7) at (5, 10) {};    
    \node [circle,fill=black,inner sep=1pt] (8) at (10, 10) {};
    \node [circle,fill=black,inner sep=1pt] (9) at (15, 10) {};    

    \node [circle,fill=black,inner sep=1pt] (10) at (20, 5) {};
    \node [circle,fill=black,inner sep=1pt] (11) at (20, 10) {};    

    \node [circle,fill=black,inner sep=1pt] (12) at (15, 15) {};
    \node [circle,fill=black,inner sep=1pt] (13) at (20, 15) {};    
    
    \node [circle,fill=black,inner sep=1pt] (14) at (25, 10) {};
    \node [circle,fill=black,inner sep=1pt] (15) at (25, 15) {};

    \draw [style=blue] (0) to (2);
    \draw [style=blue] (3) to (10);
    \draw [style=blue] (7) to (14);
    \draw [style=blue] (12) to (15);

    \draw [style=blue] (0) to (3);
    \draw [style=blue] (1) to (7);       
    \draw [style=blue] (2) to (8);   
    \draw [style=blue] (6) to (9);
    \draw [style=blue] (10) to (11);          
    \draw [style=blue] (9) to (12);
    \draw [style=blue] (11) to (13);
    \draw [style=blue] (14) to (15);     

    \node [] (28) at (2.5, 2.5) {$1$};
    \node [] (29) at (7.5, 2.5) {$2$};
    \node [] (30) at (7.5, 7.5) {$3$};
    \node [] (31) at (12.5, 7.5) {$1$};
    \node [] (32) at (17.5, 7.5) {$2$};
    \node [] (33) at (17.5, 12.5) {$3$};
    \node [] (34) at (22.5, 12.5) {$1$};

    \draw [style=blue,dotted] (1) to (3);
    \draw [style=blue,dotted] (2) to (4);
    \draw [style=blue,dotted] (5) to (7);
    \draw [style=blue,dotted] (6) to (8);      
    \draw [style=blue,dotted] (10) to (9);        
    \draw [style=blue,dotted] (11) to (12);        
    \draw [style=blue,dotted] (14) to (13);

    \node [] (28) at (1, 1) {$\sigma_3$};
    \node [] (29) at (6, 1) {$\theta_1$};
    \node [] (30) at (6, 6) {$\theta_2$};
    \node [] (31) at (11, 6) {$\sigma_3$};

    \node [] (28) at (4, 4) {$\theta_1$};
    \node [] (29) at (9, 4) {$\theta_2$};
    \node [] (30) at (9, 9) {$\sigma_3$};
    \node [] (31) at (14, 9) {$\theta_1$};          

    \node [] (32) at (16, 6) {$\theta_1$};
    \node [] (33) at (19, 9) {$\theta_2$};
    
    \node [] (32) at (16, 11) {$\theta_2$};
    \node [] (33) at (19, 14) {$\sigma_3$};

    \node [] (32) at (21, 11) {$\sigma_3$};
    \node [] (33) at (24, 14) {$\theta_1$};

\end{tikzpicture}  
$\cdots$
\caption{Snake Graphs $G_n(1,2)$ for small $n$.}
\label{fig:Ann21}

\end{figure}

\begin{proof}
We observe that for each arc $n \geq 3$, if we lift to the universal cover and then restrict to the subtriangulation of triangles crossed by arc $n$, the orientation of Figure \ref{fig:annulus} is a default orientation.  However, if we then apply the procedure of Definition \ref{def:pos_ordering} to construct a positive ordering of the associated $\mu$-invariants, we obtain the inconsistent ordering
$$ \cdots \sigma_3 > \sigma_3 > \cdots \sigma_3 > \theta_1 > \theta_2 > \theta_1 > \theta_2 > \cdots > \theta_1$$ 

$$\mathrm{~~or~~}  \cdots \sigma_3 > \sigma_3 > \cdots \sigma_3 > \theta_2 > \theta_1 > \theta_2 > \cdots > \theta_1$$ 
depending on the parity of $n$.  See Figure \ref{fig:Ann21b} for examples of how this affects the ordering of $\mu$-invariants as they arise on tiles of the associated snake graphs, depending on whether $n$ is odd or even.  In particular, comparing with Figure \ref{fig:Ann21}, the $\mu$-invariants coming first in the positive ordering are labeled with higher numbers (in parentheses) and those coming later are labeled with smaller numbers.

Any double dimer cover including a cycle whose leftmost triangle is $\sigma_3$ in Tile $1$ therefore has a rightmost triangle that comes later in this positive ordering.  (Further, if the last tile of this cycle is Tile $1$, we get a contribution of $\sigma_3\sigma_3 = 0$. We thus exclude such cycles of length divisible by $3$.) We thus get contributions of $+\sigma_3 \theta_1$ for cycles stretching from Tile $1$ to Tile $1$ and $+\sigma_3 \theta_2$ for cycles stretching from Tile $1$ to Tile $2$. 

On the other hand, a cycle whose leftmost triangle is $\theta_1$ in Tile $2$ will end with a rightmost triangle that comes earlier in this positive ordering.  We also again exclude such cycles of length divisible by $3$ to avoid contributions of $\theta_1\theta_1=0$.  Consequently, a cycle stretching from Tile $2$ to Tile $2$ contributes $-\theta_1\theta_2 = + \theta_2\theta_1$ while a cycle stretching from Tile $2$ to Tile $3$ contributes $-\theta_1\sigma_3 = + \sigma_3\theta_1$.  

Analogously, a cycle whose leftmost triangle is $\theta_2$ in Tile $3$ will end with a rightmost triangle that comes earlier in this positive ordering.  As above, we exclude such cycles of length divisible by $3$ to avoid contributions of $\theta_2\theta_2=0$.  Thus, a cycle stretching from Tile $3$ to Tile $1$ contributes $-\theta_2\theta_1$ while a cycle stretching from Tile $3$ to Tile $3$ contributes $-\theta_2\sigma_3 = + \sigma_3\theta_2$. 

Notice that in particular, we get positive coefficients on $\sigma_3\theta_1$, $\sigma_3\theta_2$, and $\theta_2\theta_1$ (with the products in this order) except for the one case of a cycle from Tile $3$ to Tile $1$, as desired.

Lastly, any double dimer cover containing more than one cycle will contains a degree four product of odd invariants.  But since we only have three odd invariants possible in this example, two of these must coincide and thus such terms vanish.
\end{proof}

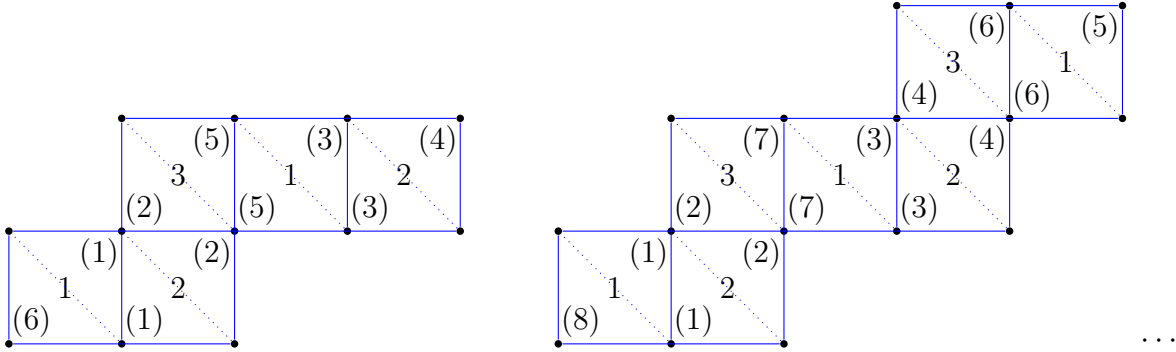
\begin{figure}
\begin{tikzpicture}[scale=0.3,every node/.style={sloped,allow upside down}]
    \node [circle,fill=black,inner sep=1pt] (0) at (0, 0) {};
    \node [circle,fill=black,inner sep=1pt] (1) at (5, 0) {};
    \node [circle,fill=black,inner sep=1pt] (2) at (10, 0) {};
    \node [circle,fill=black,inner sep=1pt] (3) at (0, 5) {};    
    \node [circle,fill=black,inner sep=1pt] (4) at (5, 5) {};
    \node [circle,fill=black,inner sep=1pt] (5) at (10, 5) {};
    \node [circle,fill=black,inner sep=1pt] (6) at (15, 5) {};
    \node [circle,fill=black,inner sep=1pt] (7) at (5, 10) {};    
    \node [circle,fill=black,inner sep=1pt] (8) at (10, 10) {};
    \node [circle,fill=black,inner sep=1pt] (9) at (15, 10) {};    

    \node [circle,fill=black,inner sep=1pt] (10) at (20, 5) {};
    \node [circle,fill=black,inner sep=1pt] (11) at (20, 10) {};

    \draw [style=blue] (0) to (2);
    \draw [style=blue] (3) to (10);
    \draw [style=blue] (7) to (11);

    \draw [style=blue] (0) to (3);
    \draw [style=blue] (1) to (7);       
    \draw [style=blue] (2) to (8);   
    \draw [style=blue] (6) to (9);
    \draw [style=blue] (10) to (11);           

    \node [] (28) at (2.5, 2.5) {$1$};
    \node [] (29) at (7.5, 2.5) {$2$};
    \node [] (30) at (7.5, 7.5) {$3$};
    \node [] (31) at (12.5, 7.5) {$1$};
    \node [] (32) at (17.5, 7.5) {$2$};

    \draw [style=blue,dotted] (1) to (3);
    \draw [style=blue,dotted] (2) to (4);
    \draw [style=blue,dotted] (5) to (7);
    \draw [style=blue,dotted] (6) to (8);      
    \draw [style=blue,dotted] (10) to (9);        
          
    \node [] (28) at (1, 1) {$(6)$};
    \node [] (29) at (6, 1) {$(1)$};
    \node [] (30) at (6, 6) {$(2)$};
    \node [] (31) at (11, 6) {$(5)$};

    \node [] (28) at (4, 4) {$(1)$};
    \node [] (29) at (9, 4) {$(2)$};
    \node [] (30) at (9, 9) {$(5)$};
    \node [] (31) at (14, 9) {$(3)$};          

    \node [] (32) at (16, 6) {$(3)$};
    \node [] (33) at (19, 9) {$(4)$};

\end{tikzpicture} 
 \hspace{2em}
\begin{tikzpicture}[scale=0.3,every node/.style={sloped,allow upside down}]

    \node [circle,fill=black,inner sep=1pt] (0) at (0, 0) {};
    \node [circle,fill=black,inner sep=1pt] (1) at (5, 0) {};
    \node [circle,fill=black,inner sep=1pt] (2) at (10, 0) {};
    \node [circle,fill=black,inner sep=1pt] (3) at (0, 5) {};    
    \node [circle,fill=black,inner sep=1pt] (4) at (5, 5) {};
    \node [circle,fill=black,inner sep=1pt] (5) at (10, 5) {};
    \node [circle,fill=black,inner sep=1pt] (6) at (15, 5) {};
    \node [circle,fill=black,inner sep=1pt] (7) at (5, 10) {};    
    \node [circle,fill=black,inner sep=1pt] (8) at (10, 10) {};
    \node [circle,fill=black,inner sep=1pt] (9) at (15, 10) {};    

    \node [circle,fill=black,inner sep=1pt] (10) at (20, 5) {};
    \node [circle,fill=black,inner sep=1pt] (11) at (20, 10) {};    

    \node [circle,fill=black,inner sep=1pt] (12) at (15, 15) {};
    \node [circle,fill=black,inner sep=1pt] (13) at (20, 15) {};    
    
    \node [circle,fill=black,inner sep=1pt] (14) at (25, 10) {};
    \node [circle,fill=black,inner sep=1pt] (15) at (25, 15) {};

    \draw [style=blue] (0) to (2);
    \draw [style=blue] (3) to (10);
    \draw [style=blue] (7) to (14);
    \draw [style=blue] (12) to (15);

    \draw [style=blue] (0) to (3);
    \draw [style=blue] (1) to (7);       
    \draw [style=blue] (2) to (8);   
    \draw [style=blue] (6) to (9);
    \draw [style=blue] (10) to (11);          
    \draw [style=blue] (9) to (12);
    \draw [style=blue] (11) to (13);
    \draw [style=blue] (14) to (15);     

    \node [] (28) at (2.5, 2.5) {$1$};
    \node [] (29) at (7.5, 2.5) {$2$};
    \node [] (30) at (7.5, 7.5) {$3$};
    \node [] (31) at (12.5, 7.5) {$1$};
    \node [] (32) at (17.5, 7.5) {$2$};
    \node [] (33) at (17.5, 12.5) {$3$};
    \node [] (34) at (22.5, 12.5) {$1$};

    \draw [style=blue,dotted] (1) to (3);
    \draw [style=blue,dotted] (2) to (4);
    \draw [style=blue,dotted] (5) to (7);
    \draw [style=blue,dotted] (6) to (8);      
    \draw [style=blue,dotted] (10) to (9);        
    \draw [style=blue,dotted] (11) to (12);        
    \draw [style=blue,dotted] (14) to (13);

    \node [] (28) at (1, 1) {$(8)$};
    \node [] (29) at (6, 1) {$(1)$};
    \node [] (30) at (6, 6) {$(2)$};
    \node [] (31) at (11, 6) {$(7)$};

    \node [] (28) at (4, 4) {$(1)$};
    \node [] (29) at (9, 4) {$(2)$};
    \node [] (30) at (9, 9) {$(7)$};
    \node [] (31) at (14, 9) {$(3)$};          

    \node [] (32) at (16, 6) {$(3)$};
    \node [] (33) at (19, 9) {$(4)$};
    
    \node [] (32) at (16, 11) {$(4)$};
    \node [] (33) at (19, 14) {$(6)$};

    \node [] (32) at (21, 11) {$(6)$};
    \node [] (33) at (24, 14) {$(5)$};
    
\end{tikzpicture}  
$\cdots$
\caption{Positive Ordering on Snake Graphs $G_7(1,2)$ and $G_8(1,2)$ where $(i) > (j)$ signifies $i$ comes before $j$, i.e. $\mu_i\mu_j$ is the order with the positive sign.  Compare with Figure \ref{fig:Ann21}.}
\label{fig:Ann21b}

\end{figure}

\begin{remark}
If we consider combinatorial interpretations for super $\lambda$-lengths $x_n$ for cases when $q\geq 3$, we would have four or more $\mu$-invariants.  While most of the proof could be adapted in such cases, the last line would not apply, and we thus would need to include contributions from double dimer covers with two or more cycles.
\end{remark}

\begin{remark}
We note that the number of odd elements ($\mu$-invariants) appearing in the coordinate ring for decorated super Teichm\"uller space is determined by the number of triangles in such a marked surface and whenever such a coordinate ring admits at least four distinct $\mu$-invariants, we may have terms of degree $\geq 4$ expressed in terms of $\mu$-invariants.  Otherwise, if only $2$ or $3$ $\mu$-invariants are admitted, we only have terms in the body (no $\mu$-invariants) or terms in the soul that are degree $2$ in terms of $\mu$-invariants.  The main examples mentioned or studied in this paper previously: Super Fibonacci Numbers (corresponding to annuli with two marked points) and Super Markov Numbers (corresponding to the once-punctured torus) both admit only $2$ $\mu$-invariants.  

  \vspace{0.5em}

Thanks to the work of \cite{pz_19}, Penner and Zeitlin imply that the super-dimension of decorated super Teichm\"uller space is  $(6g - 6 + 3s + c \bigg | 4g - 4 + 2s + c)$ where $g$ is the genus of the surface, $s$ is the number of 
 boundary components and punctures (combined) and $c$ is the number of marked points on boundary components.
 
  \vspace{0.5em}
 
 Thus the only surfaces with an odd dimension of $2$, i.e. $\sigma \theta = \epsilon$, are the thrice-punctured sphere, quadrilateral, once-punctured digon, annulus with one marked point on each boundary, and the once-punctured torus.  
 
  \vspace{0.5em}
 
 The only surfaces with an odd dimension of $3$, and thus the last case of $\sigma \theta = \epsilon_1$, $\sigma \nu = \epsilon_2$, and $\theta \nu = \epsilon_3$ such that $\epsilon_i\epsilon_j = 0$ for any $i$ and $j$ are the pentagon, once-punctured triangle, annulus with one marked point on one boundary two on the boundary, or the torus with one boundary and one marked point.

  \vspace{0.5em}

Consequently, between the author's previous work with N. Ovenhouse and S. Zhang focusing on all polygons (including the quadrilateral and pentagon) as well as the above work of this current paper on the once-punctured torus and annuli with two or three marked points in total, exhausts almost all such examples of decorated super Teichm\"uller space with odd dimension of two or three.
\end{remark}

\begin{remark}
The remaining cases include the three-punctured sphere, which we exclude for technical reasons\footnote{If one draws the quiver associated to triangulation of the three-punctured sphere, one gets a three-vertex quiver with $2$-cycles between each pair of vertices or equivalently an empty quiver.}, as well as once-punctured polygons, which the author plans to study in future work.  
This last remaining case mentioned here of an unpunctured torus with one boundary and one marked point admits only a single quiver in its mutation class as well, so would be another interesting case for further study as explored in the ordinary case in \cite{lampe2016diophantine}.  This triangulation corresponds to solutions to the $5$-variate Diophantine equation 
 $$ v w z^2 + v^2 y z + x y^2 z + w^2 x z + v w y^2 + w^2 y z + v w x^2 + v^2 x z + x^2 y z = 9  v w x y.$$  
  \end{remark}

\subsection{Other Diophantine equations or Integrable Systems}

Note that the super analogue of the Markov equation and super analogues of integrable systems constructed from Dehn twists of annuli are only two of many possible $\mathbb{Z_2}$-graded analogues of Diophantine equations or of discrete integrable systems.  For example, Y. Gyoda and K. Matsushita \cite{gyoda2022generalization} provide a number of other commutative generalizations of this equation.  P. Lampe also provides additional Diophantine equations that are related to ordinary cluster algebras \cite{lampe2016diophantine}.  These include equations of the form
$$x^2 + y^2 + z^2 + k_1 xy + k_2 yz + k_3 xz = (3+k_1 + k_2 + k_3)xyz$$
for choices of nonnegative integers $k_1,k_2$, and $k_3$, a one-parameter family of Diophantine equations 
$$x^2 + y^4 + z^4 + ky^2z^2 + 2xy^2 + 2xz^2 = (7+k)xy^2z^2$$ defined for any nonnegative integer $k$, as well as the aforementioned 
$5$-variate equation coming from the surface model of the torus cutting out a disk with a single marked point.   L. Bao and F. Li \cite{bao2024study} provide a more thorough treatment of Diophantine equations related to cluster theory and answering some of P. Lampe's conjectures.

\begin{question}
What are super analogues of the Diophantine equations arising in \cite{gyoda2022generalization, lampe2016diophantine, bao2024study}, and how can one construct algebraic formulas using the combinatorics of double dimer covers or otherwise, for their solutions?
\end{question}
 
\subsection{Super Peripheral Arcs and Super Infinite Friezes} 

As another possible future direction, we note that if instead of mutating along the periodic sequence inducing Dehn twists, we can mutate leading to peripheral arcs instead.  Unlike the bridging arcs, there are only ever a finite number of (non-self-intersecting) peripheral arcs for a given annulus.  With one marked point on the inner boundary and $q$ on the outer boundary, there are $q^2 - q$ such peripheral arcs on the outer boundary.  However, by including self-intersecting peripheral arcs as well, one gets an infinite frieze in the spirit of 
\cite{baur2016infinite} and \cite{gunawan2019cluster}.

For the special case of the annulus with one marked point on each boundary, one gets an infinite Super Frieze that includes Super Fibonacci numbers as well as odd variables (multiplied by appropriate square-roots) corresponding to the triangles of the annulus as more and more winding or shearing occurs.  See Figure \ref{fig:InfiniteSF}.  The author thanks Sophie Morier-Genoud for the question inspiring this example.

\begin{figure}
\begin{center}
$\begin{array}{ccccccccccc}
	& &Z_2& & & &Z_4& & & &\\	
	&\textcolor{red}{\theta \sqrt{Z_1Z_2}}& &\textcolor{red}{\sigma' \sqrt{Z_2Z_3}}& &\textcolor{red}{\theta'' \sqrt{Z_3Z_4}}& &\textcolor{red}{\sigma''' \sqrt{Z_4Z_5}}& &\textcolor{red}{\theta^{(iv)} \sqrt{Z_5Z_6}}&\\
	Z_1 & & & &Z_3& & & &Z_5& &\\
	&\textcolor{red}{\sigma\sqrt{Z_1Z_2}}& &\textcolor{red}{\theta'\sqrt{Z_2Z_3}}& &\textcolor{red}{\sigma'' \sqrt{Z_3Z_4}}& &\textcolor{red}{\theta''' \sqrt{Z_4Z_5}}& &\textcolor{red}{\sigma^{(iv)} \sqrt{Z_5Z_6}}&\\
	& &Z_2& & & &Z_4& & & &\\
	&\textcolor{red}{\theta \sqrt{Z_1Z_2}}& &\textcolor{red}{\sigma' \sqrt{Z_2Z_3}}& &\textcolor{red}{\theta'' \sqrt{Z_3Z_4}}& &\textcolor{red}{\sigma''' \sqrt{Z_4Z_5}}& &\textcolor{red}{\theta^{(iv)} \sqrt{Z_5Z_6}}&\\
	Z_1 & & & &Z_3& & & &Z_5& &\\
	&\textcolor{red}{\sigma\sqrt{Z_1Z_2}}& &\textcolor{red}{\theta'\sqrt{Z_2Z_3}}& &\textcolor{red}{\sigma'' \sqrt{Z_3Z_4}}& &\textcolor{red}{\theta''' \sqrt{Z_4Z_5}}& &\textcolor{red}{\sigma^{(iv)} \sqrt{Z_5Z_6}}&\\
	& &Z_2& & & &Z_4& & & &\\
	&\textcolor{red}{\theta \sqrt{Z_1Z_2}}& &\textcolor{red}{\sigma' \sqrt{Z_2Z_3}}& &\textcolor{red}{\theta'' \sqrt{Z_3Z_4}}& &\textcolor{red}{\sigma''' \sqrt{Z_4Z_5}}& &\textcolor{red}{\theta^{(iv)} \sqrt{Z_5Z_6}}&\\
	Z_1 & & & &Z_3& & & &Z_5& &\\
\end{array}$
\end{center}
\caption{An Infinite Super Frieze consisting of Super Fibonacci Numbers $\{Z_n\}$.  See \cite[Sec 11]{moz22} for notation.}
\label{fig:InfiniteSF}
\end{figure}
 
 \begin{question}
 For annuli with $1$ marked point on one boundary and $q$ marked points on the other, can we derive interesting infinite super frieze patterns based on the peripheral arcs for such triangulations?
 \end{question}
 
 \subsection{Additional Open Questions for General Annuli}
 
\begin{question} 
Can we obtain combinatorial formulas for super $\lambda$-lengths or the $\mu$-invariants associated to triangles in annuli with $p$ marked points on one boundary and $q$ on the other?  Or for other initial triangulations even when $p$ still equals $1$?
\end{question}

\begin{question}
Can we obtain combinatorial formulas for super $\lambda$-lengths of closed curves on an annulus, as well as a $\mathbb{Z}_2$-graded analogue of the bangles or bracelets bases?  Do these manifest as conserved quantities of the super-integrable system defined by equations (\ref{eq:x}) - (\ref{eq:thet}) or as double dimer covers of band graphs (following the work of \cite{mw13}) associated to such closed curves?
\end{question}

\bibliographystyle {alpha}
\bibliography {main}

\end{document}